\documentclass[amssymb,12pt]{amsart}
\usepackage{latexsym}
\newtheorem{propo}{Proposition}[section]

\newtheorem{defi}[propo]{Definition}

\newtheorem{lemma}[propo]{Lemma}
\newtheorem{corol}[propo]{Corollary}

\newtheorem{theor}[propo]{Theorem}
\newtheorem{theorem}[propo]{Theorem}
\newtheorem{examp}[propo]{Example}
\newtheorem{remar}[propo]{Remark}

\newcommand{\Ker}{\operatorname{Ker}}
\newcommand{\Aut}{{\mathrm {Aut}}}
\newcommand{\Out}{{\mathrm {Out}}}

\newcommand{\Id}{{\mathrm {Id}}}
\newcommand{\Irr}{{\mathrm {Irr}}}

\newcommand{\Ind}{{\mathrm {Ind}}}
\newcommand{\diag}{{\mathrm {diag}}}

\newcommand{\Tr}{{\mathrm {Tr}}}

\newcommand{\Spec}{{\mathrm {Spec}}}

\newcommand{\Sym}{{\mathrm {Sym}}}

\newcommand{\CC}{{\mathbb C}}
\newcommand{\RR}{{\mathbb R}}
\newcommand{\QQ}{{\mathbb Q}}
\newcommand{\ZZ}{{\mathbb Z}}

\newcommand{\SSS}{{\sf S}}

\newcommand{\AAA}{{\sf A}}
\newcommand{\FF}{{\mathbb F}}

\newcommand{\GC}{\mathcal{G}}

\newcommand{\CL}{\mathcal{C}}
\newcommand{\HC}{\mathcal{H}}
\newcommand{\IC}{\mathcal{I}}

\newcommand{\TC}{\mathcal{T}}

\newcommand{\BC}{\mathcal{B}}
\newcommand{\NC}{\mathcal{N}}
\newcommand{\XC}{\mathcal{X}}
\newcommand{\YC}{\mathcal{Y}}

\newcommand{\eps}{\epsilon}
\newcommand{\lam}{\lambda}

\newcommand{\al}{\alpha}
\newcommand{\gam}{\gamma}

\newcommand{\la}{\langle}
\newcommand{\ra}{\rangle}

\newcommand{\DA}{d_{1}}
\newcommand{\DB}{d_{2}}
\newcommand{\DD}{d_{j}}
\newcommand{\SA}{S^{1}}
\newcommand{\age}{{\sf {age}}}
\newcommand{\ages}{{\sf {age}}^{*}}
\newcommand{\an}{{\boldsymbol{\alpha}}}
\newcommand{\bn}{{\boldsymbol{\beta}}}

\newcommand{\bg}{\bar{g}}
\newcommand{\sta}{(\star)}
\newcommand{\start}{(\star\star)}
\newcommand{\stc}{(\clubsuit)}

\newcommand{\diam}{(\spadesuit)}

\newcommand{\dl}{{\mathfrak d}}
\newcommand{\fl}{{\mathfrak f}}

\newcommand{\tbf}{{\bf t}}
\newcommand{\tn}{\hspace{0.5mm}^{t}\hspace*{-0.2mm}}
\newcommand{\ta}{\hspace{0.5mm}^{2}\hspace*{-0.2mm}}
\newcommand{\tb}{\hspace{0.5mm}^{3}\hspace*{-0.2mm}}
\def\skipa{\vspace{-1.5mm} & \vspace{-1.5mm} & \vspace{-1.5mm}\\}
\renewcommand{\mod}{\bmod \,}

\begin{document}
\title[A problem of Koll\'ar and Larsen on finite linear groups]
{A problem of Koll\'ar and Larsen on finite linear groups and  
crepant resolutions}
\author{Robert M. Guralnick}
\address{Department of Mathematics, University of Southern California,
Los Angeles, CA 90089-1113, USA}
\email{guralnic@math.usc.edu}
\author{Pham Huu Tiep}
\address{Department of Mathematics, University of Arizona, Tucson, AZ 85721-0089, USA}
\email{tiep@math.arizona.edu}
\date{May 20, 2009}



\thanks{Part of this paper was written while the authors were participating
in the Representations of Finite Groups Program of the Mathematical Sciences Research
Institute (Berkeley, 2008), and the Algebraic Lie Theory Program of the Isaac Newton 
Institute for Mathematical Sciences (Cambridge, 2009). 
It is a pleasure to thank the MSRI and
the Newton Institute for their generous hospitality and support.} 

\thanks{The authors would like to thank J\'anos Koll\'ar for suggesting this problem to them, and for 
many insightful comments on the paper. They also thank Thomas Breuer for computing the character tables
of certain groups of extraspecial type for them, and Jason Fulman, Lennie Friedlander, Nick Katz, 
Klaus Lux, and Terence Tao for helpful remarks on various aspects of the paper.}

\thanks{The authors gratefully acknowledge the support of the NSF (grants
DMS-0653873 and DMS-0600967).}

\begin{abstract}
{The notion of $\age$ of elements of complex linear groups was introduced
by M. Reid and is of importance in algebraic geometry, in particular in
the study of crepant resolutions and of quotients of Calabi-Yau varieties.
In this paper, we solve a problem raised by J. Koll\'ar and M. Larsen on 
the structure of finite irreducible linear groups generated by elements of
$\age \leq 1$. More generally, we bound the dimension of finite irreducible 
linear groups generated by elements of bounded deviation. As a consequence of
our main results, we derive some properties of symmetric spaces
$GU_{d}(\CC)/G$ having shortest closed geodesics of bounded length, and 
of quotients $\CC^{d}/G$ having a crepant resolution.}
\end{abstract}

\maketitle

\section{Introduction}
Let $V = \CC^{d}$ be a $d$-dimensional complex space and let $G < GL(V)$ be a finite subgroup. 
A classical theme in group theory and representation theory, going back at least to work
of H. Blichfeldt on primitive linear groups, and work of G. C. Shephard and J. A. Todd \cite{ST} on 
complex reflection groups, is to characterize $G$ under various conditions that force
$G$ to contain non-identity elements which are ``close'' to the identity transformation on $V$. 
Recall that a {\it complex reflection group} (c.r.g. for short) is a subgroup of $GU(V)$ that is 
generated by a set of complex (pseudo)reflections. The
complex reflection groups can be arguably said to be one of the most ubiquitous objects in 
modern mathematics.   

Recently, motivated by potential applications in algebraic geometry, string theory, mirror symmetry, and
quantum cohomology, J. Koll\'ar and M. Larsen \cite{KL} have raised the problem of studying linear
groups containing elements of bounded (or small) {\it deviation}, where the deviation is defined 
in a certain way to measure the ``closeness'' of group elements to the identity transformation.  
It turns out to be most convenient to work with the following $L^{2}$-variant of the
Koll\'ar-Larsen deviation: $\DB(g)^{2} = 2(\dim(V)-|\Tr(g)|)$ for $g \in GL(V)$ 
(see \S2.2, in particular, Corollary \ref{basic2m} and Proposition \ref{arc}(iii), for various notions 
of deviation and their relationships). 
Henceforth we say that a subgroup $G \leq GL(V)$ 
{\it has property ${\mathcal P}$ up to scalars}, if there is a subgroup $H \leq GL(V)$ with
property ${\mathcal P}$ such that $Z(GL(V))G = Z(GL(V))H$.      

The first main result of the paper is the following theorem which bounds the dimension of
the representation in terms of the deviations of generators. 

\begin{theor}\label{main1}
{\sl Let $G < GL(V)$ be a finite irreducible subgroup. Assume that there is a constant
$C \geq 4$ such that, up to scalars, $G$ is generated by some elements $g_{i}$ with 
$\DB(g_{i})^{2} \leq C$, $1 \leq i \leq s$. Then one of the following holds.

{\rm (i)} $\dim(V) \leq \fl(C) := \max\{4C^{2}/63, 40C\}$.

{\rm (ii)} $Z(G) \times \AAA_{n} \leq G \leq (Z(G) \times \AAA_{n}) \cdot 2$ and $\dim(V) = n-1$, 
with $\AAA_{n}$ acting on $V$ as on its deleted natural permutation module.

{\rm (iii)} $G$ preserves a decomposition $V = V_{1} \oplus \ldots \oplus V_{m}$, with
$\dim(V_{i}) \leq C/4$ and $G$ inducing either $\SSS_{m}$ or $\AAA_{m}$ while permuting 
the $m$ subspaces $V_{1}, \ldots ,V_{m}$.}
\end{theor} 

One certainly expects the upper bound $\dim(V) \leq \fl(C)$ in Theorem \ref{main1}(i) to have 
rather a theoretical than practical value. However, we notice that for $C$ big enough (say $C \geq 630$), 
this bound is already quite close to be optimal, cf. Example \ref{optm1}. In general, as pointed out 
to the authors by Koll\'ar, Theorem \ref{main1} should have interesting implications for 
differential geometry on symmetric spaces. Consider for instance locally symmetric spaces
that behave locally like $GU_{n}(\CC)$: they are of the form $GU_{n}(\CC)/G$ for a finite subgroup 
$G < GU_{n}(\CC)$. Then the shortest closed geodesics in $GU_{n}(\CC)/G$ have length  
$2\pi \cdot \min_{1 \neq g \in G}||g||$, where $||g||$ is as defined in Definition \ref{metric}. 
Here is one consequence of Theorem \ref{main1} in this context.

\begin{corol}\label{main1g}
{\sl Let $G < GU(V)$ be a finite irreducible, primitive, tensor indecomposable subgroup. Assume that
the shortest closed geodesics in $GU(V)/G$ have length $\leq L$. Then either one of the conclusions 
{\rm (i), (ii)} of Theorem {\rm \ref{main1}} holds for $G$ with $C := \max\{4, L^{2}\}$, or 
$\dim(V) \leq (L \cdot |Z(G)|/2\pi)^{2}$.}
\end{corol}
    
The next result shows that noncentral elements $g$ of finite irreducible subgroups of $GL(V)$ usually
have deviation $\DB(g)^{2} \geq 4$, which implies that the condition $C \geq 4$ in Theorem 
\ref{main1} is natural.

\begin{theor}\label{main-b}
{\sl Let $G < GL(V)$ be a finite primitive, irreducible subgroup. Let $d := \dim(V) \geq 2$,  
$g \in G \setminus Z(G)$, and set $\Delta(g) := \dim(V)-|\Tr(g)|$. If $G$ is tensor induced, 
assume furthermore that $g$ acts nontrivially on the set of tensor factors of $V$.
Then one of the following statements hold.

{\rm (i)} $d = 2$ and $\Delta(g) \geq (3-\sqrt{5})/2$.

{\rm (ii)} $d = 3$ and $\Delta(g) \geq 3-\sqrt{3}$.

{\rm (iii)} $d = 4$ and $\Delta(g) \geq 4-2\sqrt{2}$.

{\rm (iv)} $d \geq 5$ and either $\Delta(g) \geq 8-4\sqrt{2}$, or $\Delta(g) = 2$ and $g$ is 
a scalar multiple of a reflection.

{\rm (v)} $V = A \otimes B$ is tensor decomposable as a $G$-module, $\dim(A) = 2$, 
$2 \leq \dim(B) \leq 6$, $g|_{B}$ is scalar, and $\Delta(g) \geq \dim(B) \cdot (3-\sqrt{5})/2$.}
\end{theor}

The notion of 
{\it age} of elements of complex linear groups, see Definition \ref{reid}, originates from the work of
M. Reid \cite{R1}, \cite{R2}, \cite{IR}. Its importance in algebraic geometry comes from the 
{\it Reid-Tai criterion} \cite{R1}: {\sl If the subgroup $G < GL_{d}(\CC)$ contains no complex reflections,
then $\CC^{d}/G$ is terminal, resp. canonical, if and only if $\age(g) > 1$, 
resp. $\age(g) \geq 1$ for every $1 \neq g \in G$} (see e.g. \cite{CK} for the definition of 
terminal and canonical singularities). This implies in particular the following   
result of \cite{IR}: 
{\sl If $G$ is a finite subgroup of $GL_{d}(\CC)$ and $f~:X \to \CC^{d}/G$ is a crepant resolution, 
then $G$ contains elements $g$ with $\age(g) \leq 1$}. Recall that, a resolution $f~:~X \to Y$ 
is said to be {\it crepant}, if $f^{*}K_{Y} = K_{X}$. Furthermore, in the profound programme of 
S. Mori to classify $3$-dimensional algebraic varieties and in mirror symmetry, the
singularities of type $\CC^{d}/G$ for some finite subgroup $G < GL_{d}(\CC)$ form a very good test 
class where many features of the general case can be tested in a computable setting. 
Recently, there has been a tremendous amount of research devoted to crepant resolutions. 
For instance, minimal models in Mori's programme   
utilize crepant maps. Crepant resolutions
of quotients $X/G$ of Calabi-Yau varieties $X$ are also used in works on mirror symmetry
(particularly as a way of obtaining {\it mirrors}). 
Physicists have long believed that string theories 
on a quotient space and on its crepant resolutions should be equivalent. Recent conjectures of Y. Ruan 
\cite{Ru}, and J. Bryan and T. Graber \cite{BG} state that if $f~:~X \to Y$ is a crepant resolution, then 
quantum cohomology of $X$ and of $Y$ are essentially the same. More recently, Koll\'ar and 
Larsen \cite{KL} studied quotients $X/G$ of a smooth projective Calabi-Yau variety $X$ by a 
finite group $G$ and showed in particular that the Kodaira dimension of $X/G$ is controlled by 
whether $Stab_{x}(G)$ contains non-trivial elements of age $< 1$ while acting on the tangent space
$T_{x}X$ for some $x \in X$. 
  
The next two theorems of the paper classify finite irreducible subgroups of $GL(V)$ that are 
generated by {\it junior} elements, that is, elements $g$ with $0 < \age(g) \leq 1$, when 
$\dim(V) > 8$.
 
\begin{theor}\label{main2}
{\sl Let $V = \CC^{d}$ with $d \geq 11$ and let $G < GL(V)$ be a finite irreducible subgroup. 
Assume that, up to scalars, $G$ is generated by its elements with age $\leq 1$.  
Then $G$ contains a complex bireflection of order $2$ or $3$, and one of the following statements 
holds. 

{\rm (i)} $Z(G) \times \AAA_{d+1} \leq G \leq (Z(G) \times \AAA_{d+1}) \cdot 2$, with $\AAA_{d+1}$ acting 
on $V$ as on its deleted natural permutation module.

{\rm (ii)} $G$ preserves a decomposition $V = V_{1} \oplus \ldots \oplus V_{d}$, with
$\dim(V_{i}) = 1$ and $G$ inducing either $\SSS_{d}$ or $\AAA_{d}$ while permuting 
the $d$ subspaces $V_{1}, \ldots ,V_{d}$.

{\rm (iii)} $2|d$, and $G = D:\SSS_{d/2} < GL_{2}(\CC) \wr \SSS_{d/2}$, a split extension of 
$D < GL_{2}(\CC)^{d/2}$ by $\SSS_{d/2}$. Furthermore, if $g \in G \setminus D$ has 
$\age(g) \leq 1$, then $g$ is a bireflection (and $\age(g) = 1$).}
\end{theor}

\begin{theor}\label{main3}
{\sl Let $V = \CC^{d}$ with $d \geq 9$ and let $G < GL(V)$ be a finite irreducible subgroup. 
Assume that, up to scalars, $G$ is generated by its elements with age $\leq 1$, and that
$G$ contains a scalar multiple of a non-central element $g$ with $\age(g) < 1$. 
Then one of the following statements holds. 

{\rm (i)} One of the conclusions {\rm (i), (ii)} of Theorem {\rm \ref{main2}} holds, and 
$G$ contains a scalar multiple of a complex reflection.

{\rm (ii)} The conclusion {\rm (iii)} of Theorem {\rm \ref{main2}} holds, and, modulo scalars, 
$G$ cannot be generated by its elements of $\age < 1$.}
\end{theor}

The bound $d \geq 9$ in Theorem \ref{main3} is best possible, cf. Remark \ref{r14}.
In the case $4 \leq \dim(V) \leq 10$ of Theorem \ref{main2}, the structure  
of the arising subgroups $G$ is described in Proposition \ref{main-s}. On the contrary,
from the group-theoretic viewpoint
there is not much to say about the dimensions $\leq 3$: if $1 \neq g \in SL_{d}(\CC)$
has finite order, then $\age(g) = 1$ if $d = 2$, and either $\age(g)$ or $\age(g^{-1}) = 1$ if 
$d = 3$.

A key ingredient in the proofs of Theorems \ref{main2} and \ref{main3} comes 
from Proposition \ref{arc} and its consequence Corollary \ref{bound2}, which
relate the $\age(g)$ to the $L^{2}$-deviation
$\DB(g)^{2}$ and thus allow us to invoke available results 
on character ratios for finite quasi-simple groups \cite{G}, \cite{GM}.    
Also, see Theorem \ref{min-age} for 
a lower bound on the age of any non-central element in finite linear groups. One
should compare the latter result
with the classical theorem of Blichfeldt stating that 
the shortest arc of $\SA$ which contains all eigenvalues of a non-central element in 
a finite primitive complex linear group has length at least $\pi/3$.  

In the case the finite subgroup $G < GL(V)$ fixes a non-degenerate symplectic
form on $V$, D. Kaledin \cite{Ka} and M. Verbitsky \cite{V} have shown that 
$V/G$ can have a crepant resolution only when $G$ is generated by complex 
bireflections. In general, however, it is not true that (non-central) elements
of $\age \leq 1$ are always complex bireflections (nor elements with fixed 
point subspace of codimension $2$). In this regard, one of the main
assertions of Theorem \ref{main2} is the existence of complex bireflections in
the groups $G$ satisfying the hypotheses of the theorem. If one
knows that $G$ is {\it generated} by complex bireflections
(or $G$ contains complex bireflections and is {\it quasiprimitive}), 
one can then appeal to available results on such groups, 
particularly \cite{HW}, \cite{Hu}, \cite{Wa} (see also \cite{Co}).   

Interestingly, it was shown by V. Kac and K. Watanabe \cite{KW}, and
independently by N. Gordeev \cite{Go1}, that if the ring $\Sym(V)^{G}$ of 
$G$-invariants is a complete intersection for a finite group $G < GL(V)$, 
then $G$ is generated by elements with fixed point subspace of codimension $2$.
The finite groups $G < GL(V)$ with $\Sym(V)^{G}$ being a complete intersection
have been classified by \cite{Go2} and \cite{N}.    
  
In a certain sense, Theorem \ref{main2} gives indications that crepant resolutions seem to occur
mostly in low dimensions. Indeed, let $f~:X \to \CC^{d}/G$ be a crepant resolution, and let 
$K$ be the normal subgroup of $G$ generated by all elements of $\age \leq 1$. Then by
Theorem \ref{main2}, for any irreducible summand $V$ of the $K$-module $\CC^{d}$, either 
$\dim(V) \leq 10$, or the action of $K$ on $V$ contains complex bireflections (of order $2$ or $3$), 
and so the quotient $V/K$ should behave reasonably well from the point of view of algebraic geometry. (See \cite{Ha} for the case of $\SSS_{n}$ acting on
the sum $\CC^{n} \oplus \CC^{n}$ of two copies of the natural permutation 
module.)
We formulate one consequence of our results in this regard:

\begin{corol}\label{main2cr}
{\sl Let $d \geq 11$ and let $G < GL_{d}(\CC)$ be a finite irreducible, primitive, tensor 
indecomposable subgroup. Assume that $\CC^{d}/G$ is not terminal (for instance, it has a 
crepant resolution). Then one of the following statements holds. 

{\rm (i)} $Z(G) \times \AAA_{d+1} \leq G \leq (Z(G) \times \AAA_{d+1}) \cdot 2$, 
with $\AAA_{d+1}$ acting on $\CC^{d}$ as on its deleted natural permutation module.

{\rm (ii)} All junior elements of $G$ are central, and $|Z(G)| \geq d$.}
\end{corol}
   
Recall that $(G,V)$ is a {\it basic non-RT pair} if $G < GL(V)$ is a finite irreducible subgroup 
and $G = \langle g^{G} \rangle$ for every non-central element $g \in G$ with $\age(g) < 1$. 
This notion was first introduced in \cite{KL} and is of importance for the geometry of quotients 
of Calabi-Yau varieties. Our third main result is concerned with this notion and is in fact predicted
by results of \cite{KL}. 

\begin{theor}\label{main4}
{\sl Let $G < GL(V)$ be a finite irreducible subgroup. Assume that, $G$ contains  
non-central elements $g \in G$ with $\age(g) < 1$, and that $G = \langle g^{G} \rangle$ for any 
such an element. Assume in addition that $\dim(V) > 4$. Then, up to scalars, 
$G$ is a complex reflection group.}
\end{theor}
 
Theorem \ref{main4} is not valid if $\dim(V) = 4$. Examples of $4$-dimensional
basic non-RT pairs which are not projectively equivalent to a c.r.g. are given in \cite{KL};
see also Examples \ref{newRT1} and \ref{newRT2}. One should also compare Theorem \ref{main4}
with the classical result that {\sl $\CC^{d}/G$ is smooth if and only the finite subgroup 
$G < GL_{d}(\CC)$ is a complex reflection group} (see e.g. \cite[Theorem V.5.4]{B}).

There should be similar results for representations in positive 
characteristic (where we consider the eigenvalues of semisimple
elements), and similar algebro-geometric applications.
There are results which indicate that if $G < GL(V)$ with $V$
finite dimensional over an algebraically closed field, then
$k[V]^{G}$ being a polynomial ring, resp. a complete intersection, 
implies that $G$ is generated by elements trivial on a subspace of codimension 
$1$, resp. on a subspace of codimension at most $2$, cf.
for instance \cite{KM}, \cite{KW}, \cite{S}. 
Such groups have been classified (see \cite{GS} for
the last statement and references -- also in \cite{GS} finite and
algebraic groups generated
by symplectic reflections in all characteristics were classified).
The authors have recently obtained some results on the values of Brauer characters which
should be relevant.

\section{Preliminaries}
Let $V = \CC^{n}$ be endowed with standard Hermitian form $(\cdot,\cdot)$; write 
$||v|| = \sqrt{(v,v)}$ for any $v \in V$. Also let $\SA := \{ \lam \in \CC \mid |\lam| = 1\}$
and let $\BC(V)$ be the collection of all orthonormal bases of $V$. 

\subsection{Age}
\begin{defi}\label{reid} {\rm \cite{IR}, \cite{R2}.}
{\em Let $g \in GL(V)$ be conjugate to 
$\diag\left(e^{2\pi ir_{1}}, \ldots ,e^{2\pi ir_{n}}\right)$,
where $0 \leq r_{j} < 1$. Then $\age(g) = \sum^{n}_{j=1}r_{j}$.}
\end{defi}

Classical examples of non-scalar elements with age $<1$ are: {\it reflections}, resp. 
{\it complex reflections} (or {\it pseudoreflections}), {\it bireflections},  and 
{\it complex bireflections}. These cases correspond to 
$(r_{1}, \ldots ,r_{n}) = (1/2, 0, \ldots, 0)$, $(0 < r_{1} < 1, 0, \ldots, 0)$,
$(1/2, 1/2, 0, \ldots, 0)$, and $(0 < r_{1} < 1, 1-r_{1}, 0, \ldots, 0)$, respectively.
(Note that all complex bireflections considered in this paper have determinant $1$.) 
  
To deal with scalar multiples of linear tranformations, it is also convenient to define
$$\ages(g) = \inf_{\lam \in \SA}\age(\lam g)$$
for any (diagonalizable) $g \in GU(V)$. 
    
First we record the following observations, which we usually apply to linear transformations 
of finite order (as elements of $GL(V)$).

\begin{lemma}\label{trivial}
{\sl The following statements hold for any $g \in GU(V)$.

{\rm (i)} $\age(g)$ and $\ages(g)$ are well-defined, and constant on the $GU(V)$-conjugacy class of 
$g$.

{\rm (ii)} There is some $\mu \in \SA$ (of finite order, if $|g|$ is finite) 
such that $\ages(g) = \age(\mu g)$. In particular, 
$g$ is scalar if and only if $\ages(g) = 0$. 

{\rm (iii)} If $U \subseteq V$ is a $g$-invariant subspace then 
$$\age(g|_{U}) \leq \age(g) = \age(g|_{U}) + \age(g|_{V/U}).$$ 

{\rm (iv)} If $h \in GU(W)$, then 
$$\ages\left(\diag(g,h)\right) \geq \ages(g)+\ages(h),~~~
  \ages(g \otimes h) \geq \dim(W) \cdot \ages(g).$$

{\rm (v)} If $h \in GU(V)$ and $gh = hg$, then 
$$\age(gh) \leq \age(g)+\age(h),~~~~
  \ages(gh) \leq \ages(g) + \ages(h).$$}
\end{lemma}

\begin{proof}
(i) and (iii) are obvious. 

(ii) Let $e^{2\pi ir_{1}}, \ldots ,e^{2\pi ir_{m}}$,
where $0 \leq r_{1} < r_{2} < \ldots < r_{m} < 1$, be the distinct eigenvalues of $g$. Consider
the function $f(t) := \age(e^{-2\pi i t} \cdot g)$ on the interval $(0,1]$. Note that $f$ is 
decreasing on each of the intervals $(0,r_{1}]$, $(r_{1},r_{2}]$, $\ldots$, 
$(r_{m-1},r_{m}]$, $(r_{m},1]$. It follows that $\ages(g) = \inf_{t \in (0,1]}f(t)$ is attained
as the value of $f$ at one of the points $t = r_{1}, r_{2}, \ldots ,r_{m},1$. Thus 
we can take $\mu^{-1}$ to be either $1$ or one of the eigenvalues of $g$, and so it has 
finite order in $\SA$ if $|g|$ is finite. (Also notice that
if $m \geq 2$, then $\ages(g) \geq \min\{r_{2}-r_{1},1-(r_{2}-r_{1})\}$.)    

(iv) Without loss we may assume that 
$h = \diag(s_{1}, \ldots ,s_{m})$ with $s_{j} \in \SA$, and consider any $\lam \in \SA$. Then 
by (iii) we have 
$$\age\left(\lam \cdot \diag(g,h)\right) = \age(\lam g) + \age(\lam h) \geq \ages(g) + \ages(h),$$ 
$$\age(\lam g \otimes h) = \age\left(\diag(\lam s_{1}g, \ldots ,\lam s_{m}g)\right) =
  \sum^{m}_{j=1}\age(\lam s_{j} g) \geq m\cdot \ages(g).$$ 

(v) Without loss we may assume that 
$$g = \diag(e^{2\pi ir_{1}}, \ldots ,e^{2\pi ir_{m}}), ~~~~h = \diag(e^{2\pi is_{1}}, \ldots ,e^{2\pi is_{m}}),$$ 
with $0 \leq r_{j}, s_{j} < 1$. 
Then $\age(gh) \leq \sum^{m}_{j=1}(r_{j}+s_{j}) = \age(g) + \age(h)$. Next, by (ii) 
there are $\al,\beta \in \SA$ such that $\ages(g) = \age(\al g)$ and $\ages(h) = \age(\beta h)$. 
Now 
$$\ages(gh) \leq \age(\al\beta gh) = \age(\al g \cdot \beta h) \leq \age(\al g) + \age(\beta h) 
  = \ages(g) + \ages(h).$$ 
\end{proof}

In fact, by the Chen-Ruan inequality \cite{CR}, Lemma \ref{trivial}(v) also holds 
without the condition $gh = hg$. Even more, the following inequality holds, where 
$V^{X}$ denotes the common fixed point subspace for any subset $X \leq GL(V)$.

\begin{theorem}\label{cr-age} {\rm \cite{CR}}
{\sl {\rm (i)} If $x, y \in GU(V)$, then 
$$\age(x) + \age(y)  - \age(xy) + \dim V^{x,y} - \dim V^{xy} \geq 0.$$

{\rm (ii)} If $x, y, z \in GU(V)$ and $xyz = 1$, then 
$$\age(x) + \age(y)  + \age(z) \geq \dim(V) - \dim V^{x,y,z}.$$}
\end{theorem}

This theorem follows from the existence of a cohomology theory developed in
\cite{CR}, see also \cite{Hep}. We will give an elementary proof of this result.

First we set up some notation. For $V = \CC^{n}$ and $g \in GU(V)$, write 
$[g]=(v_{1}, \ldots, v_{n})$ where $0 \leq v_{1} \leq \ldots \leq v_{n} < 1$ and the eigenvalues 
of $g$ are $e^{2 \pi iv_{j}}$, $1 \leq j \leq n$.

\begin{lemma}\label{irred} 
{\sl Let $\dim(V) > 1$ and $x,y \in GU(V)$, where $x$ is a complex reflection
with $[x]=(r, 0, \ldots,0)$, $0 < r < 1$, $[y]=(a_{1}, \ldots, a_{n})$
and $[xy]=(b_{1}, \ldots, b_{n})$.  Let $H = \langle x, y \rangle$.
Then the following conditions are equivalent:

{\rm (i)} $H$ acts irreducibly.

{\rm (ii)} $x$ and $y$ have no common eigenvector.

{\rm (iii)} The collection $\{a_{1}, \ldots, a_{n}, b_{1}, \ldots, b_{n}\}$ consists
of $2n$ distinct elements.}
\end{lemma}

\begin{proof}   
If $H$ acts reducibly, then $xy = y$ on some nontrivial
$H$-invariant space, whence $a_{j} = b_{k}$ for some $j,k$.  Thus (iii)
implies (i), and certainly (i) implies (ii).   

Now assume (ii); in particular, neither $y$ nor $xy$ has an eigenvector
on $V^{x} = u^{\perp}$ (for some $0 \neq u \in V$). Note that $xu = e^{2\pi ir} u$.
If $a_{i}=a_{j}$ for $i < j$, then $y$ has a two dimensional 
eigenspace which therefore intersects $V^{x}$ nontrivially, a contradiction.
Similarly, we see that $b_{i} \neq b_{j}$.
Suppose now that both $xy$ and $y$ have a common eigenvalue $\beta$. 
In this case, again by (ii) we can find $v,w \in u^{\perp}$ such that $y(u+v) = \beta(u+v)$ and 
$xy(u+w) = \beta(u+w)$; in particular,
$y(u+w) = e^{-2\pi ir}\beta u + \beta w$. Thus 
$y(v-w) = \beta(1-e^{-2\pi ir})u + \beta(v-w)$. Note that $|\beta| = 1$ and 
$||y(v-w)|| = ||v-w||$ as $y \in GU(V)$. It follows that  
$e^{2\pi ir}=1$, a contradiction.
\end{proof} 

The key to Theorem \ref{cr-age} is the following beautiful result 
\cite[Cor. 4.7]{BH} on eigenvalue interlacing, see also \cite{MOW}.

\begin{lemma}\label{rigid} {\rm \cite{BH}}  
{\sl Let $x,y \in GU(V)$, where $x$ is a complex reflection
with $[x]=(r, 0, \ldots,0)$, $0 < r< 1$, $[y]=(a_{1}, \ldots, a_{n})$
and $[xy]=(b_{1}, \ldots, b_{n})$.   Assume that $a_{j} < a_{j+1}$
and $b_{j} < b_{j+1}$ for $1 \leq j < n$.   Assume also that
$a_{j} \neq b_{k}$ for any $j,k$.  Then one of the following holds:

{\rm (a)} $a_{1} <  b_{1} <  \ldots < a_{n} <  b_{n}$; or
 
{\rm (b)}  $b_{1} <  a_{1} < \ldots < b_{n} <  a_{n}$.}
\end{lemma}

Note that in either case $|\age(xy) - \age(y)| < 1$ and so by considering determinants
if (a) holds above, then 
$\age(x) + \age(y) = \age(xy)$ while if (b) holds, then
$\age(x)  + \age(y) = \age(xy) + 1$.  In any case, 
$\age(x)  + \age(y) \geq \age(xy)$.  

We now sketch an elementary proof of Lemma \ref{rigid}.

\begin{proof}
1) Clearly, we may assume $\dim(V) > 1$.
Let $V^{x} = u^{\perp}$ for some $0 \neq u \in V$. By Lemma \ref{irred}, $y$ cannot have any 
eigenvector in $u^{\perp}$ or $\la u \ra_{\CC}$. For $t \in \RR \setminus \ZZ$, 
let $x(t)$ be the complex reflection with $u^{\perp}$ as its reflecting hyperplane and 
$x(t)u = e^{2 \pi it}u$. Also set $x(t) = 1_{V}$ if $t \in \ZZ$. 

Now let $w(t)=x(t)y$ for $t \in \RR$.  Note that by construction, for any $t,t' \in \RR$
with $t-t' \notin \ZZ$, $w(t)$ and $w(t')$ cannot have any common eigenvector.
(Otherwise $x(t-t')$ and $y$ have a common eigenvector $v$. This $v$ must be either in 
$u^{\perp}$ or $\la u \ra_{\CC}$, contrary to the aforementioned property of $y$.)
It then follows by Lemma \ref{irred} that, when $0 \leq t < t' < 1$, all the $n$ eigenvalues 
of $w(t)$ are distinct, and $w(t)$ and $w(t')$ have no common eigenvalue.    

2) Define $a_{kn+i} = a_{i}+k$ for $k \in \ZZ$. Also, let $[w(r)] = (b_{1}(r), \ldots, b_{n}(r))$ 
and consider any $1 \leq j \leq n$. By the conclusion of 1), 
$a_{i} < b_{j}(r) < a_{i+1}$ for some $i \in \ZZ$. Note that the spectrum of
$w(t)$ depends continuously on $t \in \RR$. Hence, for $t$ in some small neighborhood of 
$r$, the $j^{\mathrm {th}}$-entry $b_{j}(t)$ of $[w(t)]$ satisfies 
$a_{i} < b_{j}(t) < a_{i+1}$. Let 
$$X := \{s \mid r \leq s < 1,~a_{i} < b_{j}(t) < a_{i+1} \mbox{ for all } t \in [r,s]\}.$$
We claim that $X = [r,1)$. Indeed, let $f := \sup X \leq 1$ and assume $f < 1$. Then there is
a sequence $\{s_{n}\} \subseteq X$ such that $\lim_{n \to \infty}s_{n} = f$. The spectrum
continuity implies that $a_{i} \leq b_{j}(f) \leq a_{i+1}$. Since $0 < f < 1$, we must have
$a_{i} < b_{j}(f) < a_{i+1}$. It is now easy to check that there is some $\epsilon > 0$ such
that $f+\epsilon \in X$, a contradiction. Thus $f = 1$, which in turn implies that 
$X = [r,1)$. Similarly, 
$$\{s \mid 0 \leq s < r,~\forall t \in [s,r],~a_{i} < b_{j}(t) < a_{i+1}\} = (0,r].$$
We have shown that 
\begin{equation}\label{int}
  a_{i} < b_{j}(t) < a_{i+1} \mbox{ for all }t \in (0,1).
\end{equation}

3) Replacing $(x,y)$ by $(x^{-1},xy)$ if necessary, we may assume that $a_{1} < b_{1}$. Note 
that $b_{1}(0) = a_{1}$. Hence $b_{1}(t)$ is close to $a_{1}$ when $t \in (0,1)$ is small 
enough and so (\ref{int}) implies that $a_{1} < b_{1}(t) < a_{2}$ for all $t \in (0,1)$;
in particular, $b_{1} = b_{1}(r) < a_{2}$. Also, 
$a_{1} \leq b_{1}(1) = \lim_{t \to 1}b_{1}(t) \leq a_{2}$. Since $b_{1}(1)$ is some $a_{k}$,
we get $b_{1}(1) \in \{a_{1},a_{2}\}$. Moreover, if $b_{1}(1) = a_{1} = b_{1}(0)$, then
the continuity of $b_{1}(t)$ on $[0,1]$ implies that some $w(t)$, $w(t')$ with
$0 < t < t' < 1$ have a common eigenvalue, contrary to the conclusion of 1). So 
$b_{1}(1) = a_{2}$.  

4) Next, $b_{2}(0) = a_{2}$. If $b_{2} = b_{2}(r) > a_{2}$, then, as above, (\ref{int}) implies
that $a_{2} < b_{2}(t) < a_{3}$ for all $t \in (0,1)$. Assume the contrary: $b_{2} < a_{2}$.
Again by (\ref{int}) we must now have $a_{1} < b_{2}(t) < a_{2}$ for all $t \in (0,1)$. Arguing 
as in 3) we get $b_{2}(1) \in \{a_{1},a_{2}\}$ and $b_{2}(1) \neq b_{2}(0) = a_{2}$, i.e.
$b_{2}(1) = a_{1}$. On the other hand, $b_{2}(t) \geq b_{1}(t)$ for all $t \in [0,1]$, 
whence $b_{2}(1) \geq b_{1}(1) = a_{2}$, a contradiction. We have shown 
that $a_{2} < b_{2}(t) < a_{3}$ for all $t \in (0,1)$.
Continuing in the same fashion, we get 
$a_{j} < b_{j}(t) < a_{j+1}$ for all $j$ and $t \in (0,1)$.
\end{proof}

{\bf Proof of Theorem \ref{cr-age}.} Notice that $\age(z) = \dim(V) - \age(xy) - \dim V^{xy}$, 
so (i) and (ii) are equivalent. Next, diagonalize $x$ and then write $x$ as a product 
of $m \leq n$ commuting complex reflections. 
To prove statement (i), we proceed by induction on $m$.
First assume that $x$ is a complex reflection. 
Let $H = \langle x, y \rangle$. If $H$ does not act irreducibly, the result follows
by induction on $\dim V$ (by writing $V=W \perp W^{\perp}$ where $W$ is $H$-invariant).
So assume this is not the case; in particular, $V^{x,y} = 0$. We need to prove that:
$\age(x) + \age(y) \geq \age(xy) +  \dim V^{xy}$. By Lemma \ref{irred}, 
no eigenspace of $y$ or $xy$ has dimension more than $1$, and $xy$ and $y$ have no common 
eigenvalues. By Lemma \ref{rigid} and the remarks before its proof,  
$\age(x) + \age(y) \geq \age(xy)$. 
If $xy$ has no trivial eigenvalue we are done.  So we may assume
that $xy$ does have exactly one trivial eigenvalue, whence $y$ has no trivial
eigenvalue.  Thus the case (b) of Lemma \ref{rigid} holds and so  
$\age(x) + \age(y) = \age(xy) + 1$ as desired.

For the induction step, write $x=sx'$ where $s$ is a complex reflection
and $x'$ is a product of $m-1$ complex reflections than $x$ 
and $\age(x) = \age(s) + \age(x')$.
By the complex reflection case, 
$$\age(s) + \age(x'y) + \dim V^{s,x'y} \geq \age(xy) + \dim V^{xy}.$$
By induction, 
$$\age(x') + \age(y) + \dim V^{x',y} \geq \age(x'y) + \dim V^{x'y}.$$ 
Note that $V^{s,x'y} \cap V^{x',y} = V^{s,x',y} \subseteq V^{x,y}$ and 
$V^{s,x'y}, V^{x',y} \subseteq V^{x'y}$, whence 
$$\dim V^{x'y} + \dim V^{x,y} \geq \dim V^{s,x'y} + \dim V^{x',y}.$$
The last three relations on dimensions readily imply
$$\age(x) + \age(y) = \age(s) + \age(x') + \age(y) 
  \geq \age(xy) + \dim V^{xy} - \dim V^{x,y}.\hfill \Box$$

\subsection{The set-up $\sta$}
We are interested in finite subgroups of $GL(V)$ that contain nontrivial elements
of age $<1$, resp. $\leq 1$. Of course it would be very difficult to classify these
groups without extra assumptions on them.

\begin{lemma}\label{red1}
{\sl Let $W$ be a finite dimensional vector space over $\CC$ and let $G < GL(W)$ be a 
finite subgroup containing a nontrivial element $g$ with $\age(g)<1$, resp. $\age(g) \leq 1$. Then 
there is a normal subgroup $K \lhd G$ and a nonzero $K$-invariant subspace $V$ of 
$W$ such that all the following conditions hold:

{\rm (i)} $K$ acts irreducibly on $V$;

{\rm (ii)} $g \in K$ and $0 < \age(g|_{V}) < 1$, resp. $0 < \age(g|_{V}) \leq 1$;

{\rm (iii)} $K$ is generated by the set of its elements whose restrictions to $V$ have 
age $<1$, resp. $\leq 1$.\\
In fact, if $0 \neq U \subseteq W$ is any $K$-submodule, then $K$ is generated by 
the set of its elements whose restrictions to $U$ have age $<1$, resp. $\leq 1$.}
\end{lemma}

\begin{proof}
Let $\XC$ denote the set of all nontrivial elements of $G$ that have age $<1$, resp. $\leq 1$, 
and define $K = \langle \XC \rangle$. Then $K \lhd G$ and $K \ni g$. 
Decompose $W$ into a direct sum $\oplus^{s}_{i=1}V_{i}$ of irreducible $K$-submodules. 
We may assume by Lemma \ref{trivial} that $0 < \age(g|_{V_{1}}) < 1$, resp. $0 < \age(g|_{V_{1}}) \leq 1$. 
Let $\YC := \{h \in K \mid \age(h|_{V_{1}}) < 1, \mbox{ resp. }\leq 1\}$. Observe that 
$\XC \subseteq \YC$, whence $K = \langle \YC \rangle$. Thus $V := V_{1}$ satisfies (i) -- (iii).

Next, let $0 \neq U \subseteq W$ be any $K$-submodule and let 
$\YC' := \{h \in K \mid \age(h|_{U}) < 1, \mbox{ resp. }\leq 1\}$. Then again 
$\XC \subseteq \YC'$ and so $K = \langle \YC' \rangle$.  
\end{proof}  

Lemma \ref{red1} shows that it is natural to restrict our attention to the following 
set-up, which is slightly more general than the one considered in \cite{KL}:

\vspace{3mm}
$\sta~:~$\begin{tabular}{l}$G$ is a finite irreducible subgroup of $\GC = GL(V)$ and 
        $Z(\GC)G=\langle \XC \rangle$,\\
where 
$\XC:=  \{g \in Z(\GC)G \mid 0 < \age(g) < 1, \mbox{ resp. }0 < \age(g)\leq 1\}$. \end{tabular}
 
\vspace{3mm}
The condition $\sta$ means that, up to scalars, the finite irreducible subgroup $G < GL(V)$ is
generated by some nontrivial elements with age $< 1$, resp. $\leq 1$. In fact we can 
even assume that these generators are non-scalar:

\begin{remar}
{{\sl Assume $G$ satisfies $\sta$ and $\dim(V) > 1$. Then $\XC^{*} \neq \emptyset$ and 
$Z(\GC)G = Z(\GC)\langle \XC^{*}\rangle$, where
$\XC^{*}:=  \{g \in G \mid 0 < \ages(g) < 1, \mbox{ resp. }0 < \ages(g)\leq 1\}$.} 
{\rm Indeed, if $\XC^{*} = \emptyset$ then all $h \in \XC$ are scalar and so
is $G$, contradicting the condition $\dim(V) > 1$. Next, any $x \in G$ can be written as 
$g_{1} \ldots g_{m}$ with $g_{i} = \al_{i}h_{i} \in \XC$, $\al_{i} \in \SA$,
$h_{i} \in G$, and $h_{i} \in \XC^{*}$ precisely when $i \in J$ for some subset 
$J \subseteq \{1,2, \ldots m\}$. Then $x = \lam \prod_{i \in J} h_{i} \in Z(\GC) \langle \XC^{*} \rangle$
for $\lam = \prod^{m}_{i=1}\al_{i} \cdot \prod_{j \notin J}h_{j}$.}}
\end{remar}  

\subsection{Deviations}
A natural invariant metric on $GU(V)$ is defined as follows:

\begin{defi}\label{metric}
{\em Let $T \in GL(V)$ be conjugate to 
$\diag\left(e^{2\pi ir_{1}}, \ldots ,e^{2\pi ir_{n}}\right)$,
where $0 \leq r_{j} < 1$. Then $||T|| = (\sum^{n}_{j=1}\min\{r_{j},1-r_{j}\}^{2})^{1/2}$.}
\end{defi}

For our purposes it is more convenient to work with the following:

\begin{defi}\label{dj}
{\em Let $j$ be any positive number and let $T \in GU(V)$ be any unitary linear operator.
Then 
$$\DD(T) = \inf_{\lam \in \SA,~B \in \BC(V)}
  \left(\sum_{b \in B}||T(b)-\lam b||^{j}\right)^{1/j}.$$}
\end{defi}

This definition is a slight generalization of \cite[Definition 27]{KL} (where one 
takes $\lam = 1$ instead of the infimum over all $\lam \in \SA$). First we list some basic 
properties of $\DD(T)$.

\begin{lemma}\label{basic1}
{\sl Let $A, T \in GU(V)$ and $\al \in \SA$. Then the following hold:

{\rm (i)} $\DD(T) = \DD(\al T)$;

{\rm (ii)} $\DD(T) = \DD(ATA^{-1})$;

{\rm (iii)} $\DD(T) = \DD(T^{-1})$.}
\end{lemma}

\begin{proof}
(i) Clearly $\al T \in GU(V)$.  Consider any $\lam \in \SA$ and $B \in \BC(V)$. Then 
$$\sum_{b \in B}||\al T(b)-\lam b||^{j} = \sum_{b \in B}||T(b)-\al^{-1}\lam b||^{j}
  \geq \DD(T)^{j}.$$
Taking infimum over all $\lam \in \SA$ and $B \in \BC(V)$ we get 
$\DD(\al T) \geq \DD(T)$. Applying this inequality to $S := \al T$ and $\al^{-1}$ we 
obtain
$\DD(T) = \DD(\al^{-1} S) \geq \DD(S) = \DD(\al T)$, and the claim follows. 

\smallskip
(ii) Consider any $\lam \in \SA$ and $B \in \BC(V)$. Then $A^{-1}(B) \in \BC(V)$, and 
$$\sum_{b \in B}||ATA^{-1}(b)-\lam b||^{j} = 
  \sum_{c = A^{-1}b \in A^{-1}(B)}||A(T(c)-\lam c)||^{j}
  = \sum_{c \in A^{-1}(B)}||T(c)-\lam c||^{j} 
  \geq \DD(T)^{j}.$$
Taking infimum over all $\lam \in \SA$ and $B \in \BC(V)$ we get 
$\DD(ATA^{-1}) \geq \DD(T)$. Applying this inequality to $S := ATA^{-1}$ and $A^{-1}$ we 
obtain
$\DD(T) = \DD(A^{-1}S(A^{-1})^{-1}) \geq \DD(S) = \DD(ATA^{-1})$, and the claim follows. 

\smallskip
(iii) Consider any $\lam \in \SA$ and $B \in \BC(V)$. Then  
$$\sum_{b \in B}||T^{-1}(b)-\lam b||^{j} = 
  \sum_{b \in B}||\lam^{-1}T(T^{-1}(b)-\lam b)||^{j}
  = \sum_{b \in B}||T(b)-\lam^{-1}b||^{j} 
  \geq \DD(T)^{j}.$$
Taking infimum over all $\lam \in \SA$ and $B \in \BC(V)$ we get 
$\DD(T^{-1}) \geq \DD(T)$. Applying this inequality to $S := T^{-1}$ we 
get
$\DD(T) = \DD(S^{-1}) \geq \DD(S) = \DD(T^{-1})$, and so the claim follows.
\end{proof}

Most of the time we will work with $\DD(T)$ where $j = 1$ or $2$. 

\begin{lemma}\label{basic2}
{\sl For $T \in GU(V)$ the following hold:

{\rm (i)} $\DB(T) \leq \DA(T) \leq \sqrt{\dim(V)} \cdot \DB(T)$;

{\rm (ii)} $\DA(T) \geq \dim(V) - |\Tr(T)|$;

{\rm (iii)} $\DB(T)^{2} = 2(\dim(V) - |\Tr(T)|)$. Moreover, for any $B \in \BC(V)$ we have
$$\DB(T) = \inf_{\lam \in \SA}
  \left(\sum_{b \in B}||T(b)-\lam b||^{2}\right)^{1/2}.$$}
\end{lemma}

\begin{proof} 
(i) For any $\lam \in \SA$ and $B \in \BC(V)$ we have
$$\DB(T) \leq (\sum_{b \in B}||T(b)-\lam b||^{2})^{1/2} \leq \sum_{b \in B}||T(b)-\lam b||.$$ 
Taking infimum over all $\lam,B$ we get $\DB(T) \leq \DA(T)$. Next, again for any 
$\lam \in \SA$ and $B \in \BC(V)$ by the Cauchy-Schwarz inequality we have 
$$\DA(T) \leq \sum_{b \in B}||T(b)-\lam b|| \leq \sqrt{\dim(V)} \cdot 
  (\sum_{b \in B}||T(b)-\lam b||^{2})^{1/2}.$$ 
Taking infimum over all $\lam,B$ we get $\DA(T) \leq \sqrt{\dim(V)}\cdot \DB(T)$.

\smallskip
(ii) Consider any $\lam \in \SA$ and $B \in \BC(V)$. Let $(a_{ij})_{1 \leq i,j \leq n}$
be the matrix of $T$ in the basis $B$. Observe that 
$$\sum_{b \in B}||T(b)-\lam b|| = 
  \sum^{n}_{l=1} \left(\sum^{n}_{k=1}|a_{kl}-\lam\delta_{k,l}|^{2}\right)^{1/2} \geq 
  \sum^{n}_{l=1}|\lam^{-1}a_{ll}-1| \geq \sum^{n}_{l=1}\Re(1-\lam^{-1}a_{ll})$$
$$ = n-\Re(\lam^{-1}\sum^{n}_{l=1}a_{ll}) \geq n-|\lam^{-1}\sum^{n}_{l=1}a_{ll}| = 
\dim(V)-|\Tr(T)|.$$
Taking infimum over all $\lam,B$ we arrive at the claim.

\smallskip
(iii) Consider an arbitrary $B \in \BC(V)$ and let 
$A := (a_{ij})_{1 \leq i,j \leq n}$ be the matrix of $T$ in the basis $B$. 
For any $\lam \in \SA$ we have 
$$\sum_{b \in B}||T(b)-\lam b||^{2} = 
  \sum^{n}_{l=1} \sum^{n}_{k=1}|\delta_{k,l}-\lam^{-1}a_{kl}|^{2} = 
  \sum_{1 \leq k,l \leq n}X_{kl}\bar{X}_{kl} = \Tr(\tn \bar{X} \cdot X),$$
where $X := (\delta_{k,l}-\lam^{-1}a_{kl})_{1 \leq k,l \leq n} = I_{n}-\lam^{-1}A$. 
Since $T \in GU(V)$, there is a matrix $C$ with $\tn \bar{C} \cdot C = I_{n}$ and a 
diagonal matrix $E = \diag(\eps_{1}, \ldots \eps_{n})$ with $|\eps_{k}| = 1$ such that 
$A = \tn \bar{C} EC$. Then $X = \tn \bar{C}DC$ for 
$D := I_{n}-\lam^{-1}E = \diag(1-\al_{1}, \ldots ,1-\al_{n})$ with $\al_{i} := \lam^{-1}\eps_{i}$
(all of modulus $1$). It follows that
\begin{equation}\label{d21}
  \begin{array}{l}\sum_{b \in B}||T(b)-\lam b||^{2}  = 
  \Tr(\tn \bar{X} \cdot X) = \Tr(\tn \bar{C}\tn \bar{D}C \cdot \tn \bar{C} DC) = 
  \Tr(\tn \bar{D}D)\\
  \\ 
   = \sum^{n}_{l=1}|1-\al_{l}|^{2} 
    = \sum^{n}_{l=1}\left(1+|\al_{l}|^{2}-2\Re(\al_{l})\right) 
    = 2n-2\Re(\sum^{n}_{l=1}\al_{l}).\end{array}
\end{equation} 
In particular,
\begin{equation}\label{d22}
  \sum_{b \in B}||T(b)-\lam b||^{2} \geq 2n-2|\sum^{n}_{l=1}\al_{l}| = 
    2n-2|\sum^{n}_{l=1}\eps_{l}| = 2(n-|\Tr(T)|).
\end{equation}
Taking infimum over all $\lam,B$ we obtain $\DB(T)^{2} \geq 2(n-|\Tr(T)|)$. 

Now, in the above computation we choose $\lam = \lam_{0} := e^{i\theta}$, where 
$\Tr(T) = re^{i\theta}$ and $|\Tr(T)| = r \geq 0$. Then 
$$\sum^{n}_{l=1}\al_{l} = \lam_{0}^{-1}\sum^{n}_{l=1}\eps_{l} = \lam_{0}^{-1}\Tr(T) = 
  e^{-i\theta}re^{i\theta} = r = |\Tr(T)|.$$
Then (\ref{d21}) implies that 
$$\DB(T)^{2} \leq \sum_{b \in B}||T(b)-\lam_{0}b||^{2} = 2n-2\Re(\sum^{n}_{l=1}\al_{l}) = 
  2(n-|\Tr(T)|) \leq \DB(T)^{2}$$
Together with (\ref{d22}), this last inequality chain yields that 
$$\DB(T)^{2} = 2(n-|\Tr(T)|) = \sum_{b \in B}||T(b)-\lam_{0}b||^{2} = 
  \inf_{\lam \in \SA}\sum_{b \in B}||T(b)-\lam b||^{2}.$$
\end{proof}

The relationship between $||T||$ and $\DB(T)$ can be described as follows:

\begin{corol}\label{basic2m}
{\sl For $T \in GU(V)$ one has $4 \cdot \inf_{\lam \in \SA}||\lam T|| < \DB(T) \leq 2\pi \cdot ||T||$.}
\end{corol}

\begin{proof}
We may assume that $T = \diag\left(e^{2\pi ia_{1}}, \ldots ,e^{2\pi ia_{n}}\right)$ in some 
basis $B \in \BC(V)$,
where $-1/2 \leq a_{j} < 1/2$; in particular, $||T|| = (\sum^{n}_{j=1}a_{j}^{2})^{1/2}$. 
It is easy 
to check that the function $(1-\cos(x))/x^{2}$ is decreasing on $(0,\pi]$, whence 
$2/\pi^{2} \leq (1-\cos(x))/x^{2} \leq 1/2$ for $-\pi \leq x \leq \pi$. Taking 
$x = 2\pi a_{j}$, we get $4/\pi^{2} < |e^{2\pi ia_{j}}-1|^{2}/4\pi^{2}a_{j}^{2} \leq 1$, whence 
$4||T|| \leq (\sum_{b \in B}||T(b)-b||^{j})^{1/2} \leq 2\pi||T||$. Now the statement follows
by applying this inequality to $\lam T$ for all $\lam \in \SA$ and using Lemma \ref{basic2}(iii).
\end{proof}

\begin{lemma}\label{basic3}
{\sl Let $T_{1}, \ldots, T_{k} \in GU(V)$. Then the following hold:

{\rm (i)} $\DB(T_{1}T_{2} \ldots T_{k})^{2} \leq k \cdot \sum^{k}_{i=1}\DB(T_{i})^{2}$;

{\rm (ii)} $\DB(T_{1}T_{2}T_{1}^{-1}T_{2}^{-1})^{2} \leq 4 \cdot 
  \min\{\DB(T_{1})^{2},\DB(T_{2})^{2}\}$.}
\end{lemma}

\begin{proof}
(i) For any $\lam_{i} \in \SA$, $B \in \BC(V)$, and $b \in B$ we have
$$T_{1}T_{2} \ldots T_{k}(b)-\lam_{1}\lam_{2} \ldots \lam_{k}b = 
  T_{1} \ldots T_{k-1}(T_{k}b-\lam_{1}b) + \lam_{1}T_{1} \ldots T_{k-2}(T_{k-1}b-\lam_{2}b) + $$
$$ + \lam_{1}\lam_{2}T_{1} \ldots T_{k-3}(T_{k-2}b-\lam_{3}b) + \ldots + 
   \lam_{1} \ldots \lam_{k-1}(T_{1}b-\lam_{k}b).$$ 
By the Cauchy-Schwarz inequality, $||\sum^{k}_{i=1}v_{i}||^{2} \leq k\sum^{k}_{i=1}||v_{i}||^{2}$
for any $v_{1}, \ldots, v_{k} \in V$.  
Since $|\lam_{i}|=1$ and $T_{i}$ is unitary for all $i$, it now follows that 
$$||T_{1}T_{2} \ldots T_{k}(b)-\lam_{1} \ldots \lam_{k}b||^{2} \leq 
  k\cdot \sum^{k}_{i=1}||T_{i}b-\lam_{i}b||^{2},$$
But $\lam_{1} \ldots \lam_{k} \in \SA$, hence  
$\DB(T_{1} \ldots T_{k})^{2} \leq k \cdot \sum^{k}_{i=1}\sum_{b \in B}||T_{i}b-\lam_{i}b||^{2}$. 
Taking infimum over all $\lam_{i} \in \SA$ and applying Lemma \ref{basic2}(iii),
we obtain 
$\DB(T_{1} \ldots T_{k})^{2} \leq k \cdot \sum^{k}_{i=1}\DB(T_{i})^{2}$. 

\smallskip
(ii) By (i) applied to $S := (T_{1}T_{2}T_{1}^{-1}) \cdot T_{2}^{-1}$ and by 
Lemma \ref{basic1},
$$\DB(S)^{2} \leq 2(\DB(T_{1}T_{2}T_{1}^{-1})^{2} + \DB(T_{2}^{-1})^{2}) = 4\DB(T_{2})^{2}.$$
Breaking up $S = T_{1} \cdot T_{2}T_{1}^{-1}T_{2}^{-1}$ and arguing similarly, we get
$\DB(S)^{2} \leq 4\DB(T_{1})^{2}$. 
\end{proof}

Lemma \ref{basic2}(iii) and Lemma \ref{basic3} yield the following inequalities which
we believe to be new and nontrivial.

\begin{corol}\label{basic23}
{\sl Let $\chi$ be any complex character of any finite group $G$ and let 
$g_{1}, \ldots ,g_{k} \in G$. Then 

{\rm (i)} $(k^{2}-1)\chi(1)-k\sum^{k}_{i=1}|\chi(g_{i})|+|\chi(\prod^{k}_{i=1}g_{i})| \geq 0$;

{\rm (ii)} $3\chi(1)-4|\chi(g_{i})|+|\chi([g_{1},g_{2}])| \geq 0$ for $i = 1,2$.
\hfill $\Box$}
\end{corol}

Now we can prove an upper bound that links the dimension, covering number, and deviation
together.

\begin{lemma}\label{bound1}
{\sl Let $V = \CC^{n}$ with $n > 1$, $G < GL(V)$ a finite irreducible subgroup, and let $g \in G$. 
Assume that any element of $G/Z(G)$ is a product, of length at most $\bn$, of conjugates of 
$\bg = gZ(G)$. Then $\dim(V) \leq (\bn\DB(g))^{2}/2$.}
\end{lemma}

\begin{proof}
By Weyl's unitarian trick we can equip $V$ with a $G$-invariant Hermitian form and assume
$G < GU(V)$. Consider any element $h \in G \setminus Z(G)$. Then $h = g_{1}g_{2} \ldots g_{k}z$ 
with $g_{i} \in g^{G}$, $z \in Z(G)$, and $k \leq \bn$. By Schur's Lemma,
$z$ is scalar, hence $\DB(h) = \DB(h')$ for $h' := g_{1} \ldots g_{k}$ by Lemma \ref{basic1}(i).
Next, by Lemma \ref{basic1}(ii) and Lemma \ref{basic3}(i), 
$\DB(h')^{2} \leq k\sum^{k}_{i=1}\DB(g_{i})^{2} = k^{2}\DB(g)^{2}$. It follows that 
$\DB(h)^{2} \leq (\bn\DB(g))^{2}$.

Now by Burnside's theorem on zeros we can choose $h$ such that 
$\Tr(h) = 0$. By Lemma \ref{basic2}(iii), $\DB(h)^{2} = 2n$ and so
$2n \leq (\bn\DB(g))^{2}$.    
\end{proof}

Recall that a finite group $G$ is {\it almost quasi-simple}, if 
$S \lhd G/Z(G) \leq \Aut(S)$ for some finite non-abelian simple group $S$.
For any such an $S$ and any $x \in S$, let $\an(x)$ 
be the minimal number of $\Aut(S)$-conjugates of $x$ which generate the subgroup $\la S,x \ra$.
A sharp upper bound on $\an(x)$ for $1 \neq x \in \Aut(S)$ has been obtained in \cite{GS}.
We will need the following result of \cite{GT2} that uses $\an(x)$ to bound the 
dimension of eigenspaces:   

\begin{lemma}\label{fix} {\rm \cite[Lemma 3.2]{GT2}}
{\sl Let $G$ be a finite almost quasi-simple group acting faithfully and irreducibly on a 
finite dimensional vector space $V$ over a field $\FF$, and let $g \in G \setminus Z(G)$. Then 
the dimension of any eigenspace of $g$ on $V$ is at most 
$\dim(V)-\dim(V)/\an(gZ(G))$. 
\hfill $\Box$}
\end{lemma}

Next we prove key inequalities which relate the age of any element $g \in GU(V)$ to its deviations.

\begin{propo}\label{arc}
{\sl Let $g \in GU(V)$ and let $X$ be a non-empty subset of eigenvalues of $g$. Let $m \geq 1$ 
be such that any $\lam \in X$ occurs as an eigenvalue of $g$ on $V$ with multiplicity at least $m$.
Also assume that the shortest arc of $\SA$ that contains $X$ has length $\geq \delta >0$. Then

{\rm (i)} $2\pi \cdot \age(g) - \DA(g) \geq m(\delta-2\sin(\delta/2))$; in particular,
$\DA(g) \leq 2\pi\cdot \age(g)$.

{\rm (ii)} $4\pi \cdot \age(g) - \DB(g)^{2} \geq 2m(\delta-1+\cos(\delta))$.

{\rm (iii)} $\DB(g)^{2} \leq (2.9)\pi\cdot \age(g)$. In fact, if 
$\delta \geq \pi-\sin^{-1}(0.725)$ then 
$$(2.9)\pi\cdot\age(g)-\DB(g)^{2} \geq m\left\{(1.45)\delta - 2(1-\cos(\delta))\right\}.$$}
\end{propo}

\begin{proof}
Let $g$ be represented by $\diag(\eps_{1}, \ldots ,\eps_{n})$ in a basis $B_{0} \in \BC(V)$,
where $\eps_{j} = e^{2\pi ir_{j}}$, $0 \leq r_{j} < 1$; in particular,
$\age(g) = \sum^{n}_{j=1}r_{j}$.  

(i) Consider the function $f(x) := 2\pi x-|e^{2\pi ix}-1| = 2\pi x - 2\sin(\pi x)$ on 
$[0,\infty)$. Since $f'(x) = 2\pi-2\pi\cos(\pi x) \geq 0$, $f$ is increasing; also, $f(x) \geq 0$.
Taking $\lam = 1$ and $B = B_{0}$ in the proof of Lemma \ref{basic2}(ii), we see that 
$\DA(g) \leq \sum^{n}_{j=1}|\eps_{j}-1|$. Hence 
$$2\pi\cdot\age(g)-\DA(g) \geq \sum^{n}_{j=1}f(r_{j}) \geq \sum_{j~:~\eps_{j} \in X}f(r_{j}).$$  
Suppose that $0 \leq r_{j} \leq \delta'/2\pi < \delta/2\pi$ for all $\eps_{j} \in X$. Then 
$X$ is contained in the arc (from $1$ to $e^{i\delta'}$) of length $\delta' < \delta$, 
contrary to the assumption. So without loss we may assume that $\eps_{1} \in X$ and 
$r_{1} \geq \delta/2\pi$. Since $\eps_{1}$ occurs as an eigenvalue of $g$ with multiplicity
$\geq m$, we get 
$$2\pi\cdot\age(g)-\DA(g) \geq  \sum_{j~:~\eps_{j} \in X}f(r_{j}) \geq mf(r_{1}) 
  \geq mf(\delta/2\pi) = m(\delta-2\sin(\delta/2)).$$ 

\smallskip
(ii) Consider the function 
$h(x) := 4\pi x-|e^{2\pi ix}-1|^{2} = 4\pi x - 2(1-\cos(2\pi x))$ on 
$[0,\infty)$. Since $h'(x) = 4\pi(1-\sin(2\pi x)) \geq 0$, $h$ is increasing, whence 
$h(x) \geq h(0) = 0$.
Taking $\lam = 1$ and $B = B_{0}$ in the proof of Lemma \ref{basic2}(iii), we see that 
$\DB(g)^{2} \leq \sum^{n}_{j=1}|\eps_{j}-1|^{2}$. Hence 
$$4\pi\cdot\age(g)-\DB(g)^{2} \geq \sum^{n}_{j=1}h(r_{j}) 
  \geq \sum_{j~:~\eps_{j} \in X}h(r_{j}).$$  
As in (i), we may assume without loss that $\eps_{1} \in X$ and 
$r_{1} \geq \delta/2\pi$. Since $\eps_{1}$ occurs as an eigenvalue of $g$ with multiplicity
$\geq m$, we get 
$$4\pi\cdot\age(g)-\DB(g)^{2} \geq  \sum_{j~:~\eps_{j} \in X}h(r_{j}) \geq mh(r_{1}) 
  \geq mh(\delta/2\pi) = 2m(\delta-1+\cos(\delta)).$$ 

\smallskip
(iii) Consider the function 
$t(x) := (1.45)x-|e^{ix}-1|^{2} = (1.45)x - 2(1-\cos(x))$ on 
$[0,2\pi]$. Since $t'(x) = 1.45-2\sin(x)$, $t$ is increasing on 
$[0,\theta] \cup [\pi-\theta,2\pi]$ and descreasing on $[\theta,\pi-\theta]$, 
where $\theta := \sin^{-1}(0.725)$. Now $t(0) = 0$ and $t(\pi-\theta) > 0.0018$, and so
$t(x) \geq 0$ on $[0,2\pi]$. As above, 
$\DB(g)^{2} \leq \sum^{n}_{j=1}|\eps_{j}-1|^{2}$, hence 
$$(2.9)\pi\cdot\age(g)-\DB(g)^{2} \geq \sum^{n}_{j=1}t(2\pi r_{j}) \geq 0.$$
Next suppose that $\delta \geq \pi-\theta$. As in (i), we may assume without loss that 
$\eps_{1} \in X$ and $r_{1} \geq \delta/2\pi$. Since $\eps_{1}$ occurs as an eigenvalue of $g$ 
with multiplicity $\geq m$, $t(x) \geq 0$ and $t$ is increasing on $[\pi-\theta,2\pi]$, we
see that $(2.9)\pi\cdot\age(g)-\DB(g)^{2} \geq m \cdot t(2\pi r_{1}) \geq m \cdot t(\delta)$.
\end{proof}

Proposition \ref{arc} yields the following immediate consequence:

\begin{corol}\label{bound2}
{\sl Let $G < GL(V)$ be a finite subgroup. Assume $g \in G$ is such that $\ages(g) \leq 1$. 
Then $\DA(g) \leq 2\pi$ and $\DB(g)^{2} \leq (2.9)\pi < 9.111$. Furthermore,
$\dim(V)-|\Tr(g)| \leq (1.45)\pi < 4.556$. In fact, 
$$\dim(V)-|\Tr(g)| < \left\{ \begin{array}{ll}  
  4.278, & \mbox{if }\delta \geq \pi,\\
  3.632, & \mbox{if }\delta \geq 6\pi/5,\\
  3.019, & \mbox{if }\delta \geq 4\pi/3,\\
  2.676, & \mbox{if }\delta \geq 7\pi/5,\\
  2.139, & \mbox{if }\delta \geq 3\pi/2,\end{array} \right.$$
where $\delta$ is the length of the shortest arc of $\SA$ that contains all eigenvalues of $g$.}
\end{corol}

\begin{proof}
We apply Lemma \ref{basic2}(iii), and Proposition \ref{arc}(i), (iii), with
$X = \Spec(\lam g,V)$ for any $\lam \in \SA$. Then the claims follow by taking infimum 
over all $\lam \in \SA$.
\end{proof}

Taking $g = \diag(1,1, \ldots ,1,-1,-1) \in GL_{n}(\CC)$ with $n \geq 4$, we see that 
$\age(g) = 1$ and $\DB(g)^{2} = 8$. In fact, the the complex reflection 
$g = \diag(e^{2\pi i/3},1, \ldots, 1) \in GL_{n}(\CC)$ has 
$\age(g) = 1/3$ and $\DB(g)^{2} = 2(n-\sqrt{n^{2}-3n+3})$. Hence, when $n \to \infty$, 
$\DB(g)^{2} \to 3$, yielding $\DB(g)^{2}/\age(g) \to 9$. Thus the constant 
$(2.9)\pi \approx 9.111$ in Proposition \ref{arc}(iii) and Corollary 
\ref{bound2} is quite good.  

\begin{lemma}\label{red2}
{\sl Let $G < GL(V)$ be irreducible, primitive, and tensor indecomposable 
on $V$, with a normal subgroup $L \lhd G$ such that $L' \not\leq Z(G)$.

{\rm (i)} For any non-scalar $g \in G$, there exists $h \in L \setminus Z(L)$ such that 
$\DB(h)^{2} \leq 4\DB(g)^{2}$. 

{\rm (ii)} Assume furthermore that $\dim(V) > 1$ and that $G$ satisfies the set-up $\sta$.   
Then there exists $h \in L \setminus Z(L)$ such that $\DB(h)^{2} < 36.444$.}
\end{lemma}

\begin{proof}
(i) By \cite[Lemma 2.5]{GT3}, $L'$ is irreducible on $V$. We claim that there exists 
$u \in L$ such that $[g,u] \notin Z(L)$. Assume the contrary: 
$[g,u] \in C_{G}(L)$ for any $u \in L$. Then for any $u,v \in L$ we have 
$[[u,v],g]= ([[v,g],u] \cdot [[g,u],v])^{-1} = 1$, whence $[g,L'] = 1$. By Schur's Lemma,
the irreducibility of $L'$ on $V$ now implies that $g$ is scalar, a contradiction.
Now we define $h = [g,u] \in L \setminus Z(L)$, and we are done by Lemma \ref{basic3}(ii).

(ii) If every $g \in \XC$ acts scalarly on $V$, then so does $G$. But in this case 
$\dim(V) = 1$, a contradiction. Hence at least one $g \in \XC$ is non-scalar, and has
age $\leq 1$. Now the claim follows from (i) and Corollary \ref{bound2}.  
\end{proof}

\subsection{Elements of small order}
To estimate the age of elements of small order, we will need the following two statements.

\begin{lemma}\label{small1}
{\sl Assume $g \in GU(V)$ is conjugate to  
$$\diag(\al_{1}, -\al_{1}, \al_{2}, -\al_{2}, \ldots ,\al_{m},-\al_{m},\beta_{1}, \ldots ,\beta_{s}).$$ 
Then $\ages(g) \geq m/2$. Moreover, if $\ages(g) = m/2$ then $g$
has exactly two distinct eigenvalues.}
\end{lemma}

\begin{proof}
Suppose $\ages(g) \leq m/2$. Then $\age(\mu g) \leq m/2$ for some $\mu \in \SA$ by Lemma \ref{trivial}(ii). 
Note that the contribution of the pair $(\mu\al_{i},-\mu\al_{i})$ to $\age(g)$ is at least $1/2$, and it 
equals $1/2$ precisely when $\al_{i} = \pm \mu^{-1}$. Next, the contribution of $\mu\beta_{j}$ to $\age(g)$ is 
at least $0$, and it equals $0$ precisely when $\beta_{j} = \mu^{-1}$. Hence the statements follow.  
\end{proof}

\begin{lemma}\label{small2}
{\sl Let $g \in GU(V)$ be a non-scalar element of age $\leq 1$, $\dim(V) \geq 4$, 
and let $\lam g$ have order $1 < m \leq 5$ for some $\lam \in \SA$. Then there is some 
$\mu \in \SA$ such that either $\mu g$ is a complex reflection, or one of the following statements holds for a suitable choice of $i = \sqrt{-1}$. 

{\rm (i)} $m = 2$, and $g$ is a bireflection.

{\rm (ii)} $m = 3$, and one of the following holds, where $\omega = e^{2\pi i/3}$.

\hspace{5mm}{\rm (a)} $\mu g$ is conjugate to $\diag(\omega,\omega,1, \ldots ,1)$.

\hspace{5mm}{\rm (b)} $\age(g) = 1$, and $g$ is conjugate to $\diag(\omega,\omega^{2},1, \ldots ,1)$ or 
$\diag(\omega,\omega,\omega,1, \ldots ,1)$.  

{\rm (iii)} $m = 4$, and one of the following holds.

\hspace{5mm}{\rm (a)} $\mu g$ is conjugate to one of the elements 
$$\diag(i,i,1, \ldots ,1), ~~~\diag(i,i,i,1, \ldots ,1), 
  ~~~\diag(i,-1,1, \ldots ,1)$$

\hspace{5mm}{\rm (b)} $\age(g) = 1$, and $g$ is conjugate to one of the elements 
$$\diag(i,-i,1, \ldots ,1), ~~~\diag(-1,-1,1, \ldots ,1), 
  ~~~\diag(i,i,-1,1, \ldots ,1),~~~\diag(i,i,i,i,1, \ldots ,1)$$

{\rm (iv)} $m = 5$, and one of the following holds, where $\eps = e^{2\pi i/5}$.

\hspace{5mm}{\rm (a)} $\mu g$ is conjugate to one of the elements 
$$\begin{array}{c}\diag(\eps,\eps,1, \ldots ,1), ~~~\diag(\eps^{2},\eps^{2},1, \ldots ,1),\\ 
  \diag(\eps,\eps,\eps,1, \ldots ,1), ~~~\diag(\eps,\eps,\eps,\eps,1, \ldots ,1),\\
  \diag(\eps,\eps^{2},1, \ldots ,1), ~~~\diag(\eps,\eps^{3},1, \ldots ,1), 
  ~~~\diag(\eps,\eps,\eps^{2},1, \ldots ,1).\end{array}$$

\hspace{5mm}{\rm (b)} $\age(g) = 1$, and $g$ is conjugate to one of the elements 
$$\begin{array}{c}\diag(\eps,\eps^{4},1, \ldots ,1), ~~~\diag(\eps^{2},\eps^{3},1, \ldots ,1), 
  ~~~\diag(\eps,\eps^{2},\eps^{2},1, \ldots ,1),\\ 
  \diag(\eps,\eps,\eps^{3},1, \ldots ,1), ~~~\diag(\eps,\eps,\eps,\eps^{2},1, \ldots ,1), 
  ~~~\diag(\eps,\eps,\eps,\eps,\eps,1, \ldots ,1).\end{array}$$}
\end{lemma}

\begin{proof}
The proofs of all these statements are similar, and we display it for (iv). By the assumption,
there is some $t \in [0,1/5)$, and integers $a,b,c,d,e \geq 0$ such that $a+b+c+d+e = \dim(V) \geq 4$ and  
$$1 \geq \age(g) = at + b(t+1/5) + c(t+2/5) + d(t+3/5) + e(t+4/5);$$
in particular, $b + 2c + 3d +4e \leq 5$. Now the statement (iv) follows by an exhaustive enumeration.  
\end{proof}
 
\subsection{Character ratios}
We will need the following result of Gluck and Magaard, cf. \cite{G} and \cite{GM}.

\begin{propo}\label{ratio1}
{\sl Let $G$ be a finite group, $\chi \in \Irr(G)$ be of degree $> 1$, and 
$g \in G \setminus Z(G)$.

{\rm (i) \cite[Theorem 2.4]{GM}} Assume $G$ is a finite quasi-simple group, not 
$\AAA_{n}$ nor $2\AAA_{n}$ with $n \geq 10$. Then $|\chi(g)/\chi(1)| \leq 19/20$. 

{\rm (ii) \cite[Theorem 1.6]{GM}} Let $G = \SSS_{n}$ or $\AAA_{n}$ with $n \geq 5$, and let
$c(g)$ be the number of cycles of the permutation $g$. Then 
$|\chi(g)/\chi(1)| \leq 1/2 + c(g)/2n$.
\hfill $\Box$}
\end{propo}  

Next we address the character ratios for spin representations of $2\AAA_{n}$ and $2\SSS_{n}$.

\begin{lemma}\label{ratio2}
{\sl Let $G = 2\SSS_{n}$ or $2\AAA_{n}$ with $n \geq 6$, $\chi \in \Irr(G)$ a faithful 
character of $G$, and let $g \in G \setminus Z(G)$. Then $|\chi(g)/\chi(1)| \leq 7/8$.}
\end{lemma}  

\begin{proof}
Since $g \notin Z(G) = C_{G}(G')$ and since $G' = 2\AAA_{n}$ is generated by 
commutators $[x,y]$, with $x, y$ being inverse images of all $3$-cycles, there exists an inverse 
image $t$ of a $3$-cycle such that $h := [g,t] \notin Z(G)$. Observe that 
$h = gtg^{-1} \cdot t^{-1}$ projects onto the product of two $3$-cycles. It follows
that (a $G$-conjugate of) $h$ is contained in a natural subgroup $K \cong 2\AAA_{6}$ of $G$.
(See \cite[Lemma 2.5]{GM} for a similar argument.) 
Clearly, $h \notin Z(K)$ since $[g,t] \notin Z(G)$. Also, the restriction $\chi|_{K}$ is a 
sum of faithful irreducible characters of $K$. Inspecting \cite{Atlas}, one can check that
$|\chi(h)| \leq \chi(1)/2$, and so $\DB(h)^{2} \geq \chi(1)$. It now follows from Lemma 
\ref{basic3}(ii) that $\DB(g)^{2} \geq \chi(1)/4$, whence $|\chi(g)/\chi(1)| \leq 7/8$ by 
Lemma \ref{basic2} (iii).     
\end{proof}

\subsection{Tensor decomposable and tensor induced modules}
First we recall a well-known remark:

\begin{lemma}\label{tensor0}
{\sl Let $G$ be a finite irreducible subgroup of $GL(W)$. Assume that there is a tensor decomposition
$W = U \otimes V$ such that $G < GL(U) \otimes GL(V)$. Then there is a finite central
extension $1 \to Z \to \tilde{G} \to G \to 1$ of $G$ and irreducible representations 
$\Phi~:~\tilde{G} \to GL(U)$ and $\Psi~:~\tilde{G} \to GL(V)$ such that
$g = \Phi(\tilde{g}) \otimes \Psi(\tilde{g})$ for any $g = \tilde{g}Z \in G$.}
\end{lemma}

\begin{proof}
First we observe that if $a \otimes b = c \otimes d$ for some $a,c\in GL(U)$ and 
$b,d \in GL(V)$, then there is some $\gam \in \CC^{\times}$ such that 
$a = \gam c$ and $b = \gam^{-1}d$. Now, by the hypothesis, there are 
maps $A~:~G \to GL(U)$ and $B~:~G \to GL(V)$ such that 
$g = A(g) \otimes B(g)$ for any $g \in G$. If $h \in G$, then
$$A(gh) \otimes B(gh) = gh = (A(g) \otimes B(g)) \cdot (A(h) \otimes B(h)) = 
  (A(g) \cdot A(h)) \otimes (B(g) \cdot B(h)).$$
By our observation, we see that $A(gh) = \lam(g,h)A(g)A(h)$ for some $2$-cocycle
$\lam~:~G \times G \to \CC^{\times}$, and so $A$ is a projective (irreducible) 
representation of $G$. Thus $A$ lifts to a linear representation 
$\Phi~:~\tilde{G} \to GL(U)$ of a finite central extension $\tilde{G}$ of $G$:
$A(g) = \al(\tilde{g})\Phi(\tilde{g})$, where $\al~:~\tilde{G} \to \CC^{\times}$ and
$g = \tilde{g}Z$. Now it is easy to check that the map $\Psi~:~\tilde{G} \to GL(V)$, defined  
by $\Psi(\tilde{g}) = \al(\tilde{g})B(g)$ for $g = \tilde{g}Z$, is a group 
homomorphism, and $g = \Phi(\tilde{g}) \otimes \Psi(\tilde{g})$.  
\end{proof}

Lemma \ref{tensor0} shows that if a finite irreducible subgroup $G$ of $GL(V)$ preserves 
a tensor decomposition of $V$, then we may (and will) view $V$ as the tensor product of 
two modules for some central extension $\tilde{G}$ of $G$, and then replace $G$ by
$\tilde{G}$.

Let $V = \CC^{d}$ be a $G$-module, which is tensor induced. This means that there is a tensor decomposition
$V = V_{1}^{\otimes m}$ such that (the action of) $G$ (on $V$) is contained in 
$GL(V_{1})^{\otimes m}:\SSS_{m}$, with $\SSS_{m}$ naturally permuting the $m$ tensor factors
of $V$. (Note that we do not claim that $G \leq H^{\otimes m}:\SSS_{m}$ for some finite subgroup 
$H \in GL(V_{1})$.) 
   
\begin{lemma}\label{t-ind1}
{\sl Under the above assumptions, assume $G$ is finite and $g \in G$ projects onto $h \in \SSS_{m}$, a 
product of $s$ disjoint cycles. Then $|\Tr(g)| \leq \dim(V_{1})^{s}$.}
\end{lemma}

\begin{proof}
First we observe that if $y = a \otimes b$ has finite order for $a\in GL(U)$ and 
$b \in GL(V)$, then there is some $\delta \in \CC^{\times}$ such that both 
$c: = \delta^{-1}a$ and $d := \delta b$ have finite order, and $y = c \otimes d$. (Indeed,
$I = y^{N} = a^{N} \otimes b^{N}$, where we use $I$ to denote any identity matrix. 
So by the first sentence of the proof of Lemma \ref{tensor0},
$a^{N} = \gam I$ and $b^{N} = \gam^{-1}I$ for some $\gam \in \CC^{\times}$ and $0 < N \in \ZZ$. 
Now choose $\delta$ to be an $N^{\mathrm {th}}$-root of $\gam$.) 

In the case $s> 1$, conjugating $g$ with a suitable element in $\SSS_{m}$ we may assume that  
$g$ preserves a tensor decomposition of $V$. Using the 
above observation and proceeding by induction on $s$, we may assume that $s=1$ and 
$h = (1,2, \ldots ,m)$.
Now $g = hb$ with $b = B_{1} \otimes B_{2} \otimes \ldots \otimes B_{m}$ and 
$B_{i} \in GL(V_{1})$. Then one can check (see also \cite{GI}) that 
$\Tr(g) = \Tr(B_{1}B_{2} \ldots B_{m})$. Since $G$ and $\SSS_{m}$ are finite,
there is some integer $N > 1$ such that $g^{N} = h^{N} = \Id$. Since  
$I = h^{N} = g^{N} = h^{N}b^{h^{N-1}}b^{h^{N-2}} \ldots b^{h}b$ (where $b^{x} := x^{-1}bx$), 
we have 
$$I = g^{N} = (B_{2}B_{3} \ldots B_{m}B_{1})^{N/m} \otimes 
      (B_{3}B_{4}\ldots B_{m}B_{1}B_{2})^{N/m} \otimes 
      \ldots \otimes (B_{1}B_{2} \ldots B_{m})^{N/m}.$$
Pick an arbitrary eigenvalue $\lam$ of $v:=B_{1}B_{2} \ldots B_{m}$. Note that all 
the matrices $B_{2}B_{3} \ldots B_{m_{1}}B_{1}$, $B_{3}B_{4} \ldots B_{m_{1}}B_{1}B_{2}$, $\ldots$ are conjugate to 
$v$. Hence $\underbrace{\lam^{N/m} \ldots \lam^{N/m}}_{m} = \lam^{N}$ is an eigenvalue of
$g^{N} = I$. We have shown that each eigenvalue of 
$v$ is an $N^{\mathrm {th}}$-root of unity and so it has absolute value $1$.
Hence $|\Tr(g)| = |\Tr(v)|$ is at most the size of 
$v$, which is $\dim(V_{1})$.   
\end{proof}

We will also need the following technical statement:

\begin{lemma}\label{maxc}
{\sl Let $\CL$ be a collection of finite simple groups and let $G$ be any finite group. Then 
$G$ has a unique normal subgroup $R$ such that 

{\rm (i)} every composition factor of $R$ belongs to $\CL$; and

{\rm (ii)} If $N \lhd G$ and every composition factor of $N$ belongs to $\CL$, then $N \leq R$.\\
Furthermore, $R$ is a characteristic subgroup of $G$.}
\end{lemma}

\begin{proof}
Let $\XC$ be the collection of all normal subgroups $N \lhd G$ with the property that all 
composition factors of $N$ belong to $\CL$. For any $M, N \in \XC$,  
$MN \lhd G$, and every composition factor of $MN$ also belongs to $\CL$ since 
$MN/N \cong M/(M \cap N)$, whence $MN \in \XC$. Now the subgroup $R = \prod_{N \in \XC}N$ clearly 
satisfies (i) and (ii). Let $\varphi \in \Aut(G)$. Then $\varphi(R) \lhd G$ and 
$\varphi(R) \in \XC$ since $\varphi(R) \cong R$. By (ii), $\varphi(R) = R$.    
\end{proof}

\section{Proof of Theorem \ref{main-b}}

\subsection{Reduction to the almost quasi-simple case}

\begin{propo}\label{bound-r}
{\sl It suffices to prove Theorem \ref{main-b} for the case $G$ is an almost quasi-simple 
group which is irreducible, primitive, tensor indecomposable, and not tensor induced on $V$.}
\end{propo}

\begin{proof}
Let $\chi$ denote the character of $G$ afforded by $V$.

(i) First we consider the case $G$ is tensor induced on $V$: $V = V_{1} \otimes V_{2} \otimes \ldots \otimes V_{m}$, 
with $\dim(V_{i}) = a > 1$ and $G$ permutes the $m$ tensor factors $V_{1}, \ldots ,V_{m}$ (transitively). By assumption,
$g$ acts nontrivially on the set $\{V_{1}, \ldots ,V_{m}\}$. Hence, $|\chi(g)| \leq a^{m-1} \leq d/2$ by Lemma \ref{t-ind1}.
Now if $d = a^{m} \geq 8$ then $\Delta(g) \geq d/2 \geq 4$. On the other hand, if $d = a^{m} < 8$, then 
$d = 4$ and $\Delta(g) \geq d/2 = 2$.

(ii) Now assume that we are in the extraspecial case (i.e. the case (iii) of  \cite[Proposition 2.8]{GT3}). 
Then $d = p^{m}$ for some prime $p$ and some integer $m \geq 2$. By \cite[Lemma 2.4]{GT1}, 
$|\chi(g)| \leq p^{m-1/2} \leq d/\sqrt{2}$. In particular, if $d \geq 8$, then 
$\Delta(g) \geq d(1-1/\sqrt{2}) \geq 8-4\sqrt{2}$. If $d = 5$ or $7$, then 
$\Delta(g) \geq d-\sqrt{d} \geq 5-\sqrt{5} > 8-4\sqrt{2}$. If $d = 4$, then 
$\Delta(g) \geq 4(1-1/\sqrt{2}) = 4-2\sqrt{2}$. If $d = 3$, then 
$\Delta(g) \geq 3-\sqrt{3}$, and if $d = 2$, then 
$\Delta(g) \geq 2-\sqrt{2} > (3-\sqrt{5})/2$.

(iii) Next we consider the tensor decomposable case: $V = V_{1} \otimes \ldots \otimes V_{m}$, where
$G$ is tensor indecomposable and primitive on $V$, $\dim(V_{i}) \geq 2$, and $m \geq 2$.
By \cite[Proposition 2.8]{GT3} and by the hypothesis, we may assume that Theorem \ref{main-b} holds for 
$g$ acting on $V_{i}$ as long as $g|_{V_{i}}$ is not scalar. Let $\al_{i}$ be the character afforded by $V_{i}$. 

Suppose that there is some $j$ such that $\dim(V_{j}) \geq 3$ and $g|_{V_{j}}$ is non-scalar. Then
$\al_{j}(1)-|\al_{j}(g)| \geq 4-\sqrt{8}$ by Theorem \ref{main-b} applied to $(G,g,V_{j})$. Hence
$$\chi(1)-|\chi(g)| \geq \chi(1)-\frac{\chi(1)}{\al_{j}(1)} |\al_{j}(g)| =  
  \frac{\chi(1)}{\al_{j}(1)}(\al_{j}(1)-|\al_{j}(g)|) \geq 2(4-\sqrt{8}) = 8-4\sqrt{2},$$
as required. So we may assume that $\dim(V_{i}) = 2$ whenever $g|_{V_{i}}$ is non-scalar. But $g$ is non-scalar,
so without loss we may suppose that $\dim(V_{1}) = 2$ and $g|_{V_{1}}$ is non-scalar. By Theorem 
\ref{main-b} applied to $(G,g,V_{1})$ we have $\al_{1}(1)-|\al_{1}(g)| \geq (3-\sqrt{5})/2$. Arguing as above, 
we obtain that 
$$\Delta(g) = \chi(1)-|\chi(g)| \geq \frac{\chi(1)}{\al_{1}(1)}(\al_{1}(1)-|\al_{1}(g)|) 
  \geq d(3-\sqrt{5})/4.$$ 
If $d \geq 13$ in addition, then in fact $d \geq 14$ and $\Delta(g) \geq 7(3-\sqrt{5})/2 > 8-4\sqrt{2}$.
If $d = 6$ or $10$, then $m = 2$, $\dim(V_{2}) = 3$ or $5$, and so $g|_{V_{2}}$ is scalar, whence we arrive at the 
conclusion (v) of Theorem \ref{main-b}. The same holds if $m = 2$ and $d \in \{8,12\}$. 
We also arrive at the same conclusion when $d = 4$, as otherwise
$\Delta(g) \geq 4 - ((1+\sqrt{5})/2)^{2} > 4-\sqrt{8}$. Finally, consider the case where $d \in \{8,12\}$ but 
$m > 2$; that is, $m = 3$. Then we may assume that $\dim(V_{2}) = 2$ and $g|_{V_{2}}$ is not scalar (as otherwise
Theorem \ref{main-b}(v) holds). As in the case $d = 4$, we get 
$\al_{1}(1)\al_{2}(1)-|\al_{1}(g)\al_{2}(g)| > 4-\sqrt{8}$, whence 
$$\Delta(g) = \chi(1)-|\chi(g)| \geq \frac{\chi(1)}{\al_{1}(1)\al_{2}(1)}(\al_{1}(1)\al_{2}(1)-|\al_{1}(g)\al_{2}(g)|) 
  > 8-4\sqrt{2}.$$
We are done by \cite[Proposition 2.8]{GT3}.   
\end{proof}

Throughout the rest of this section we will assume that $G$ is an almost quasi-simple group. In fact, we will
prove more than we need for the proof of Theorem \ref{main-b}: we will describe all triples $(G,V,g)$, where 

$\stc:~$\begin{tabular}{l}$G < GL(V)$ is an almost quasi-simple, irreducible, primitive,\\ 
       tensor indecomposable subgroup, $g \in G \setminus Z(G)$, and \\ 
       either $0 < \ages(g) \leq 1$, or $\Delta(g) := \dim(V)-|\Tr(g)| \leq 8-4\sqrt{2}$.\end{tabular}
 
As usual, we denote by $\chi$ the character of $G$ afforded by $V$, $L := G^{(\infty)}$, $S := L/Z(L)$. The 
set-up $\stc$ implies that $\chi|_{L}$ is irreducible, and that 
\begin{equation}\label{forg}
  \Delta(g) = \chi(1)-|\chi(g)| < 4.556
\end{equation} 
by Corollary \ref{bound2}.

\subsection{Alternating groups}
First we dispose of the case $S = \AAA_{n}$ with $n \geq 8$.
For $1 \leq k \leq n-1$, let $R_{n}(k)$ denote the set of partitions $\lam \vdash n$, where either
$\lam$ or the conjugate partition $\lam^{*}$ has the form $(n-k,\mu)$ for some $\mu \vdash k$. 
We will need the following statement which follows from the main result of \cite{Ra}:

\begin{lemma}\label{ras}
{\sl Let $\rho = \rho^{\lam} \in \Irr(\SSS_{n})$ be labeled by the partition $\lam \vdash n$.

{\rm (i)} If $n \geq 15$, then either $\rho(1) \geq n(n-1)(n-5)/6$, or $\lam \in \cup^{2}_{k=1}R_{n}(k)$.

{\rm (ii)} If $n \geq 22$, then
either $\rho(1) \geq n(n-1)(n-2)(n-7)/24$, or $\lam \in \cup^{3}_{k=1}R_{n}(k)$.
\hfill $\Box$}
\end{lemma}

We will now estimate $\rho^{\lam}(g)$. Let $\tbf$ denote the transposition $(1,2) \in \SSS_{n}$.

\begin{lemma}\label{rhoS}
{\sl Let $1+\al(g)$ denote the number of fixed points of the permutation $g \in \SSS_{n}$,
and let $n \geq 9$. Then 
$$\rho^{\lam}(g) = \left\{ \begin{array}{ll}
  (\al(g)^{2}-\al(g^{2}))/2, & \lam = (n-2,1^{2}),\\
  (\al(g)^{2}+\al(g^{2}))/2 -\al(g) -1, & \lam = (n-2,2),\\
  (\al(g)^{3}-3\al(g)\al(g^{2})+2\al(g^{3}))/6, & \lam = (n-3,1^{3}),\\
  (\al(g)^{3}-\al(g^{3}))/3 - \al(g)^{2} +1, & \lam = (n-3,2,1),\\
  (\al(g)^{3}+3\al(g)\al(g^{2})+2\al(g^{3}))/6 - \al(g)^{2} -\al(g), & \lam = (n-3,3). \end{array} \right.$$
In particular, if $g \neq 1$ then $|\rho^{\lam}(g)| \leq \rho^{\lam}(\tbf)$ for any of the 
above $\lam$.}
\end{lemma}
  
\begin{proof}
It is well known that $\Sym^{2}(\al) = \rho^{(n-2,2)} + \rho^{(n-1,1)} + \rho^{(n)}$, 
$\wedge^{2}(\al) = \rho^{(n-2,1^{2})}$, and $\wedge^{3}(\al) = \rho^{(n-3,1^{3})}$, cf. \cite{FH} for instance.
Using the Littlewood-Richardson rule, one can see that 
$$\left( \Ind^{\SSS_{n}}_{\SSS_{n-1}}(\rho^{(n-1)})\right) \otimes \wedge^{2}(\al) = 
  \Ind^{\SSS_{n}}_{\SSS_{n-1}}\left((\rho^{(n-2,1^{2})})|_{\SSS_{n-1}}\right) = $$
$$= \Ind^{\SSS_{n}}_{\SSS_{n-1}}\left( \rho^{(n-3,1^{2})} + \rho^{(n-2,1)} \right) = 
  \rho^{(n-3,1^{3})} + \rho^{(n-3,2,1)} + 2\rho^{(n-2,1^{2})} + \rho^{(n-2,2)} + \rho^{(n-1,1)}$$
and so
$$\al \otimes \wedge^{2}(\al) = 
  \rho^{(n-3,1^{3})} + \rho^{(n-3,2,1)} + \rho^{(n-2,1^{2})} + \rho^{(n-2,2)} + \rho^{(n-1,1)}.$$  
Similarly,  
$$\al \otimes \rho^{(n-2,2)} = 
  \rho^{(n-3,3)} + \rho^{(n-3,2,1)} + \rho^{(n-2,1^{2})} + \rho^{(n-2,2)} + \rho^{(n-1,1)}.$$ 
It now follows that
$$\begin{array}{lll}
  \rho^{(n-3,2,1)} & = & \al \otimes \wedge^{2}(\al) -\wedge^{3}(\al) - \al \otimes \al + 1,\\
  \rho^{(n-3,3)} & = & \al \otimes \rho^{(n-2,2)} - \al \otimes \wedge^{2}(\al) + \wedge^{3}(\al),\end{array}$$ 
and we arrive at the above formulae for $\rho^{\lam}(g)$. 

Next assume that $g$ has exactly $k_{i}$ cycles of length $i$, $i = 1,2,\ldots$ in its decomposition 
into disjoint cycles. We will write $g = (1^{k_{1}}2^{k_{2}} \ldots )$ in this case. 
Then $\al(g) = k_{1}-1$, $\al(g^{2}) = k_{1}+2k_{2}-1$, and $\al(g^{3}) = k_{1}+3k_{3}-1$.
Let $\rho = \rho^{\lam}$ for short. We also assume that
$g \neq 1$ nor $g$ is a $2$-cycle; in particular, $-1 \leq \al(g) \leq n-4$ and
$-1 \leq \al(g^{2}) \leq n-1$.

Consider the case $\lam = (n-2,1^{2})$. Then $\rho(\tbf) = (n^{2}-7n+10)/2$, and 
$$1-n \leq -\al(g^{2}) \leq 2\rho(g) = \al(g)^{2}-\al(g^{2}) \leq (n-4)^{2} + 1 \leq 
  2\rho(\tbf).$$ 

Next assume that $\lam = (n-2,2)$. Then $\rho(\tbf) = (n^{2}-7n+12)/2$. Furthermore, 
$$-1-2n \leq 2\rho(g) = \al(g)(\al(g)-2)+\al(g^{2})-2 \leq (n-4)(n-6) + n-3 \leq 
  2\rho(\tbf).$$

Now we consider the case $\lam = (n-3,3)$. Then $\rho(\tbf) = (n-3)(n-4)(n-5)/6$, and 
$$\rho(g) = (k_{1}-1)(k_{1}-2)(k_{1}-3)/6 + (k_{1}-1)(k_{2}-1) + k_{3}.$$ 
The desired estimate is clear if $k_{1} = 0$. Assume that $k_{1}, k_{2} \geq 1$, in particular
$\rho(g) \geq 0$. Since $\rho(g)$ is increasing when we replace $(1^{k_{1}}2^{k_{2}}3^{k_{3}})$ by 
$(1^{k_{1}+k_{3}}2^{k_{2}+k_{3}})$, we may assume that $k_{3} = 0$. Also, since $\rho(g)$ is 
increasing when we replace $(1^{k_{1}}2^{k_{2}})$ by $(1^{k_{1}+2}2^{k_{2}-1})$ for 
$k_{2} \geq 2$, we may assume that $k_{2} = 1$.
It follows that $\rho(g)$ is maximized when $g$ is a $2$-cycle. Finally, let $k_{2} = 0$. 
Again the desired estimate is clear if $1 \leq k_{1} \leq 5$, so we may assume $k_{1} \geq 6$ 
and $k_{j} \geq 1$ for some $j \geq 3$; in particular, $\rho(g) \geq 0$. Notice that $\rho(g)$ 
increases when we replace 
a $j$-cycle by $(1^{j-3}3^{1})$ for $j \geq 4$, and when we replace $(1^{k_{1}}3^{k_{3}})$ by 
$(1^{k_{1}+3}3^{k_{2}-1})$ for $k_{3} \geq 2$. Hence 
$\rho(g) \leq \rho(3\mbox{-cycle}) \leq \rho(\tbf)$.
  
Next assume that $\lam = (n-3,1^{3})$. Then $\rho(\tbf) = (n-2)(n-3)(n-7)/6$. The desired estimate
is clear if $\al(g) \leq 1$ or if $n = 9$. On the other hand, if $2 \leq \al(g) \leq n-5$ and $n \geq 10$,
then 
$$6|\rho(g)| = |\al(g)^{3}-3\al(g)\al(g^{2})+2\al(g^{3})| \leq (n-5)^{3} + 3(n-5) + 2(n-1) \leq 6\rho(\tbf).$$
Also, if $\al(g) = n-4$ and $n \geq 10$, then 
$6\rho(g) = (n-4)^{2}(n-7) + 2(n-1) < 6\rho(\tbf)$.

Finally, we consider the case $\lam = (n-3,2,1)$. Then $\rho(\tbf) = (n-2)(n-4)(n-6)/3$, and 
$$3\rho(g) = (k_{1}-1)^{3}-3(k_{1}-1)^{2}-(k_{1}-1)-3(k_{3}-1).$$ The desired estimate 
is clear if $k_{1} \leq 4$, so we may assume $k_{1} \geq 5$. Observe that $\rho(g)$ increases when we 
replace a $j$-cycle by $(1^{j-2}2^{1})$ for $j \geq 4$, or if we replace a $3$-cycle by 
$(1^{1}2^{1})$ for $k_{3} \geq 1$, or if we replace a $2$-cycle by $(1^{2})$ for $k_{2} \geq 2$. It now 
readily follows that $|\rho(g)| \leq \rho(\tbf)$. 
\end{proof}

\begin{propo}\label{an}
{\sl Let $G$ be as in $\stc$ and $S = \AAA_{n}$ for some $n \geq 8$. Then $\chi(1) = n-1$, $L = \AAA_{n}$,
and $L$ acts on $V$ as on its deleted natural permutation module. Moreover, one of the following holds.

{\rm (i)} $\ages(g) = 1/2$, $\Delta(g) = 2$, and a scalar multiple of $g$ is a $2$-cycle, acting on $V$ as 
a reflection.

{\rm (ii)} $\ages(g) = 1$, $\Delta(g) = 3$ or $4$, and a scalar multiple of $g$ is a $3$-cycle, or a 
double transposition, both acting on $V$ as a (complex) bireflection.}
\end{propo}

\begin{proof}
1) First we consider the case $L = 2\AAA_{n}$. Since $\Aut(\AAA_{n}) = \SSS_{n}$ and $C_{G}(L/Z(L)) = Z(G)$, we may 
replace $G$ by $H \in \{2\AAA_{n},2\SSS_{n}\}$. By Lemma \ref{ratio2} and (\ref{forg}) we have 
$4.556 > \Delta(g) \geq \chi(1)/8$ and so $\chi(1) \leq 36$. It is well known (cf. e.g. \cite{KT}) that 
$\chi(1) \geq 2^{\lfloor n/2 \rfloor -1}$, hence $n \leq 13$. Now we can go through the irreducible spin characters
of $H$ for $8 \leq n \leq 13$ as listed in \cite{Atlas} and check that $\Delta(g)$ can be less than $4.556$ only
when $\chi(1) = 8$, $n = 8$ or $9$, and $\Delta(g) = 4$. However, in this exceptional case, $\ages(g) > 1$.  

2) Next we assume that $L = \AAA_{n}$ and moreover $\chi|_{L}$ is not the character of the deleted natural
permutation module. Again as above we may replace $G$ by $H \in \{\AAA_{n},\SSS_{n}\}$. By Proposition \ref{ratio1}(ii),
$|\chi(g)/\chi(1)| \leq 1-1/2n$, whence $4.556 > \Delta(g) \geq \chi(1)/2n$ and $\chi(1) < (9.112)n$. 
Also we choose $\lam \vdash n$ such that $\chi|_{L}$ is an irreducible constituent of 
$\rho^{\lam}|_{L}$. By our assumption, $\lam \notin R_{n}(1)$. 

Consider the case $n \geq 14$. Then by Lemma \ref{ras}(i) (and by \cite{GAP} for $n = 14$), either
$\rho^{\lam}(1) \geq n(n-1)(n-5)/6$, or $\lam \in R_{n}(2)$. Since $\chi(1) \geq \rho^{\lam}(1)/2$, in the former case
we would have $\chi(1) \geq (9.75)n$, a contradiction. Hence $\lam \in R_{n}(2)$; in particular,
$\chi|_{L} = \rho^{\lam}|_{L}$. But in this case, Lemma \ref{rhoS} and its proof imply that 
$\Delta(g) \geq \Delta(\tbf) \geq 2n-6 \geq 22$, again a contradiction. 

Finally, let $8 \leq n \leq 13$. An inspection of irreducible characters of $H$ \cite{Atlas} reveals that 
$\Delta(g) > 4.556$ in all cases. 

3) We have shown that $\chi(1) = n-1$ and $\chi|_{L}$ is the character of the deleted natural permutation module.
We may write $g = \al h$, where $h \in \SSS_{n}$ and $\al \in \CC^{\times}$. Then 
$|\chi(g)| = |\chi(h)| = |\mu(h)-1|$, where $\mu(h)$ is the number of points fixed by the permutation $h$. 
Since $\Delta(g) < 4.556$ and $n \geq 8$, we see that $n -2 \geq \mu(h) \geq n-4$. If $\mu(h) = n-2$, then
$h$ is a $2$-cycle, $\Delta(g) = 2$ and $\ages(g) = \age(h) = 1/2$.  If $\mu(h) = n-3$, then
$h$ is a $3$-cycle, $\Delta(g) = 3$ and $\ages(g) = \age(h) = 1$. If $\mu(h) = n-4$, then
$\Delta(g) = 4$ and $h$ is either a double transposition, or a $4$-cycle. In the former case 
$\ages(g) = \age(h) = 1$. In the latter case $\Delta(g) = 4$ and $\ages(g) > 1$ by Lemma \ref{small2}(iii).    
\end{proof}

From now on we may assume that $S \not\cong \AAA_{n}$ for any $n \geq 8$. By Lemma \ref{red2}(i) and 
(\ref{forg}), there is some $h \in L \setminus Z(L)$ with 
\begin{equation}\label{dim1}
  \Delta(h) \leq 4\Delta(g) < 18.224,
\end{equation} 
which implies by Proposition \ref{ratio1}(i) that $\chi(1)/20 < 18.224$ and so 
\begin{equation}\label{dim2}
  \chi(1) \leq 364.
\end{equation}

Let $\dl(S)$ denote the smallest degree of a projective complex irreducible representation of $S$. We will 
freely use the precise value of $\dl(S)$ as recorded in \cite{T}. 

\subsection{Classical groups}
To handle the finite classical groups, we will also need to estimate character ratios for their 
{\it Weil representations} (cf. \cite{TZ2}, \cite{GMST} and references therein, for definitions and detailed 
information on Weil representations). 

\begin{lemma}\label{r-slu}
{\sl Let $\chi$ be an irreducible complex Weil character of $L = SL_{n}(q)$ or $SU_{n}(q)$, $n \geq 3$, 
$(n,q) \neq (3,2)$, $(3,3)$, $(4,2)$, and let $g \in L \setminus Z(L)$. Then 
$$\frac{|\chi(g)|}{\chi(1)} < \frac{q^{n-1}+q^{2}}{q^{n}-q} \leq \frac{2}{3}.$$}
\end{lemma}

\begin{proof}
First we consider the case $L = SU_{n}(q)$ and let $\NC = \FF_{q^{2}}^{n}$ denote the 
natural module for $L$. Fix a primitive $(q+1)^{\mathrm {th}}$-root $\delta$
of unity in $\FF_{q^{2}}$, and let $d_{k}$ denote the dimension of the subspace 
$\Ker(g-\delta^{k}\cdot \Id)$ of $\NC$, for $0 \leq k \leq q$. Then the explicit formula for $\chi$
as given in \cite{TZ2} implies that 
$\chi(1) \geq (q^{n}-q)/(q+1)$ and $(q+1)|\chi(g)| \leq S := \sum^{q}_{k=0}q^{d_{k}}$. Clearly, 
$\sum^{q}_{k=0}d_{k} \leq n$ and $0 \leq d_{k} \leq n-1$. Without loss we may assume that 
$d_{1} = \max_{0 \leq k \leq q}d_{k}$. Now $S \leq (q+1)q^{n-3} < q^{n-1}$ if $d_{1} \leq n-3$,
$S \leq q^{n-2} +q^{2} +q-1 < q^{n-1}+q^{2}$ if $d_{1} = n-2$, and 
$S \leq q^{n-1}+2q-1 < q^{n-1}+q^{2}$ if $d_{1} = n-1$, and so we are done.   

Next, let $L = SL_{n}(q)$ and let $\NC = \FF_{q}^{n}$ denote the 
natural module for $L$. Fix a primitive $(q-1)^{\mathrm {th}}$-root $\eps$
of unity in $\FF_{q}$, and let $e_{k}$ denote the dimension of the subspace 
$\Ker(g-\eps^{k}\cdot \Id)$ of $\NC$, for $0 \leq k \leq q-2$. Then the explicit formula for $\chi$
as given in \cite{TZ2} implies that 
$\chi(1) \geq (q^{n}-q)/(q-1)$ and $(q-1)|\chi(g)| \leq R := \sum^{q-2}_{k=0}q^{d_{k}} +2q-2$. Clearly, 
$\sum^{q-2}_{k=0}e_{k} \leq n$ and $0 \leq e_{k} \leq n-1$. Without loss we may assume that 
$e_{1} = \max_{0 \leq k \leq q}e_{k}$. Now $R \leq (q-1)(q^{n-3}+2) < q^{n-1}+q^{2}$ if $e_{1} \leq n-3$,
$R \leq q^{n-2} +q^{2} +3q-5 < q^{n-1}+q^{2}$ if $e_{1} = n-2$, and 
$R \leq q^{n-1}+4q-5 < q^{n-1}+q^{2}$ if $e_{1} = n-1$, and so we are again done.   
\end{proof}  

\begin{lemma}\label{r-sp}
{\sl Let $\chi$ be an irreducible complex Weil character of $L = Sp_{2n}(q)$, $q$ odd, $n \geq 2$, 
and let $g \in L \setminus Z(L)$. Then 
$$\frac{|\chi(g)|}{\chi(1)} \leq \left\{ \begin{array}{ll}
  \dfrac{q^{n-1/2}+1}{q^{n}+1}, & q \equiv 1 (\mod 4)\mbox{ and }\pm g \mbox{ is a transvection},\\
  \dfrac{(q^{2n-1}+1)^{1/2}}{q^{n}-1}, & q \equiv 3 (\mod 4)\mbox{ and }\pm g \mbox{ is a transvection},\\
  \dfrac{2q^{n-1}}{q^{n}-1}, & \pm g \mbox{ is not a transvection}. \end{array} \right.$$
In particular, $|\chi(g)/\chi(1)| \leq 0.675$ unless $(n,q) = (3,3)$, $(2,3)$.}
\end{lemma}

\begin{proof}
Note that $\chi(1) = (q^{n}-\eps)/2$ for some $\eps = \pm 1$. If $\pm g$ is a 
transvection in $L$ then by \cite{TZ2}, 
$|\chi(g)| = (q^{n-1/2}-\eps)/2$ when $q \equiv 1 (\mod 4)$, and
$|\chi(g)| = \sqrt{q^{2n-1}+1}/2$ when $q \equiv 3 (\mod 4)$. Assume $\pm g$ is not a transvection, i.e.
the subspace $\Ker(g \pm \Id)$ on the natural module $\FF_{q}^{2n}$ of $L$ has dimension at most $2n-2$.
Also consider the {\it reducible Weil character} $\omega$ of $L$ (that has $\chi$ as one of its irreducible 
constituents), cf. \cite{GMST}. This character arises from the action of $L$ as an outer automorphism 
subgroup of the extraspecial $p$-group of order $p^{1+2nf}$ and exponent $p$, where $q = p^{f}$ and 
$p$ is prime. By \cite[Proposition 2.8]{GT3}, $|\omega(g)|, |\omega(-g)| \leq q^{n-1}$. One can write 
$\omega = \chi +\eta$ for another irreducible Weil character $\eta$ of $L$, and moreover, 
$|\omega(-g)| = |\chi(g) -\eta(g)|$. It follows that $|\chi(g)| \leq q^{n-1}$.  
\end{proof}

\subsubsection{$\boldsymbol{S = PSL_{n}(q)}$, $\boldsymbol{n \geq 3}$, 
$\boldsymbol{(n,q) \neq (3, q \leq 7)}$, $\boldsymbol{(4,3)}$, $\boldsymbol{(5,2)}$}\label{sln}

Under these assumptions, $\dl(S) = (q^{n}-q)/(q-1)$. Hence (\ref{dim2}) implies that 
$3 \leq n \leq 8$; moreover, $q = 2$ if $n = 7,8$, $q \leq 3$ if $n = 6$, $q \leq 4$ if $n = 5$, 
$q \leq 5$ if $n = 4$, and $q \leq 17$ if $n = 3$. In fact, if in addition $\chi|_{L}$ is a Weil representation,
then $18.224 > \Delta(h) \geq \chi(1)/3$ by Lemma \ref{r-slu}, 
and so instead of (\ref{dim2}) we have the much stronger upper
bound $\chi(1) \leq 54$. Now in the cases $(n,q) = (8,2)$, $(7,2)$, $(6,3)$, $(5,4)$, $(5,3)$, and
$(3, q \geq 8)$, the upper bound (\ref{dim2}) and \cite[Theorem 3.1]{TZ1} imply that 
$\chi|_{L}$ is indeed a Weil representation, of degree at least $72$, giving a contradiction.  
Also, the case $(n,q) = (4,2)$ has already been considered in Proposition \ref{an}.   
 
Assume $(n,q) = (6,2)$ or $(4,4)$. Then $L = SL_{n}(q)$, and its character table is available in \cite{GAP}. It is 
straightforward to check that there is no nontrivial $\chi \in \Irr(L)$ and $h \in L \setminus Z(L)$ with
$\Delta(h) < 18.224$ (notice that we need to check only the non-Weil characters of degree at most $364$). 

It remains to analyze the case $S = PSL_{4}(5)$. The character degrees of $SL_{4}(5)$ are listed by F. L\"ubeck
\cite{Lu2}. In particular, we see that all the non-trivial irreducible characters of $R := SL_{4}(5)$ have degree
$155$, $156$ (and they are Weil characters in these two cases), $248$ (and there are exactly two characters of
this degree), or at least $403$. Hence we may assume that $\chi(1) = 248$. 
An inspection of character degrees as listed in \cite{St} 
shows that $GL_{4}(5)$ has no irreducible characters of degree $248$. Thus
$\chi|_{R}$ is not stable under $GL_{4}(5)$. Since $\Out(R)$ is a dihedral group of order $8$ and $GL_{4}(5)$ 
induces the unique cyclic subgroup of order $4$ of $\Out(R)$, it follows that the inertia group of $\chi|_{R}$ in
$\Out(R)$ is an elementary abelian $2$-group. But $\chi|_{R}$ extends to $G$. Thus $G$ can induce only an 
elementary abelian $2$-subgroup of $\Out(R)$. We conclude that $g^{2} \in Z(G)L$. Notice that 
$\chi(1)-|\chi(v)| \geq \chi(1)/20 = 12.4$ for any $v \in L \setminus Z(L)$ by Proposition \ref{ratio1}(i). 
Together with (\ref{dim1}), this implies that $4\Delta(g) \geq \Delta(h) \geq 12.4$ and so $\Delta(g) \geq 3.1$. 
We will complete the case $S = PSL_{4}(5)$ by showing that $\ages(g) > 1$.

First we suppose that $g^{2} \notin Z(G)$. Then $\Delta(g^{2}) \geq 12.4$ as above, and so
$\ages(g^{2}) \geq 24.8/(2.9\pi) > 2.72$ by Proposition \ref{arc}(iii). It now follows by 
Lemma \ref{trivial}(v) that 
$\ages(g) > \ages(g^{2})/2 > 1.36$. Finally, assume that $g^{2} \in Z(G)$. Then $g$ acts on a suitable basis of $V$ 
via the matrix 
$\al \cdot \diag\left(\underbrace{1, \ldots ,1}_{k},\underbrace{-1, \ldots ,-1}_{l}\right)$ 
for some $\al \in \CC^{\times}$ and $1 \leq k,l < k+l = 248$. It is shown in 
\cite{GS} that $\an(gZ(G)) \leq 6$, whence $k,l \geq 42$ by Lemma \ref{fix}. It follows that 
$|\chi(g)| = |248-2l| \leq 164$, $\Delta(g) \geq 84$, and $\ages(g) \geq 168/(2.9\pi) > 18$ again
by Proposition \ref{arc}(iii).

\subsubsection{$\boldsymbol{S = PSU_{n}(q)}$, $\boldsymbol{n \geq 3}$, 
$\boldsymbol{(n,q) \neq (3, q \leq 8)}$, 
$\boldsymbol{(4,2)}$, $\boldsymbol{(4,3)}$, $\boldsymbol{(5,2)}$, $\boldsymbol{(6,2)}$}

Under these assumptions, $\dl(S) = (q^{n}-q)/(q+1)$ if $n$ is odd and $(q^{n}-1)/(q+1)$ if $2|n$. Hence 
(\ref{dim2}) implies that $3 \leq n \leq 10$; moreover, $q = 2$ if $7 \leq n \leq 10$, $q \leq 3$ if $n = 6$, 
$q \leq 4$ if $n = 5$, $q \leq 7$ if $n = 4$, and $q \leq 19$ if $n = 3$. As in \S\ref{sln}, if in addition 
$\chi|_{L}$ is a Weil representation, then instead of (\ref{dim2}) we have the much stronger upper
bound $\chi(1) \leq 54$ (in fact $\chi(1) \leq 39$ if $(n,q) = (7,2)$ or $(4,4)$). 
Now in the cases $(n,q) = (10,2)$, $(9,2)$, $(8,2)$, $(6,3)$, $(5,4)$, $(5,3)$, 
$(4,7)$, and $(3, q \geq 9)$, the upper bound (\ref{dim2}) and \cite[Theorem 4.1]{TZ1} imply that 
$\chi|_{L}$ is indeed a Weil representation, of degree at least $60$, giving a contradiction. The same argument 
applies to $(n,q) = (7,2)$ as the Weil representations of $SU_{7}(2)$ have degree at least $42$ and the non-Weil
representations have degree at least $860$.    
 
Assume $(n,q) = (4,4)$. Then $L = SU_{4}(4)$, and its character table is available in \cite{GAP}. It is 
straightforward to check that there is no nontrivial $\chi \in \Irr(L)$ and $h \in L \setminus Z(L)$ with
$\Delta(h) < 18.224$ (notice that we need to check only the non-Weil characters of degree at most $364$). 

It remains to analyze the case $S = PSU_{4}(5)$. The character degrees of $SU_{4}(5)$ are listed by F. L\"ubeck
\cite{Lu2}. In particular, we see that all the non-trivial irreducible characters of $SU_{4}(5)$ have degree
$104$, $105$ (and they are Weil characters in these two cases), $273$ (and there are exactly two characters of
this degree), or at least $378$. Hence we may assume that $\chi(1) = 273$. Checking the character table of 
$PSU_{4}(5)$ (available in \cite{GAP}), we see that it also has exactly two irreducible characters of degree
$273$. It follows that $L = S = PSU_{4}(5)$. Direct inspection of these two characters of $S$ reveals that 
$\Delta(h) \geq 250$, a contradiction.

\subsubsection{$\boldsymbol{S = PSp_{2n}(q)}$, $\boldsymbol{n \geq 2}$, $\boldsymbol{(n,q) \neq (2,q \leq 5)}$, 
$\boldsymbol{(3,2)}$, $\boldsymbol{(3,3)}$, $\boldsymbol{(4,2)}$}

Under these assumptions, $\dl(S) = (q^{n}-1)/2$ if $q$ is odd and $(q^{n}-1)(q^{n}-q)/2(q+1)$ if $2|q$. Hence 
(\ref{dim2}) implies that $2 \leq n \leq 6$; moreover, $q = 3$ if $n = 6$, $q \leq 3$ if $n = 5$,  
$q = 3,5$ if $n = 4$, $q = 5,7,9$ if $n = 3$; if $n = 2$ then either $q \leq 27$ and $q$ odd
or $q = 8$. Moreover, if in addition $q$ is odd and  
$\chi|_{L}$ is a Weil representation, then, by Lemma \ref{r-sp}, instead of (\ref{dim2}) we have the much 
stronger upper bound $\chi(1) \leq 56$ (in fact $\chi(1) \leq 29$ if $(n,q) = (2,7)$ or $(2,9)$). 
Now in the cases $(n,q) = (6,3)$, $(5,3)$, $(4,5)$, $(3,5)$, $(3,7)$, $(3,9)$, 
and $(2, q \geq 11)$, the upper bound (\ref{dim2}) and \cite[Theorem 5.2]{TZ1} imply that 
$\chi|_{L}$ is indeed a Weil representation, of degree at least $60$, giving a contradiction. 

Assume $(n,q) = (2,7)$ or $(2,9)$. The character table of $Sp_{2n}(q)$ is determined in \cite{Sr}. It is 
now straightforward to check that $\Delta(h) \geq 100$ if $\chi|_{L}$ is a non-Weil character of degree at most 
$364$. Moreover, the Weil characters of $Sp_{4}(9)$ have degree $40$ or $41$, larger than the bound $29$ 
mentioned above. On the other hand, when $(n,q) = (2,7)$, none of the Weil characters (of degree $24$ or
$25$) is fixed by an outer automorphism of $Sp_{4}(7)$. This implies that $G = Z(G)L$ and so we may assume
$g \in L$ in this case. Hence, if $\chi|_{L}$ is a Weil character, then 
$4.556 > \Delta(g) \geq (1-0.675)\chi(1)$ by Lemma \ref{r-sp}, and so $\chi(1) \leq 14$, a contradiction. 
The same argument excludes the Weil characters of $Sp_{8}(3)$; all other nontrivial irreducible characters of
$Sp_{8}(3)$ have degree at least $780$ by \cite[Theorem 5.2]{TZ1}, hence we are done in the case $(n,q) = (4,3)$.
If $(n,q) = (5,2)$, then $\Out(L) = 1$ and so we may assume that $g \in L$, whence 
$\Delta(g) \geq \chi(1)/20 \geq 7.75$ as $\dl(S) = 155$.
Finally, inspecting the character table of $Sp_{4}(8)$ (available in \cite{GAP}), we see that 
$\Delta(h) \geq 168$, again a contradiction when $(n,q) = (2,8)$.

\subsubsection{$\boldsymbol{S = P\Omega^{\eps}_{n}(q)}$, $\boldsymbol{n \geq 7}$, 
$\boldsymbol{(n,q) \neq (7,3)}$, $\boldsymbol{(8,2)}$, $\boldsymbol{(10,2)}$}
If $(n,q) \neq (8,3)$ in addition, then $\dl(S) \geq 620$ by \cite{TZ1}, and so we are done. Consider the
case $S = P\Omega^{\pm}_{8}(3)$. Notice that $Spin_{7}(3)$ embeds in $Spin^{\pm}_{8}(3)$ and any faithful 
irreducible character of $Spin_{7}(3)$ has degree at least $520$. Hence the bound (\ref{dim2}) implies that 
$L = S$ (this can also be deduced using the list of character degrees of $Spin^{\pm}_{8}(q)$ as given 
in \cite{Lu2}). Inspecting the character table of $P\Omega^{\pm}_{8}(3)$ (available in \cite{Atlas}), we see that 
$\Delta(h) \geq 189$, a contradiction. 

\subsubsection{$\boldsymbol{S = PSL_{2}(q)}$, $\boldsymbol{q \geq 37}$}
In these cases, $\chi(1) \geq (q-1)/\gcd(2,q-1)$ and $|\chi(h)| \leq (\sqrt{q}+1)/2$, cf. \cite{D}. In particular,
$|\chi(h)/\chi(1)| \leq 1/(\sqrt{q}-1) < 0.2$, and so (\ref{dim1}) implies that $\chi(1) \leq 22$. Since we 
are assuming $q \geq 37$, this in turns forces that $q = 37$, $41$, or $43$, and $\chi|_{L}$ is in fact 
a Weil character. But for these values of $q$, none of the Weil characters of $L$ is fixed by an outer automorphism
of $L$. Hence $G = Z(G)L$, and so we may assume that $g \in L$. Thus 
$|\chi(g)/\chi(1)| < 0.2$ as above, and $\Delta(g) > (0.8)\chi(1) \geq 14.4$, a contradiction.     

\subsection{Exceptional groups of Lie type}
Let $S$ be a simple exceptional group of Lie type. If $S$ is not isomorphic to 
$\ta B_{2}(q)$ with $q \leq 32$, $G_{2}(q)$ with $q \leq 7$, $\tb D_{4}(q)$ with $q \leq 3$, $\ta F_{4}(2)'$,
or $F_{4}(2)$, then $\dl(S) \geq 504$, see e.g. \cite{Lu1}. Consider the case $S = G_{2}(7)$. Then $L$ has a
unique nontrivial irreducible character of degree at most $364$ (namely $344$), and this character is labeled
as $\chi_{32}$ in the generic character table of $G_{2}(q)$ \cite{H}. One can now check that 
$\Delta(h) \geq 332$ for $\chi|_{L} = \chi_{32}$. Similarly, if $S = \tb D_{4}(3)$ then $L$ has a
unique nontrivial irreducible character of degree at most $364$ (namely $219$). This character is 
unipotent, and its values are computed in \cite{Sp}. In particular, one can check that 
$\Delta(h) \geq 195$ in this case.
   
\subsection{Small groups}\label{small}
The list of our ``small'' groups consists of all the finite simple groups not considered in the above subsections,
that is: $\AAA_{n}$ with $5 \leq n \leq 7$, $PSL_{2}(q)$ with $7 \leq q \leq 32$, $PSL_{3}(q)$ with 
$3 \leq q \leq 7$, $PSL_{4}(3)$, $SL_{5}(2)$, $PSU_{3}(q)$ with $3 \leq q \leq 8$, $SU_{4}(2)$, $PSU_{4}(3)$,
$SU_{5}(2)$, $PSU_{6}(2)$, $Sp_{4}(4)$, $PSp_{4}(5)$, $Sp_{6}(2)$, $PSp_{6}(3)$, $Sp_{8}(2)$,   
$\Omega_{7}(3)$, $\Omega^{\pm}_{8}(2)$,  $\Omega^{\pm}_{10}(2)$,
$\ta B_{2}(q)$ with $8 \leq q \leq 32$, $G_{2}(q)$ with $3 \leq q \leq 5$, $\tb D_{4}(2)$, $\tb F_{4}(2)'$,
$F_{4}(2)$, and $26$ sporadic simple groups. Notice that the character table of the universal cover of $S$
is known (see \cite{GAP}) in all these cases. 

Recall we are assuming that $L = G^{(\infty)}$ is quasi-simple, and $\chi \in \Irr(G)$ is irreducible over $L$; 
moreover, and $1 < \chi(1) \leq 364$ by (\ref{dim2}). The last condition excludes the cases 
$S \in \{J_{4}, Fi_{23},Fi_{24}',Ly,BM = F_{2},M = F_{1}\}$. We will use the character tables of the universal 
cover of $S$ as given in \cite{Atlas}, as well as the notation therein for the conjugacy classes in $G/Z(G)$.

\subsubsection{Sporadic groups}  One can check that

$\bullet$ $\Delta(g) \geq 6$ if $S = M_{22}$, $Suz$; 
 
$\bullet$ $\Delta(g) \geq 8$ if $S = M_{11}$, $M_{12}$, or if $S = J_{2}$ but $\chi(1) > 6$; and 

$\bullet$ $\Delta(g) \geq 12$ if $S = M_{23}$, $M_{24}$, $J_{1}$, $J_{3}$, $HS$, $McL$, $He$, $Ru$, $HN$,
$Fi_{22}$, $Co_{3}$, $Co_{2}$, $Co_{1}$, $O'N$, $Th$\\ 
for all $\chi$ satisfying the above hypotheses. 

\smallskip
Assume that $S = J_{2}$ and $\chi(1) = 6$; in particular, 
$L = 2 \cdot J_{2}$ and $G = Z(G)L$.  Then one can check that $\Delta(g) \geq 5-\sqrt{5} > 8-4\sqrt{2}$. Next, 
suppose that $0 < \ages(g) \leq 1$; in particular, $|\chi(g)| > 1.444$. Then, in the notation of \cite{Atlas}, 
we may assume that $\chi|_{L} = \chi_{22}$, and  
the class of $gZ(G)$ in $S$ is one of the following: $2A$, $3A$, $4A$, $5B$, $5C$, $10D$, and $15B$.
The first two cases lead to the row of $2 \cdot J_{2}$ in Table I. In the last two cases, in the notation of
Corollary \ref{bound2} we have $\delta \geq 6\pi/5$, but $\Delta(g) > 4.38$, whence $\ages(g) > 1$ by 
Corollary \ref{bound2}. In the case of class $4A$, a multiple of $g$ has spectrum 
$1,1,i,i,-i,-i$, and so $\ages(g) > 1$ by Lemma \ref{small2} (with $i = \sqrt{-1}$). 
Finally, in the case of classes
$5B$ and $5C$, none of the eigenvalue of $g$ occurs with multiplicity $\geq 3$, and so 
$\ages(g) > 1$ by Lemma \ref{small2}. 

\subsubsection{Small alternating groups: $S = \AAA_{n}$ with $5 \leq n \leq 7$} 
Arguing as in the proof of 
Proposition \ref{an} (and using Lemmas \ref{small1} and \ref{small2}), 
we may assume that $\chi|_{L}$ is not the character of the deleted natural permutation 
module. First we consider the case 
$S = \AAA_{5}$. Direct check using \cite{Atlas} shows that $\Delta(g) \geq (3-\sqrt{5})/2$ if $d = 2$ and 
$\Delta(g) \geq (5-\sqrt{5})/2 > 3-\sqrt{3}$ if $d = 3$. Assume $d = 4$ (and so $L = 2 \cdot \AAA_{5}$ by our 
assumptions). If $gZ(G)$ belongs to the class $5A$ or $5B$ of $G/Z(G)$, then $\Delta(g) = 3$ and 
$\ages(g) > 1$ by Lemma \ref{small2}(iv). If $gZ(G)$ belongs to the class $4A$, then $\Delta(g) = 4$ and 
$\ages(g) > 1$ by Lemma \ref{small1} (with $m = 2$). Similar arguments apply to the case $d = 5$. If $d = 6$,
then $\Delta(g) \geq 6-\sqrt{2} > 4.556$. 

\smallskip
Assume $n = 6$. Then the assumptions on $\chi$, $L$, and $\Delta(g)$ lead to one of the 
following three possibilities.

$\bullet$ $d = 3$, $L = 3\AAA_{6}$, and $\Delta(g) \geq (5-\sqrt{5})/2 > 3-\sqrt{3}$.

$\bullet$ $d = 4$, $L = 2 \cdot \AAA_{6}$, and $\Delta(g) \geq 2$. 
The classes $2A$, $2B$, $2C$, $3A$, $3B$, and $6B$ lead to three rows of Table I.
The other classes are excluded by Lemmas \ref{small1} and \ref{small2}. 

$\bullet$ $d = 6$, $L = 3 \cdot \AAA_{6}$, $\Delta(g) = 4$, and $gZ(G)$ belongs to the class 
$2A$, which leads to a row in Table I.  
 
\smallskip
Assume $S = \AAA_{7}$. Then the assumptions on $\chi$, $L$, and $\Delta(g)$ lead to one of the following 
two possibilities.

$\bullet$ $d = 4$, $L = 2 \cdot \AAA_{7}$, $G = Z(G)L$, and $\Delta(g) \geq 2$. 
The classes $2A$, $3A$, $3B$, and $7A$ lead to two rows of Table I. The other classes are excluded by Lemmas
\ref{small1} and \ref{small2}.

$\bullet$ $d = 6$, $L = 3 \cdot \AAA_{7}$, $G = Z(G)L$, $\Delta(g) = 4$, and $gZ(G)$ belongs to the classes 
$2A$ or $6A$. The former case leads to a row in Table I, and $\ages(g) > 1$ in the latter case by Lemma
\ref{small1} (with $m = 2$).    

\begin{lemma}\label{newRT1}
{\sl There are subgroups $G = C_{3} \times 2\AAA_{m} < GL_{4}(\CC)$ with $m = 6,7$ which give a basic non-RT pair
not of reflection type. This pair is of $AV$-type if $m = 6$.}
\end{lemma}

\begin{proof}
The faithful representation of $G$ on $V = \CC^{4}$ gives rise to a unique conjugacy 
class $g^{G}$ of non-central elements of age $< 1$, namely class $3A$ in $G/Z(G) \simeq \AAA_{m}$. Let 
$K := \langle g^{G} \rangle$. Then $Z(G)K = G$ by simplicity of $G/Z(G)$, but $g \notin [G,G] = 2\AAA_{m}$. 
It follows that $K = G$ and so $G$ gives a basic non-RT pair. Furthermore, $Z(GL(V))G$ does not contain any 
complex reflection, hence this pair is not of reflection type.
Finally, if $m = 6$ then the representation of $G = C_{3} \times SL_{2}(9) < C_{3} \times Sp_{4}(3)$ on $V$ can 
be written over $\QQ(\sqrt{-3})$ and so the corresponding basic non-RT pair is of AV-type. 
\end{proof}

\subsubsection{Small finite groups of Lie type} Let $S$ be any of the small simple finite 
groups of Lie type listed at the beginning of \S\ref{small}. Using \cite{Atlas}, it is 
straightforward to check that $\Delta(g) \geq 5$ for all characters $\chi$ satisfying the 
above hypotheses, except possibly for one of the following cases. (Note that, it suffices to 
consider only subgroups of $G$ that induce {\it cyclic extensions} of $S$ in $\Aut(S)$, since
$\langle g,L \rangle$ is such a subgroup.)

\smallskip
$\bullet$ $d = 8$, $L = 2 \cdot \Omega^{+}_{8}(2)$. Here, either $\Delta(g) = 2$, $gZ(G)$ belongs
to class $2F$, and $g$ acts as a reflection, or $\Delta(g) \geq 3$. In the latter case, either
we get the row of $\Omega^{+}_{8}(2)$ in Table I with complex bireflections of order $2$ and $3$,
or $\ages(g) > 1$ by Lemma \ref{small2} (when $gZ(G)$ has order $\leq 5$) and Corollary 
\ref{bound2} (with $\delta \geq 4\pi/3$). 

\smallskip
$\bullet$ $d = 7$ or $8$, and $S = Sp_{6}(2)$. If $d = 8$, then $\Delta(g) \geq 4$ and 
$\ages(g) > 1$. Assume $d = 7$. Then either $\Delta(g) = 2$, $gZ(G)$ belongs
to class $2F$, and $-g$ acts as a reflection, or $\Delta(g) \geq 3$. In the latter case, either
we get the row of $Sp_{6}(2)$ in Table I with complex bireflections of order $2$ and $3$,
or $\ages(g) > 1$ by Lemma \ref{small2}.  

\smallskip
$\bullet$ $d = 10$, $L = SU_{5}(2)$, $\Delta(g) = 4$, $gZ(G)$ belongs
to class $2A$, and $g$ acts as a bireflection.
   
\smallskip
$\bullet$ $d = 6$, $20$, or $21$, and $S = PSU_{4}(3)$. Assume $d = 6$. Then either 
$\Delta(g) = 2$, $gZ(G)$ belongs to class $2D$, and $g$ acts as a reflection, or 
$\Delta(g) \geq 3$. In the latter case, either
we get the row of $PSU_{4}(3)$ in Table I with complex bireflections of order $2$ and $3$
and an element with spectrum $(1,1,1,e^{2\pi/3},e^{2\pi/3},e^{2\pi/3})$,
or $\ages(g) > 1$ by Lemma \ref{small2} and Corollary \ref{bound2} (with $\delta \geq 4\pi/3$)
(and a direct check for some elements of order $6$). If $d = 21$ and $L = S$, then 
$\Delta(g) \geq 12$. In all the remaining cases, $\Delta(h) \geq 13$ for all 
$h \in L \setminus Z(G)$, and so $\Delta(g) \geq 13/4$ by Lemma \ref{red2}. We claim that we
also have $\ages(g) > 1$. Assume the contrary: $\ages(g) \leq 1$. Let $K$ be any subgroup of 
$G$ that contains $L$ and induces a subgroup $C_{2}$ of $\Out(S) = D_{8}$ while acting on
$L$. It is straighforward to check that, for any $h \in K \setminus Z(K)$, 
$\Delta(h) \geq 10$ and so $\ages(h) > 2.19$ by Corollary \ref{bound2}. Notice that 
$\Out(S) = D_{8}$ has exponent $4$. Hence, if $g^{2} \notin Z(G)$, we have 
$g^{2} \in K \setminus Z(K)$ for a subgroup $K$ of the aforementioned type, and so 
$\ages(g^{2}) > 2.19$ and $\ages(g) > 1.095$ by Lemma \ref{trivial}(v). 
Thus $g^{2} \in Z(G)$, and so modulo scalars we may assume that $g$ has two eigenvalues $1$, resp. 
$-1$, with multiplicity $m$, resp. $d-m$. By \cite{GS}, $\an(g) \leq 6$ and so $m,d-m \leq d-4$ by 
Lemma \ref{fix}. It follows that $|\chi(g)| = |2m-d| \leq d-8$, whence $\Delta(g) \geq 8$ and 
$\ages(g) > 1$. 

\smallskip
$\bullet$ $d = 4$, $5$, or $6$, and $S = SU_{4}(2) \simeq PSp_{4}(3)$. Assume $d = 6$. Then either 
$\Delta(g) = 2$, $gZ(G)$ belongs to class $2C$, and $g$ acts as a reflection, or 
$\Delta(g) \geq 3$. In the latter case, either we get a row with $(d,L) = (6,SU_{4}(2))$ in Table I,
or $\ages(g) > 1$ by Lemma \ref{small2} and Corollary \ref{bound2} (with $\delta \geq 4\pi/3$).
Assume $d = 5$. Then either 
$\Delta(g) = 2$, $gZ(G)$ belongs to class $2A$, and $-g$ acts as a reflection, or 
$\Delta(g) \geq 5 -\sqrt{7} > 8-4\sqrt{2}$. In the latter case, either
we get two rows with $(d,L) = (5,SU_{4}(2))$ in Table I,
or $\ages(g) > 1$ by Lemma \ref{small2} and Corollary \ref{bound2} (with $\delta \geq 4\pi/3$)
(and a direct check for some elements of order $6$). Finally, assume $d = 4$, and so
$L = Sp_{4}(3)$. This case by far has the most (twelve) classes of elements $g$ with 
$0 < \ages(g) \leq 1$ (leading to two rows in Table I), and is handled by a direct case-by-case 
argument. In this case we always have $\Delta(g) \geq 4-\sqrt{7}$. 

\smallskip
$\bullet$ $d = 6$ or $7$, and $S = SU_{3}(3)$. Here we have $\Delta(g) \geq 3$, and, aside from 
the entries with $(d,L) = (6,SU_{3}(3))$ and $(7,SU_{3}(3))$ in Table I, $\ages(g) > 1$ by 
Lemmas \ref{small1}, \ref{small2}, and Corollary \ref{bound2} (with $\delta \geq 4\pi/3$).

\smallskip
$\bullet$ $d = 6$ and $L = 6 \cdot PSL_{3}(4)$. Here, $\Delta(g) \geq 4$, and either 
we are in the row of $(d,L) = (6,6 \cdot PSL_{3}(4))$ in Table I, or $\ages(g) > 1$.

\smallskip
$\bullet$ $d = 6$ or $7$, and $S = PSL_{2}(13)$. If $d = 7$, then $\Delta(g) > 4.69$. If $d = 6$,
then either $\Delta(g) \geq 5$, or $\Delta(g) > 3.69$ and $\ages(g) > 1$ by Corollary \ref{bound2} 
(with $\delta = 16\pi/13$).  

\smallskip
$\bullet$ $d = 5$ or $6$, and $S = PSL_{2}(11)$. If $d = 6$,
then either $\Delta(g) \geq 5$, or $\Delta(g) > 4.26$ and $\ages(g) > 1$ by Corollary \ref{bound2} 
(with $\delta = 14\pi/11$). If $d = 5$, then $\Delta(g) \geq 5-\sqrt{3}$, and either we are in the 
row of $(d,L) = (5, PSL_{2}(11))$ in Table I, or $\ages(g) > 1$ (by direct calculation).

\smallskip
$\bullet$ $d = 3$, $4$, or $6$, and $S = PSL_{2}(7)$. If $d = 3$, then 
$\Delta(g) \geq 3-\sqrt{2}$. If $d = 6$, then either we are in the row
$(d,L) = (6,PSL_{2}(7))$ of Table I, or $\Delta(g) > 4$ and $\ages(g) > 1$
by Corollary \ref{bound2} (with $\delta \geq 5\pi/4$). Finally, if $d = 4$, then 
$\Delta(g) \geq 4-\sqrt{2}$, and either we arrive at the row $(d,L) = (4,SL_{2}(7))$ of Table I, or
$\ages(g) > 1$ (by a direct check). 

We have completed the proof of Theorem \ref{main-b}, as well as of the following result
which we will need later.

\begin{theor}\label{main-d}
{\sl Let $G < GL(V)$ be an almost quasi-simple subgroup such that the $G$-module $V$ is 
irreducible, primitive, and tensor indecomposable. Assume that $0 < \ages(g) \leq 1$ for some
$g \in G$, and that $d := \dim(V) \geq 4$. Then $(d, G^{(\infty)}, g, \Delta(g),\ages(g))$ 
is as listed in Table I.
$\hfill \Box$}
\end{theor}

In Table I, in the cases where $0 < \ages(g) < 1$, we indicate a minimal group $G$ 
containing $h$ with $\age(h) = \ages(g)$. We also list the conjugacy class of $gZ(G)$ in 
$G/Z(G)$ using the notation of \cite{Atlas}, for one representative of the $\Aut(L)$-conjugacy 
class of $\chi|_{L}$.
 
\newpage
\centerline
{{\sc Table} I. Non-central elements of small age in almost quasi-simple groups.}
\vspace{0.3cm}
\begin{tabular}{|c||c|c|c|c|c|} \hline \skipa
 $d$ & $L := G^{(\infty)}$ & $G$ & $g$ & $\Delta(g) = \DB(g)^{2}/2$ & $\ages(g)$ \\ 
       \skipa \hline \hline \skipa
 $4$ & $Sp_{4}(3)$ & $L \times 3$ & $\left\{ \begin{array}{c}3B \\3A, ~3C, ~6A\end{array}\right.$
       & $\left\{ \begin{array}{c}4-\sqrt{7}\\ 4-\sqrt{7}, ~2, ~4-\sqrt{3}\end{array}\right.$ & 
       $\left\{\begin{array}{c}1/3\\ 2/3\end{array}
        \right.$\\ \skipa
 $4$ & $2 \cdot \AAA_{6}$ & $L \times 3$ & $3A$ & $2$ & $2/3$\\
 $4$ & $2 \cdot \AAA_{7}$ & $L \times 3$ & $3A$ & $2$ & $2/3$\\
 $4$ & $2 \cdot \AAA_{5}$ & $L$ & $2A$, resp. $3A$ & $4$, resp. $3$ & $1$\\ 
 $4$ & $2 \cdot \AAA_{5}$ & $L \cdot 2$ & $2B$, resp. $6A$ & $4$, resp. $4-\sqrt{3}$ & $1$\\ 
 $4$ & $SL_{2}(7)$ & $L$ & $2A$, resp. $3A$, $7B$ & $4$, resp. $3$, $4-\sqrt{2}$ & $1$\\
 $4$ & $2 \cdot \AAA_{6}$ & $L$ & $2A$, resp. $3B$ & $4$, resp. $3$ & $1$\\ 
 $4$ & $2 \cdot \AAA_{6}$ & $L \cdot 2$ & $2B$, resp. $2C$, $6B$ & 
       $4$, resp. $4$, $4-\sqrt{3}$ & $1$\\
 $4$ & $2 \cdot \AAA_{7}$ & $L$ & $2A$, resp. $3B$, $7A$ & $4$, resp. $3$, $2$ & $1$\\ 
 $4$ & $Sp_{4}(3)$ & $L$ & 
       $\left\{ \begin{array}{c} 2AB,~3A,~3D\\ 4A,~6BE,~12A \end{array}\right.$ & 
       $\left\{ \begin{array}{c} 4,~4-\sqrt{7},~3\\2,~4-\sqrt{3},~3\end{array}\right.$ 
     & $1$\\ 
\hline 
 $5$ & $SU_{4}(2)$ & $L \times 2$ & $2A$ & $2$ & $1/2$\\
 $5$ & $SU_{4}(2)$ & $L \times 3$ & $3A$ & $5 - \sqrt{7}$ & $2/3$\\
 $5$ & $\AAA_{5}$ & $L$ & $2A$ & $4$ & $1$\\
 $5$ & $\AAA_{6}$ & $\SSS_{6}$ & $(12)(34)(56)$ & $4$ & $1$\\
 $5$ & $PSL_{2}(11)$ & $L$ & $2A$ & $4$ & $1$\\
 $5$ & $SU_{4}(2)$ & $L$ & $2B$, resp. $3D$ & $4$, resp. $3$ & $1$\\ 
\hline
 $6$ & $SU_{4}(2)$ & $L \cdot 2$ & $2C$ & $2$ & $1/2$\\
 $6$ & $6_{1} \cdot PSU_{4}(3)$ & $L \cdot 2_{2}$ & $2D$ & $2$ & $1/2$\\
 $6$ & $PSL_{2}(7)$ & $L$ & $2A$ & $4$ & $1$\\
 $6$ & $3 \cdot \AAA_{6}$ & $L$ & $2A$ & $4$ & $1$\\
 $6$ & $3 \cdot \AAA_{7}$ & $L$ & $2A$ & $4$ & $1$\\
 $6$ & $6 \cdot PSL_{3}(4)$ & $L$ & $2A$ & $4$ & $1$\\
 $6$ & $SU_{3}(3)$ & $L$ & $2A$, resp. $3A$ & $4$, resp. $3$ & $1$\\
 $6$ & $SU_{4}(2)$ & $L$ & $2AB$, resp. $3AB$, $3C$ & $4$, resp. $3$, $3$ & $1$\\
 $6$ & $6_{1} \cdot PSU_{4}(3)$ & $L$ & $2A$, resp. $3A$, $3B$ & $4$, resp. $3$, $3$ & $1$\\
 $6$ & $2 \cdot J_{2}$ & $L$ & $2A$, resp. $3B$ & $4$, resp. $3$ & $1$\\
\hline
 $7$ & $Sp_{6}(2)$ & $L \times 2$ & $2A$ & $2$ & $1/2$\\
 $7$ & $SU_{3}(3)$ & $L$ & $2A$ & $4$ & $1$\\
 $7$ & $Sp_{6}(2)$ & $L$ & $2C$, resp. $3A$ & $4$, resp. $3$ & $1$\\
\hline
 $8$ & $2 \cdot \Omega^{+}_{8}(2)$ & $L \cdot 2$ & $2F$ & $2$ & $1/2$\\
 $8$ & $2 \cdot \Omega^{+}_{8}(2)$ & $L$ & $2B$, resp. $3A$ & $4$, resp. $3$ & $1$\\
\hline
 $10$ & $SU_{5}(2)$ & $L$ & $2A$ & $4$ & $1$\\
\hline
 $n-1$ & $\AAA_{n}$ & $\SSS_{n}$ & $2$-cycle & $2$ & $1/2$\\ 
 $n-1$ & $\AAA_{n}$ & $L$ &  $(123)$, resp. $(12)(34)$ & $3$, resp. $4$ & $1$ \\
\hline 
\end{tabular}

\begin{remar}\label{types}
{\sl The spectra of elements $g$ with $0 < \ages(g) \leq 1$ that occur in Table I are listed as
follows:

\smallskip
{\rm (a1)} $(-1, 1, \ldots ,1)$ (a reflection, age $=1/2$);

{\rm (a2)} $(e^{2\pi i/3}, 1, 1, 1)$ (a complex reflection, age $=1/3$);

{\rm (a3)} $(e^{4\pi i/3}, 1, 1, 1)$ (a complex reflection, age $=2/3$);

{\rm (a4)} $(e^{2\pi i/3}, e^{2\pi i/3}, 1, \ldots ,1)$ (age $=2/3$);

{\rm (a5)} $(e^{\pi i/3}, -1, 1, 1)$ (age $=2/3$);

\smallskip
{\rm (b1)} $(-1, -1,1, \ldots ,1)$ (a bireflection, age $=1$);

{\rm (b2)} $(e^{2\pi i/3}, e^{4\pi i/3}, 1, \ldots ,1)$ (a complex bireflection, age $=1$);

{\rm (b3)} $(e^{2\pi i/3}, e^{2\pi i/3}, e^{2\pi i/3},1, \ldots ,1)$ (age $=1$);

{\rm (b4)} $(e^{\pi i/2}, e^{3\pi i/2}, 1, 1)$ (a complex bireflection, age $=1$);

{\rm (b5)} $(e^{\pi i/3}, e^{2\pi i/3},-1, 1)$ (age $=1$);

{\rm (b6)} $(e^{\pi i/3}, e^{\pi i/3},e^{4\pi i/3}, 1)$ (age $=1$);

{\rm (b7)} $(e^{2\pi i/7}, e^{4\pi i/7}, e^{8\pi i/7},1)$ (age $=1$).

{\rm (b8)} $(e^{\pi i/6}, e^{2\pi i/3}, e^{7\pi i/6},1)$ (age $=1$).\\
Additionally, the following spectra also occur for the groups of extraspecial type:

\smallskip
{\rm (c1)} $(e^{\pi i/2}, e^{\pi i/2},1,1)$ (age $=1/2$).

{\rm (c2)} $(e^{\pi i/4}, e^{\pi i/2},e^{5\pi i/4},1)$ (age $=1$).

{\rm (c3)} $(e^{\pi i/4}, e^{3\pi i/4},-1,1)$ (age $=1$).

{\rm (c4)} $(e^{2\pi i/5}, e^{4\pi i/5}, e^{4\pi i/5},1,1)$ (age $=1$).

{\rm (c5)} $(e^{2\pi i/5}, e^{2\pi i/5}, e^{6\pi i/5},1,1)$ (age $=1$).

{\rm (c6)} $(e^{\pi i/2}, e^{\pi i/2}, e^{\pi i/2},e^{\pi i/2},1,1,1,1)$ (age $=1$).
}
\end{remar}

\section{Linear groups generated by elements of bounded deviation}\label{Sdev}
In this section we will prove Theorem \ref{main1}. The standing hypothesis throughout the section is
that

\vspace{3mm} 
$\start~:~$\begin{tabular}{l}
$V = \CC^{d}$, $d > 1$, $G$ is a finite irreducible subgroup of $\GC = GL(V)$, $\XC \subset \GC$,\\ 

$C \geq 4$ a given constant, $Z(\GC)G = \langle \XC \rangle$, and $\DB(g)^{2} \leq C$ for all $g \in \XC$.
\end{tabular}

\vspace{3mm}
Let $\chi$ denote the character of $Z(\GC)G$ afforded by $V$. 

\subsection{The imprimitive case} 
Here we consider the case where $G$ (transitively) permutes the $m$ summands of a decomposition
$V = V_{1} \oplus \ldots \oplus V_{m}$, $m > 1$. For any $g \in \XC$, let $\mu(g)$ denote the number of
subspaces $V_{i}$ that are moved by $g$. Then $g$ fixes (setwise) precisely $m-\mu(g)$ 
subspaces $V_{i}$, whence $|\chi(g)| \leq (m-\mu(g))\dim(V_{1})$. It follows that 
\begin{equation}\label{4mu}
  C \geq \DB(g)^{2} = 2(\chi(1)-|\chi(g)|) \geq 2\mu(g)\dim(V_{1}).
\end{equation} 
Since $Z(\GC)G = \langle \XC \rangle$, there must be some $g \in \XC$ which acts nontrivially on 
the set $\{V_{1}, \ldots ,V_{m}\}$, for which $\mu(g) \geq 2$. Thus $\dim(V_{1}) \leq C/4$ (also see
Example \ref{impr1} for a partial converse).

Now we choose a $G$-invariant decomposition $V = V_{1} \oplus \ldots \oplus V_{m}$ with $m > 1$ 
smallest possible. This means that the induced action $\pi$ of $G$ on the set $\{V_{1}, \ldots ,V_{m}\}$
is primitive. Assume in addition that $(G,V)$ does not satisfy the conclusion (iii) of Theorem 
\ref{main1}, i.e. $\pi(G) \not\geq \AAA_{m}$. By \cite[Corollary 3]{LS}, $\mu(g) > 2(\sqrt{m}-1)$ 
(for some $g \in \XC$). On the other hand, as shown above, $1 \leq k := \dim(V_{1}) \leq C/4$. Now 
(\ref{4mu}) implies that $m < (1+C/4k)^{2}$ and so $\dim(V) < k(1+C/4k)^{2} =: h(k)$. Since the function
$h(k)$ is decreasing on $[1,C/4]$, we see that $\dim(V) < h(1) = (1+C/4)^{2}$. Thus we have shown that
if $G$ is as in $\start$ and $G$ is imprimitive, then either $G$ satisfies Theorem \ref{main1}(iii), or 
$\dim(V) < (1+C/4)^{2}$. Observe that $(1+C/4)^{2} < \fl(C)$ as $C \geq 4$. Hence we have proved 
Theorem \ref{main1} in the case $G$ is imprimitive.      

\begin{examp}\label{impr1}
{\sl Let $C \geq 4$, $V_{1} = \CC^{k}$ with $1 \leq k \leq C/4$, and let $H \leq GL(V_{1})$ be any subgroup generated by 
$\{ [x^{-1}hx,y] \mid x,y \in H\}$ for a fixed element $h \in H$ {\em (for instance, one can consider 
any quasisimple subgroup $H$ and any $h \in H \setminus Z(H)$)}. Then for any $m \geq 2$, the subgroup 
$G = H \wr \SSS_{m} < GL_{mk}(\CC)$ is generated by $g^{G}$ for some element $g \in G$ satisfying 
$\DB(g)^{2} \leq C$.} {\em Indeed, we may write $V = \CC^{mk}$ as the set of $m$-tuples 
$(v_{1}, \ldots ,v_{m})$, $v_{i} \in V_{1}$, and define $g$ via 
$g(v_{1}, v_{2}, v_{3}, \ldots ,v_{m}) = (v_{2}, h(v_{1}), v_{3}, \ldots ,v_{m})$; in particular,
$\Tr(g) = (m-2)k$ and so $\DB(g)^{2} = 2(\dim(V) - |\Tr(g)|) = 4k \leq C$. 
For any $x, y \in H$, observe that $G$ contains the elements 
$$\tilde{x}~:~(v_{1}, v_{2}, \ldots ,v_{m}) \mapsto (x(v_{1}), v_{2}, \ldots ,v_{m}),~~~
  \tilde{y}~:~(v_{1}, v_{2}, \ldots ,v_{m}) \mapsto (y(v_{1}), v_{2}, \ldots ,v_{m}).$$
Then $K := \langle g^{G} \rangle$ contains the element 
$[\tilde{x}^{-1}g^{2}\tilde{x},\tilde{y}] = (\tilde{x}^{-1}g\tilde{x})^{2} \cdot 
 (\tilde{y}\tilde{x}^{-1} \cdot g \cdot \tilde{x}\tilde{y}^{-1})^{-2}$ which acts on $V$ via 
$(v_{1}, v_{2}, \ldots ,v_{m}) \mapsto ([x^{-1}hx,y](v_{1}), v_{2}, \ldots ,v_{m})$. The assumption on $H$ now implies
that $K > H \times 1 \times \ldots \times 1$. But $g$ induces the transposition $(1,2)$ while acting on 
the $m$-tuples $(v_{1}, \ldots ,v_{m})$, $v_{i} \in V_{1}$. Hence $K > H^{m}$ and $K/H^{m} \cong \SSS_{m}$,
and so $K = G$ as stated.}     
\end{examp}

\subsection{Tensor decomposable case} 
Here we assume that $G$ is primitive but tensor decomposable on $V$: 
$\chi = \al_{1} \otimes \ldots \otimes \al_{n}$, where $n \geq 2$, $\al_{i} \in \Irr(G)$ is 
primitive and tensor indecomposable for each $i$, and 
$\al_{1}(1) \geq \ldots \geq \al_{n}(1) \geq 2$. Then we can find $g \in \XC$ 
such that $\DB(g)^{2} \leq C$ and $|\al_{n}(g)| < \al_{n}(1)$. In the case $\beta := \al_{n}$ is 
tensor induced, among such elements $g$ we can find one that acts nontrivially on the set of tensor
factors of $\beta$ (as otherwise $\beta|_{G}$ would be tensor decomposable). By Theorem \ref{main-b} 
applied to $(G,\beta,g)$, $\beta(1) - |\beta(g)| \geq \delta:= (3-\sqrt{5})/2$; moreover, 
if $\beta(1) \geq 13$ then $\beta(1) - |\beta(g)| \geq 2$. In the latter case,
$$C \geq 2(\al(1)\beta(1) - |\al(g)\beta(g)|) \geq 2\al(1)(\beta(1)-|\beta(g)|) \geq 4\al(1),$$
and so $\al(1) \leq C/4$, where $\al := \al_{1} \otimes \ldots \otimes \al_{n-1}$. By the choice of
$\al_{n}$, we get $\beta(1) \leq \al(1)$ and so $\chi(1) \leq C^{2}/16 < \fl(C)$. It remains to consider
the case $2 \leq \beta(1) \leq 12$. Then  
$$C \geq 2(\al(1)\beta(1) - |\al(g)\beta(g)|) \geq 2\al(1)(\beta(1)-|\beta(g)|) \geq 2\delta\al(1),$$
and so $\al(1) \leq C/2\delta$. Therefore, 
$$\chi(1) \leq 12 \cdot C/2\delta = 12C/(3-\sqrt{5}) < 16C < \fl(C),$$
and we are done.

\begin{examp}\label{optm1}
{\em Given any $C \geq 4$, choose $m = 1 + \lfloor C/4 \rfloor$. Let $\SSS_{m}$ act on $\CC^{m-1}$ as on its 
deleted natural permutation module. This induces a natural action of $G = \SSS_{m} \times \SSS_{m}$ on 
$V = \CC^{m-1} \otimes \CC^{m-1}$. Consider the element $g_{1} = (\tau,1)$ and $g_{2} = (1,\tau)$, where 
$\tau$ is the transposition $(1,2) \in \SSS_{m}$. Then $\chi(1) = (m-1)^{2}$, 
$\chi(g_{i}) = (m-1)(m-3)$ and so $\DB(g_{i})^{2} = 4(m-1)$. By the choice of $m$, we see that 
$\DB(g_{i})^{2} \leq C$ is very close to $C$ and $\dim(V)$ is very close to $C^{2}/16$. Clearly, 
$G$ satisfies the set-up $\start$ with $\XC := g_{1}^{G} \cup g_{2}^{G}$. (Adding to $G$ an involution
inverting the two factors $\CC^{m-1}$ of $V$, we then have $G = \langle g_{1}^{G} \rangle$.)}
\end{examp}

\subsection{Tensor induced case}
Consider the case $G$ is tensor induced on $V$: $V = V_{1} \otimes V_{2} \otimes \ldots \otimes V_{m}$, with
$\dim(V_{i}) = a > 1$ and $G$ permutes the $m$ tensor factors $V_{1}, \ldots ,V_{m}$ (transitively). Then we can
find $g \in \XC$ such that $\DB(g)^{2} \leq C$ and $g$ acts nontrivially on the set $\{V_{1}, \ldots ,V_{m}\}$. 
By Lemma \ref{t-ind1}, $|\chi(g)| \leq a^{m-1} \leq \chi(1)/2$. Hence $C \geq 2(\chi(1)-|\chi(g)|) \geq \chi(1)$.
  
\subsection{Extraspecial case}\label{Sdev-ext}
Here we consider the case (iii) of \cite[Proposition 2.8]{GT3}.
In this case, $\dim(V) = p^{m}$ for some prime $p$ and some integer $m \geq 2$. Pick any 
$g \in \XC \setminus Z(\GC)$. By \cite[Lemma 2.4]{GT1}, $|\chi(g)| \leq p^{m-1/2} \leq \chi(1)/\sqrt{2}$. Thus 
$C \geq 2(\chi(1)-|\chi(g)|) \geq \chi(1)(2-\sqrt{2})$ and so $\chi(1) \leq C(1+1/\sqrt{2})$. 

\subsection{Almost quasi-simple groups}
Here we consider the case $G$ is a finite almost quasi-simple group that satisfies $\start$. In particular,
$L := G^{(\infty)}$ is quasi-simple and $L/Z(L)$ is the unique non-abelian composition factor of $G$. By the above, 
we may assume in addition that $G$ is primitive and tensor indecomposable on $V$. Since $d > 1$ and
$Z(\GC)G = \langle \XC \rangle$, there exists $g \in G \setminus Z(G)$ with $\DB(g)^{2} \leq C$. By Lemma 
\ref{red2} and its proof, $L$ acts irreducibly on $V$ and there is $h \in L \setminus Z(L)$ such that 
$\DB(h)^{2} \leq 4C$. 

First assume that $L \not\cong \AAA_{n}$ for any $n \geq 10$. Then by Proposition \ref{ratio1}(i) and 
Lemma \ref{ratio2}, $|\chi(h)/\chi(1)| \leq 19/20$. It follows that 
$4C \geq \DB(h)^{2} = 2(\chi(1)-|\chi(h)|) \geq \chi(1)/10$, and so $\dim(V) \leq 40C \leq \fl(C)$.   

We may now assume that $L = \AAA_{n}$ for some $n \geq 10$, and moreover $V|_{L}$ is not isomorphic to 
the deleted permutation module of $L$ (as otherwise $(G,V)$ satifies Theorem \ref{main1}(ii)). 
Up to scalars we may also assume that $\AAA_{n} \leq G \leq \SSS_{n}$. By Proposition 
\ref{ratio1}(ii), $|\chi(g)/\chi(1)| \leq 1/2 + (n-1)/2n = 1-1/2n$, whence
\begin{equation}\label{forS}
  C \geq \DB(g)^{2} = 2(\chi(1)-|\chi(g)|) \geq \chi(1)/n.
\end{equation} 
In particular, if $n \leq 40$, then $\chi(1) \leq 40C \leq \fl(C)$. Henceforth we may 
assume that $n \geq 41$.     

Now we choose $\lam \vdash n$ such that $\chi|_{L}$ is an irreducible constituent of 
$\rho^{\lam}|_{L}$, where $\rho^{\lam} \in \Irr(\SSS_{n})$ is labeled by $\lam$, and apply 
Lemma \ref{ras} to $\rho^{\lam}$.
Assume we are in the former case of Lemma \ref{ras}(ii). Then 
$\chi(1) \geq \rho^{\lam}(1)/2 \geq n(n-1)(n-2)(n-7)/48 > 26n^{2}$ as $n \geq 41$. Together with
(\ref{forS}), we now have that $26n^{2} < \chi(1) \leq nC$. It follows that $n \leq C/26$ and so 
$\chi(1) < C^{2}/26 < \fl(C)$.    

It therefore remains to consider the case $\lam \in R_{n}(2) \cup R_{n}(3)$; in particular,
$\chi_{L} = \rho^{\lam}|_{L}$. 

Consider the case $\lam \in R_{n}(2)$. By Lemma \ref{rhoS} and its proof, 
$$C \geq 2(\chi(1)-|\chi(g)|) \geq 2(\chi(1) - |\chi(\tbf)|) \geq 4n-12 \geq 152,$$ 
whence $n \leq 3+C/4$. Thus $\chi(1) \leq (n-1)(n-2)/2 \leq (1+C/4)(2+C/4)/2 \leq \fl(C)$.

Finally, assume that $\lam \in R_{n}(3)$. By Lemma \ref{rhoS} and its proof, 
$$C \geq 2(\chi(1)-|\chi(g)|) \geq 2(\chi(1) - |\chi(\tbf)|) \geq 2(n-2)(n-5) > 68n,$$ 
whence $n < C/68$. Hence (\ref{forS}) implies that $\chi(1) \leq nC < C^{2}/68 < \fl(C)$.  
Thus we have proved Theorem \ref{main1} in the case $G$ is almost quasi-simple (and primitive, 
tensor indecomposable on $V$). By \cite[Proposition 2.8]{GT3}, we have therefore completed the proof
of Theorem \ref{main1}.

\section{Linear groups generated by elements of age $\leq 1$}

In this section we will address the following situation

\vspace{3mm} 
$\diam:~$\begin{tabular}{l}
$V = \CC^{d}$, $d > 1$, $G$ is a finite irreducible subgroup of $\GC = GL(V)$,\\ 
$\XC \subset \GC$, $Z(\GC)G = \langle \XC \rangle$, and $0 < \ages(g) \leq 1$ for all $g \in \XC$.
\end{tabular}

By Corollary \ref{bound2}, such groups $G$ satisfy the set-up $\start$ of \S\ref{Sdev} with 
$C := 9.111$. We will denote by $\chi$ the character of $Z(\GC)G$ afforded by
$V$, and frequently refer to intermediate results established in \S\ref{Sdev}.

\subsection{Imprimitive groups} 

First we record the following easy observation:

\begin{lemma}\label{impr2}
{\sl Let a finite subgroup $G < GL(V)$ preserve a decomposition 
$W = V_{1} \oplus \ldots \oplus V_{n}$ of a subspace $W \subseteq V$, with $\dim(V_{i}) = 1$ and 
$G$ inducing either $\SSS_{n}$ or $\AAA_{n}$ while permuting the $n$ subspaces 
$V_{1}, \ldots ,V_{n}$. Then for any element $g \in G$ with $\ages(g) \leq 1$, 
one of the following holds.

{\rm (i)} $g$ acts either trivially, or as a $2$-cycle on $\{V_{1}, \ldots ,V_{n}\}$.

{\rm (ii)} Some scalar multiple $\mu g$ of $g$ is a complex bireflection of order $2$ 
or $3$ on $V$, and $g$ acts as a $3$-cycle, or a double transposition on $\{V_{1}, \ldots ,V_{n}\}$.
Furthermore, $\ages(g) = 1$.}
\end{lemma}

\begin{proof}
Observe that, if $h$ acts as an $m$-cycle on $\{V_{1}, \ldots ,V_{m}\}$, then it has minimal 
polynomial $t^{m}-\al$, and $m$ eigenvalues $\beta e^{2\pi i j/m}$, $0 \leq j \leq m-1$, on
$V_{1} \oplus \ldots \oplus V_{m}$, for some $\al,\beta \in \CC^{\times}$. In particular, 
\begin{equation}\label{4c}
  \ages\left(h|_{V_{1} \oplus \ldots \oplus V_{m}}\right) \geq (m-1)/2.
\end{equation}
Now assume $g \in G$ has $\ages(g) \leq 1$ but $g$ does not satisfy (i). By Lemma \ref{trivial}(ii),
there is $\mu \in \SA$ such that $\age(\mu g) \leq 1$. By Corollary \ref{bound2},
$\DB(g)^{2} \leq (2.9)\pi$, and so $g$ cannot move more than $4$ subspaces $V_{i}$ by (\ref{4mu}). Thus
$g$ acts as a $3$-cycle, a double transposition, or a $4$-cycle on $\{V_{1}, \ldots ,V_{n}\}$.
In the third case, $\ages(g) > 1$ by (\ref{4c}). 
In the first two cases, (\ref{4c}), Lemma 
\ref{trivial}(iii), and the condition $\age(\mu g) \leq 1$ force $\mu g$ to act as a complex 
bireflection of order $3$, resp. $2$, on $W$, and trivially on a complement $U$ to $W$ in $V$.
The last claim in (ii) now follows from (\ref{4c}) and the assumption that $\ages(g) \leq 1$.   
\end{proof}

\begin{lemma}\label{impr3}
{\sl Let $G < GL(V)$ be as in $\diam$. Assume that $G$ preserves a decomposition 
$V = V_{1} \oplus \ldots \oplus V_{n}$, with $n > 1$ smallest possible. Let $\pi$ denote 
the induced permutation action of $Z(\GC)G$ on $\{V_{1}, \ldots ,V_{n}\}$. Then one of the following
holds:

{\rm (i)} $\dim(V_{i}) = 1$, and $(\pi(G),n) = (\SSS_{n},n)$, $(\AAA_{n},n)$,
$(ASL_{3}(2),8)$, $(SL_{3}(2),7)$, $(\AAA_{5},6)$, $(D_{10},5)$. If $g \in \XC$ and $\pi(g)$ is not
$1$ nor a $2$-cycle, then $g$ is a complex bireflection of order $2$ or $3$ and 
$\ages(g) = 1$.

{\rm (ii)} $\dim(V_{i}) = 2$, $\ages(g) = 1$ for any $g \in \XC$ with $\pi(g) \neq 1$, 
and the conclusion {\rm (iii)} of Theorem {\rm \ref{main2}} holds.}
\end{lemma}

\begin{proof}
Our assumptions imply that $\pi(G)$ is a primitive subgroup of $\SSS_{n}$, and that there must be
some $h \in \XC$ that moves $\mu(g) \geq 2$ subspaces $V_{i}$. By (\ref{4mu}), 
$2\mu(h)\dim(V_{1}) \leq 9.111$ and so $\mu(h)\dim(V_{1}) \leq 4$; in particular, 
$\dim(V_{1}) = 1$ or $2$.  
Consider the former case: $\dim(V_{1}) = 1$; in particular, $\mu(t) \leq 4$ for all
$t \in \XC$. If $\mu(t) = 2$, resp. $3$, for some $t \in \XC$, then $\pi(G)$ is a 
primitive permutation group containing a $2$-cycle, resp. a $3$-cycle, whence  
$\pi(G) = \SSS_{n}$ or $\AAA_{n}$ by \cite[Theorem 13.3]{W}. Otherwise $\pi(t)$ is a 
double transposition for all $t \in \XC$ with $\pi(t) \neq 1$. Thus $\pi(G)$ is a primitive 
subgroup of $\SSS_{n}$ generated by some double transpositions. Assume in addition that
$\pi(G) \not\geq \AAA_{n}$. Then $4 > 2(\sqrt{n}-1)$ by \cite[Corollary 3]{LS}, i.e. 
$n \leq 8$, and we arrive at the primitive permutation groups listed in (i). 
The second claim in (i) follows from Lemma \ref{impr2}(ii).

In the latter case, $\mu(h) = 2$, i.e. $h$ is a transposition. This conclusion in fact holds for 
any $g \in \XC$ with $\pi(g) \neq 1$. Thus $\pi(G)$ is a primitive permutation group generated by 
transpositions, and so $\pi(G) = \SSS_{n}$. Let $D := \Ker(\pi)$ and consider any 
$g \in G \setminus D$ with $\age(g) \leq 1$. Then we may assume that 
$g~:~V_{1} \leftrightarrow V_{2}$ and $g(V_{j}) = V_{j}$ for all $j \geq 3$. 
It is not difficult to see that $\Spec(g|_{V_{1} \oplus V_{2}})$ is the union of two cosets of 
$C_{2} = \langle -1 \rangle$ in $\SA$. By
Lemma \ref{small1} (with $m = 2$) and Lemma \ref{trivial}(iii), the condition $\age(g) \leq 1$
implies that $g^{2} = 1_{V}$, $\age(g) = 1$, $g$ is trivial on each $V_{j}$ with $j \geq 3$;
in particular, $g$ is a bireflection. This argument also shows that $\ages(t) = 1$ for all 
$t \in \XC$ with $\pi(t) \neq 1$. We will apply this observation to suitable inverse images 
(in $G$) of transpositions $(i,i+1)$, $1 \leq i \leq n-1$, of $\SSS_{n}$ to show that $G$ is a split 
extension of $D$ by $\SSS_{n}$. Indeed, denote the element $g$ we have just analyzed by $g_{1}$. 
For any $1 \leq i \leq n-1$, a $G$-conjugate $g_{i}$ of $g_{1}$ will project onto the transposition
$(i,i+1)$ and have age $=1$. Hence our observation (applied to $g_{i}$) yields that 
$$g_{i}~:~V_{i} \leftrightarrow V_{i+1},~~~g_{i}^{2} = 1_{V},~~~(g_{i})|_{V_{j}} = 1_{V_{j}}
  \mbox{ for }j \neq i,i+1.$$ 
Clearly, $(g_{i}g_{j})^{2} = 1_{V}$ if $|i-j| > 1$. Next, if $v \in V_{i}$ then
$$(g_{i}g_{i+1})^{3}~:~v \stackrel{g_{i+1}}{\longmapsto} v \stackrel{g_{i}}{\longmapsto} g_{i}(v) 
  \stackrel{g_{i+1}}{\longmapsto} g_{i+1}g_{i}(v) \stackrel{g_{i}}{\longmapsto} g_{i+1}g_{i}(v)
  \stackrel{g_{i+1}}{\longmapsto} g_{i}(v) \stackrel{g_{i}}{\longmapsto} v,$$
and similarly for all $v \in V_{j}$ with $j \neq i$, whence $(g_{i}g_{i+1})^{3} = 1_{V}$. Thus
$H = \langle g_{1}, \ldots ,g_{n-1}\rangle$ is a quotient of $\SSS_{n}$, and 
$DH = G = D \cdot \SSS_{n}$. It follows that $H \cong \SSS_{n}$. In fact, one can find a basis
$(u_{i},v_{i})$ of each $V_{i}$ such that $H$ acts via permuting the indices of the $u_{i}$'s, resp.
of the $v_{i}$'s:  
\begin{equation}\label{4sn}
  g_{i}~:~u_{i} \leftrightarrow u_{i+1},~v_{i} \leftrightarrow v_{i+1}, ~u_{j} \mapsto u_{j},~
    ~v_{j} \mapsto v_{j}, \mbox{ for }j \neq i,i+1.
\end{equation}     
\end{proof}

\begin{lemma}\label{impr4}
{\sl Let $G < GL(V)$ be a finite irreducible subgroup. Assume that $G$ preserves a decomposition 
$V = V_{1} \oplus \ldots \oplus V_{n}$, with $\dim(V_{i}) = 1$ and $G$ inducing either $\SSS_{n}$ or 
$\AAA_{n}$ while permuting the $n$ subspaces $V_{1}, \ldots ,V_{n}$, and that $n \geq 10$.
Then $G$ contains a complex bireflection of order $3$.}
\end{lemma}

\begin{proof}
We will represent elements of $G$ by their matrices with respect to a basis 
$(e_{1}, \ldots ,e_{n})$ with $V_{i} = \langle e_{i} \rangle_{\CC}$.   
Let $D$ be the normal subgroup of $G$ consisting of diagonal elements, so that 
$\SSS_{n} \geq G/D \rhd A := \AAA_{n}$. Notice that, as an $A$-module, every chief factor of $D$ is
either the trivial module $\IC$, or the heart $\HC$ of the natural permutation module, in 
characteristic $p$ for some prime $p$. It is well known that $H^{2}(A,\IC) = 0$ if $n \geq 8$ and 
$p > 2$. Furthermore, $H^{2}(A,\HC) = 0$ if $n \geq 10$ by the main result of \cite{KP}. It follows
that $O_{2'}(D) \cdot A$ splits over $O_{2'}(D)$. Since $D = O_{2}(D) \times O_{2'}(D)$, we may assume
that $G$ contains a subgroup $H = O_{2}(D)\cdot A$. 

Now inside $H$ we can find an inverse image $g$ of order $3$ of a $3$-cycle in $A$. We may assume
that 
$g = \diag\left(\begin{pmatrix} & a & \\ & & b\\c & &\end{pmatrix},d_{1}, \ldots ,d_{n-3}\right)$, 
where $abc = 1$ and $d_{i}^{3} = 1$. Since
$n-3 \geq 4$, we may also assume that $d_{4} = d_{5}$. Next, in $H$ we can find an element
$h = \diag\left(\begin{pmatrix} & x\\ y &\end{pmatrix},z,\begin{pmatrix} & u\\ v &\end{pmatrix},
 w_{1},\ldots,w_{n-5}\right)$. Then 
$[g,h] = \diag\left(\begin{pmatrix} & & a' \\ b' & & \\& c' &\end{pmatrix},1, \ldots ,1\right)$, 
where $a'b'c' \neq 0$; in particular, it permutes $V_{1}$, $V_{2}$, $V_{3}$ cyclically. Notice that,
$[g,h] \equiv g^{-1} (\mod O_{2}(D))$ and so $[g,h] \in O_{2}(D) \cdot C_{3}$, where 
$C_{3} = \langle g \rangle$. Hence, a suitable 
$2$-power $t$ of $[g,h]$ has order $3$, acts as a $3$-cycle on $\{V_{1},V_{2},V_{3}\}$, and fixes 
every $e_{i}$ with $i \geq 4$. Clearly, such a $t$ is conjugate to 
$\diag(e^{2\pi i/3},e^{4\pi i/3},1, \ldots, 1)$ and so it is a complex bireflection of order $3$.     
\end{proof}

\begin{lemma}\label{d-sn}
{\sl Let $\TC = GL_{1}(\CC)^{n}$ be a maximal torus of $\GC = GL_{n}(\CC)$, so that
$N := N_{\GC}(\TC) = \TC \cdot \SSS_{n}$, and $n \geq 5$. Assume $D < \TC$ is a finite 
subgroup which is normalized by the subgroup $\TC \cdot \AAA_{n}$ of $N$. Then $D \lhd N$.}
\end{lemma}

\begin{proof}
1) Without loss we may assume that $D$ is a $p$-group for some prime $p$. If 
$\exp(D) = q = p^{c}$, then $D$ is contained in $T := \{x \in \TC \mid x^{p^{c}} = 1\}$. Using
the additive notation, we may identify $T$ with the natural permutation 
$R\SSS_{n}$-module $M := \langle e_{1}, \ldots ,e_{n} \rangle_{R}$, where 
$R := \ZZ/q\ZZ$. It suffices
now to prove that any $\AAA_{n}$-submodule $N$ of $M$ is $\SSS_{n}$-invariant. 

2) Assume that ${p\not{|}n}$, with $n \geq 5$ or $(n,p) = (4,5)$, and consider the module  
$L := \{ \sum^{n}_{i=1}a_{i}e_{i} \mid a_{i} \in R,~\sum^{n}_{i=1}a_{i} = 0\}$. Let 
$\HC$ denote the heart of the natural permutation $\FF_{p}\SSS_{n}$-module. The 
condition on $(n,p)$ implies that $\HC$ is irreducible over $\AAA_{n}$. Now observe that 
the $\AAA_{n}$-module $L$ is uniserial, with $\HC$ as the 
unique composition factor. Hence any $\AAA_{n}$-submodule $Y$ of $L$ is 
$\SSS_{n}$-invariant. (Indeed, if $t \in \SSS_{n}$ then the 
$\AAA_{n}$-modules $Y$ and $tY$ have same composition length and so 
$Y = tY$ as $L$ is uniserial.)

3) Consider the case ${p\not{|}n}$. Then 
$M = A \oplus B$ as $\SSS_{n}$-modules, where 
$A = \langle \sum^{n}_{i=1}e_{i}\rangle_{R}$, and 
$B = \{ \sum^{n}_{i=1}a_{i}e_{i} \mid a_{i} \in R,~\sum^{n}_{i=1}a_{i} = 0\}$. By the result of
2) applied to the $\SSS_{n}$-module $L := B$, any $\AAA_{n}$-submodule $Y$ of $B$ is 
$\SSS_{n}$-invariant. On the other hand, $\SSS_{n}$ acts trivially on $A$, whence any 
$\AAA_{n}$-submodule $X$ of $A$ is obviously $\SSS_{n}$-invariant. 
Now set $X:= N \cap A$ and $Y := N \cap B$. Observe that any composition factor of the 
$\AAA_{n}$-module $N/(X \oplus Y)$ is a common composition factor of $M/A \cong B$ and 
$M/B \cong A$. Hence $N = X \oplus Y$ and so it is $\SSS_{n}$-invariant.    

4) Finally, we assume $p|n$ and consider the natural subgroups $\AAA_{n-1}$ and 
$\SSS_{n-1}$ in $\SSS_{n}$, which fix $e_{1}$. Then 
$M = A \oplus B$ as $\SSS_{n-1}$-modules, where 
$A = \langle e_{1},\sum^{n}_{i=2}e_{i}\rangle_{R}$, and 
$B = \{ \sum^{n}_{i=2}a_{i}e_{i} \mid a_{i} \in R,~\sum^{n}_{i=2}a_{i} = 0\}$. 
Then the conclusion of 2) applied to the $\SSS_{n-1}$-module $L := B$ implies that 
any $\AAA_{n-1}$-submodule $Y$ of $B$ is $\SSS_{n-1}$-invariant.
Also, $\SSS_{n-1}$ acts trivially on $A$, whence any $\AAA_{n-1}$-submodule $X$ of 
$A$ is obviously $\SSS_{n-1}$-invariant. 
Now set $X:= N \cap A$ and $Y := N \cap B$. As in 3), we see that $N = X \oplus Y$, and 
so it is $\SSS_{n-1}$-invariant. Thus 
$N$ is invariant under $\langle \AAA_{n},\SSS_{n-1} \rangle = \SSS_{n}$.     
\end{proof}

One of the main results of this subsection is the following

\begin{theor}\label{impr-a}
{\sl Let $G < \GC := GL(V)$ be a finite irreducible subgroup that preserves a decomposition 
$V = V_{1} \oplus \ldots \oplus V_{n}$, with $n > 1$ smallest possible. Assume in 
addition that $n \geq 3$ and $\dim(V_{i}) = 1$.

{\rm (i)} Assume $G$ satisfies $\diam$ and contains a non-central element $g$ with 
$0 < \ages(g) < 1$. Assume in addition that $\pi(G) \geq \AAA_{n}$, 
where $\pi$ denotes the permutation action of $G$ on $\{V_{1}, \ldots, V_{n}\}$. 
Then there is a finite subgroup $Z < Z(\GC)$ such that $ZG$ contains a complex reflection. 

{\rm (ii)} If $(G,V)$ is a basic non-RT pair, then there is a finite subgroup $Z < Z(\GC)$ and 
a complex reflection group $H = G(d,1,n)$ with $d > 1$ (in the notation of \cite{ST}) such that $ZG = ZH$.
Conversely, any $G(d,1,n)$ with $d > 1$ yields a basic non-RT pair.}
\end{theor}

\begin{proof}
Fix a basis vector $e_{i}$ for each $V_{i}$, and let $D := \Ker(\pi) \lhd G$ be consisting of 
all the elements of $G$ that act diagonally on the basis $(e_{1}, \ldots ,e_{n})$. 

\smallskip
1) Here we show that {\it if there is an element $g \in G \setminus D$ with $\age(g) < 1$, then
either 

{\rm (a)} $D$ contains a non-scalar element $h$ with $\age(h) < 1$, or 

{\rm (b)} $\lam g$ is a reflection for some $\lam = e^{-2\pi it}$, with 
$0 \leq t < 1/2n$ and $\lam^{2} \cdot 1_{V} \in G$.}
 
For, by Lemma \ref{impr2}, $g$ has the matrix 
$\diag\left( \begin{pmatrix}0 & a\\b & 0\end{pmatrix},c_{3}, \ldots ,c_{n}\right)$ in the given
basis, for some $a,b,c_{i} \in \CC^{\times}$. Then 
$\Spec(g) = \{\sqrt{ab},-\sqrt{ab},c_{3}, \ldots ,c_{n}\}$. Since $g$ has finite order, we may 
write $c_{j} = e^{2\pi ir_{j}}$ with $0 \leq r_{j} < 1$ for $j > 2$ and 
$\{\sqrt{ab},-\sqrt{ab}\} = \{e^{2\pi ir_{1}},e^{2\pi i(r_{1}+1/2)}\}$ with 
$0 \leq r_{1} < 1/2$. By our assumptions, $1 > \age(g) = 1/2 + 2r_{1} + \sum^{n}_{i=3}r_{i}$, and so
$2r_{1}+ \sum^{n}_{i=3}r_{i} < 1/2$. Observe that 
$g^{2} = \diag(ab,ab,c_{3}^{2}, \ldots ,c_{n}^{2})$ and 
$\age(g^{2}) \leq 4r_{1} + 2\sum^{n}_{i=3}r_{i} < 1$. Now if $g^{2}$ is non-scalar, then we can
set $h = g^{2}$. Assume $g^{2}$ is scalar; in particular, $ab = c_{i}^{2}$ for all $i > 2$ and
$c_{3}^{2} \cdot 1_{V} = g^{2} \in G$. 
Notice that $c_{3}$ has finite order in $\SA$ as $|g|$ is finite. Then 
$\Spec(c_{3}^{-1}g) =  \{1,-1,1, \pm 1, \ldots ,\pm 1\}$. By Lemma \ref{small2} (with $m = 2$),
the condition $\age(g) < 1$ now implies that 
$\Spec(c_{3}^{-1}g) =  \{1,-1,1, \ldots ,1\}$, and so $c_{3}^{-1}g$ is a reflection. Finally,
$\Spec(g) = \{ -c_{3},c_{3}, \ldots ,c_{3} \}$ and $\age(g) < 1$, so $c_{3} = e^{2\pi it}$ with
$0 \leq t < 1/2n$. Thus $\lam := c_{3}^{-1}$ has the properties specified in (b). 

\smallskip
2) Now we consider the situation of (ii). Then $G$ contains some non-central element $g$ 
with $\age(g) < 1$ such that $G = \langle g^{G} \rangle$. Since $G$ is irreducible and 
$n \geq 3$, $G \neq D$, and so $g \notin D$. Now we can apply the result of 1) to the element 
$g$. In the case $D \ni h$ with $h$ non-scalar and $\age(h) < 1$, we would have 
$G = \langle h^{G} \rangle \leq D$ (as $(G,V)$ is a basic non-RT pair), a contradiction. Hence
$\lam g$ is a reflection for some $\lam$ as specified in (b). Since 
$G = \langle g^{G} \rangle$, we see that 
$ZG = ZH$ for $Z = \langle \lam \cdot 1_{V} \rangle < Z(\GC)$ and 
$H = \langle (\lam g)^{G} \rangle$ is a finite group generated by reflections. Since the c.r.g. $H$
acts imprimitively on $V$ (inducing $\SSS_{n}$ on $\{V_{1}, \ldots ,V_{n}\}$), by \cite{ST} we must 
have that $H = G(de,e,n)$ for some positive integers $d,e$. Assume $e > 1$. Then $ZG$ contains
non-central elements $r := \diag(1, \ldots ,1,e^{2\pi i/e})$ and $\lam^{-1}r$, with 
$\age(r) = 1/e \leq 1/2$ and $\age(\lam^{-1}r) = 1/e + n/t < 1$. If $\lam \cdot 1_{V} \in G$,
then $G = ZG$ contains $r$. Otherwise, $G$ has index $2$ in $ZG$ (since $\lam^{2} \cdot 1_{V} \in G$)
and so either $r$ or $\lam^{-1}r$ belongs to $G$. In either case, we see that $D$ contains 
a non-central element $s \in \{r,\lam^{-1}r\}$ with $\age(s) < 1$ and   
$\langle s^{G} \rangle \leq D < G$, a contradiction. So $e = 1$. Also $d > 1$ as otherwise 
$H = G(1,1,n) = \SSS_{n}$ is reducible on $V$.

Conversely, we show that any c.r.g. $H = G(d,1,n) = D:\SSS_{n}$ with $d > 1$ yields a basic non-RT pair. 
Indeed, since $D < SL(V)$, for any non-central $x \in D$ we have $0 < \age(x) \in \ZZ$ and so 
$\age(x) \geq 1$. Now consider any non-central $y \in H$ with $\age(y) < 1$ (such elements exist, for 
instance, one can take any transposition in $\SSS_{n}$). By our observation and by Lemma \ref{impr2}, 
$y$ induces a transposition, say $(12)$, on $\{V_{1}, \ldots ,V_{n}\}$. We need to show that 
$K := \langle y^{H} \rangle$ coincides with $H$. It is clear that $KD = H$. Next, for  
$\delta := e^{2\pi i/d}$ we have $z := \diag(\delta,1,\delta^{-1},1, \ldots ,1) \in D$ and 
$K \ni yzy^{-1}z^{-1} = \diag(\delta^{-1},\delta,1, \ldots ,1)$. It is now easy to see that the set
of all $K$-conjugates of $yzy^{-1}z^{-1}$ generates $D$, and so $K \geq KD = H$.                 

\smallskip
3) From now on we will assume that {\it we are in the situation of {\rm (i)} but there is no
finite subgroup $Z < Z(\GC)$ such that $ZG$ contains a complex reflection.}  By Lemma 
\ref{trivial}(ii), there is $\mu \in \SA$ of finite order such that 
$\age(\mu g) = \ages(g) < 1$. Replacing $G$ by $\langle \mu \cdot 1_{V} \rangle \cdot G$ and 
$g$ by $\mu g$, we may (and will) assume that $\age(g) < 1$. By the conclusion of 1), we 
see that $D$ contains non-central elements $h$ with $\age(h) < 1$. By Lemma \ref{d-sn},
$D$ is normalized by the monomial subgroup $S \cong \SSS_{n}$ of $GL(V)$ 
(that act via permuting the basis vectors $e_{1}, \ldots ,e_{n}$). In what follows we will 
freely conjugate elements of $D$ by elements of $S$.   
  
Let $A = \{ x_{1} \mid \exists ~\diag(x_{1}, \ldots ) \in D\}$ be the finite subgroup of $\SA$ 
consisting of all the first diagonal entries of all the elements in $D$. Also, let 
$$\begin{array}{c}B = \{ x_{1}/x_{2} \mid \exists ~\diag(x_{1},x_{2}, \ldots) \in D\},\\ 
  \vspace{-3mm}\\
  C = \{ \diag(z_{1}, \ldots ,z_{n}) \mid z_{i} \in B,\prod^{n}_{i=1}z_{i} = 1\}.\end{array}$$ 
Observe that $C \leq D$. Indeed, if $x = \diag(x_{1},x_{2},x_{3}, \ldots ,x_{n}) \in D$, then 
some $S$-conjugate of $x$ equals $y = \diag(x_{2},x_{1},x_{3}, \ldots ,x_{n}) \in D$, and so
$D \ni xy^{-1} = \diag(\al,\al^{-1},1, \ldots,1)$ with $\al = x_{1}/x_{2} \in B$. Conjugating
$xy^{-1}$ suitably, we see that any diagonal matrix with spectrum 
$\{\al,\al^{-1},1, \ldots ,1\}$ (with counting multiplicities and 
$\al \in B$) belongs to $D$. Now any matrix 
$\diag(z_{1}, \ldots ,z_{n})$ with $z_{i} \in B$ and $\prod^{n}_{i=1}z_{i} = 1$, is the product of
$n-1$ diagonal matrices $\diag(z_{1},z_{1}^{-1},1, \ldots ,1)$, 
$\diag(1,z_{1}z_{2},(z_{1}z_{2})^{-1},1, \ldots ,1)$, 
$\ldots$, $\diag(1, \ldots ,1,z_{1} \ldots z_{n-1},(z_{1} \ldots z_{n-1})^{-1})$, all having
spectrum of indicated shape, and so belongs to $D$.  

\smallskip
4) Set $Z_{1} = \{z \cdot 1_{V} \mid z \in A\} < Z(\GC)$. Claim that $DZ_{1} = CZ_{1}$. Indeed,  
consider any $x = \diag(x_{1}, \ldots ,x_{n}) \in D$. 
Conjugating $x$ suitably, we see that $y_{i} := x_{i}/x_{n} \in B$ for $1 \leq i \leq n-1$ and 
$x_{n} \in A$. Now express $y := \diag(y_{1},y_{2}, \ldots ,y_{n-1},1) \in DZ_{1}$ as 
$y = t_{1}t_{2} \ldots t_{n-1}$, 
where 
$$\begin{array}{l}t_{1} = \diag(y_{1},y_{1}^{-1},1, \ldots ,1),~~~ 
  t_{2} = \diag(1,y_{1}y_{2},(y_{1}y_{2})^{-1},1, \ldots ,1), \ldots ,\\  
  t_{n-2} =  \diag(1, \ldots ,1,y_{1} \ldots y_{n-2},(y_{1} \ldots y_{n-2})^{-1},1),~~~
  t_{n-1} =  \diag(1, \ldots ,1,y_{1} \ldots y_{n-1},1).\end{array}$$ 
Notice that $t_{1}, \ldots ,t_{n-2} \in C$. If 
$y_{1} \ldots y_{n-1} \neq 1$, then obviously $t_{n-2} \in DZ_{1} < Z_{1}G$ is a complex reflection, a 
contradiction. It follows that $y_{1} \ldots y_{n-1} = 1$. Now
$y = t_{1} \ldots t_{n-2} \in C$ and $x = x_{n}y \in CZ_{1}$ for all 
$x \in D$, and so $DZ_{1} = CZ_{1}$, as stated. 

Let $|B| = b$. Denoting $\eps := e^{2\pi i/b}$, we have 
$v := \diag(\eps,\eps, \ldots ,\eps,\eps^{1-n}) \in C$ and 
$CZ_{1} \ni \eps^{-1}v  = \diag(1, \ldots ,1,\eps^{-n})$. In particular, if $\eps^{n} \neq 1$, then  
$\eps^{-1}v$ is a complex reflection, again a contradiction. Hence $\eps^{n}=1$ and so $b|n$. 
Now we turn our attention to the non-central element $h \in D$ with $\age(h) < 1$. Since 
$DZ_{1} = CZ_{1}$, we may write $h = \lam c$ for some $\lam \in \SA$ and $c \in C$. Clearly, 
$\ages(c) \leq \age(h) < 1$. By Lemma \ref{trivial}(ii) and its proof, there is some $\mu \in \SA$, 
where $\mu^{-1}$ is either $1$ or one of eigenvalues of $c$, 
such that $\age(\mu c) = \ages(c) < 1$. By the construction
of $C$, $\mu \in B$. Also, $\det(\mu c) = \mu^{n}\det(c) = 1$ as $b|n$. Now observe that for 
any element $u \in SL(V)$, $0 \leq \age(u) \in \ZZ$. Applying this observation to $\mu c$, we see that
$\age(\mu c) \in \ZZ$. Since $0 \leq \age(\mu c) < 1$, we must have that $\age(\mu c) = 0$, and so
$\mu c = 1_{V}$. Thus $h = \lam c$ is central, a contradiction.     
\end{proof}

\begin{examp}\label{impr5}
Let $2|n \geq 6$. We exhibit an example of a finite irreducible (imprimitive) subgroup
of $GL_{n}(\CC)$ which is generated by elements of $\age = 2/3$, but cannot be generated by 
complex reflections (up to scalars).
{\em First consider any $n \geq 5$. Pick a basis $(e_{1}, \ldots ,e_{n})$ of 
$V = \CC^{n}$ and consider 
$G = \langle y_{1}, x_{2},x_{3}, \ldots ,x_{n-2},z_{n-1} \rangle$, where
$$\begin{array}{lll}y_{1}& : & e_{1} \leftrightarrow e_{2},
    ~e_{j} \mapsto e_{j}\mbox{ for }3 \leq j \leq n-1,~e_{n} \mapsto e^{\pi i/3}e_{n},\\
  x_{i} & : & e_{i} \leftrightarrow e_{i+1},
    ~e_{j} \mapsto e_{j}\mbox{ for }j \neq i,i+1,n,~e_{n} \mapsto -e_{n},~~~\mbox{ for }
    1 \leq i \leq n-2,\\
  z_{n-1} & : &e_{n-1} \leftrightarrow e_{n},
    ~e_{j} \mapsto e_{j}\mbox{ for }1 \leq j \leq n-3,~e_{n-2} \mapsto -e_{n-2}. \end{array}$$ 
Also, consider the subgroup 
$G_{n} = \langle y_{1}^{3} = x_{1},x_{2}, \ldots ,x_{n-2},z_{n-1}\rangle$ of $G$. 
Clearly, both $G$ and $G_{n}$ induce $\SSS_{n}$ while permuting the $1$-spaces 
$\langle e_{1} \rangle, \ldots ,\langle e_{n} \rangle$. Next, 
$y_{1}^{2} = \diag(1, \ldots, 1, e^{2\pi i/3})$, and so $(y_{1}^{2})^{G}$ generates a normal
subgroup $E$ of order $3^{n}$ of $G$; furthermore, $G = E:G_{n}$. We claim that 
$G_{n}$ is an extension of 
$F = \{ \diag(a_{1}, \ldots ,a_{n}) \mid a_{i} = \pm 1, \prod^{n}_{i=1}a_{i} = 1\}$ 
by $\SSS_{n}$. Indeed,
it is clear that $G_{n} < SL(V)$, the normal subgroup $F_{1}$ of all diagonal elements
of $G_{n}$ is contained in $F$, and $G_{n}/F_{1} \simeq \SSS_{n}$. We will obtain the
claim, showing by induction on $n \geq 5$ 
that $F_{1} = F$. When $n = 5$, a direct check using \cite{GAP} shows that 
$|G_{5}| = 2^{4} \cdot |\SSS_{5}|$ and so $F_{1} = F$. For the induction step, 
$\langle x_{2}, \ldots ,x_{n-2},z_{n-1} \rangle$ fixes $e_{1}$ and plays the role of 
$G_{n-1}$ while acting on $\langle e_{2}, \ldots ,e_{n}\rangle$. By the induction hypothesis,
$G_{n} \ni f := \diag(1,1, \ldots ,1,-1,-1)$, whence $F_{1} = F$.   

Next we show that $K := \langle (y_{1})^{G} \rangle$ equals $G$, and so $G$ is generated 
by elements of $\age = 2/3$. Clearly, $K$ induce $\SSS_{n}$ while permuting the $1$-spaces 
$\langle e_{1} \rangle, \ldots ,\langle e_{n} \rangle$. Also, $K \ni y_{1}^{2}$, and so 
$K > E$. Observe that $f = y_{1}^{3} \cdot z_{n-1}y_{1}^{3}z_{n-1}^{-1} \in K$, whence $K > F$
and so $K = EG_{n} = G$. 

Finally, assuming $2|n$, we show that any complex reflection in $Z(GL(V))G$ is diagonal,
and so $G$ cannot be generated by complex reflections (up to scalars). Assume the contrary:
there is some $t \in G$ such that $\Spec(t) = \{\gam,\delta, \ldots ,\delta\}$ with 
$\gam \neq \delta$ and $t$ is not diagonal. Since $\ages(t) < 1$, by Lemma \ref{impr2} we
may assume that $t \equiv x_{1} (\mod EF)$, i.e. $t = x_{1}u$, with 
$u = \diag(u_{1}, \ldots ,u_{n})$, $u_{j} = \eps^{m_{j}}$ for some $m_{j} \in \ZZ$ and 
$\eps := e^{\pi i/3}$, and $\sum^{n}_{j=1}m_{j} \in 2\ZZ$. Since 
$\Spec(t) = \{\sqrt{u_{1}u_{2}},-\sqrt{u_{1}u_{2}},u_{3}, \ldots ,u_{n-1},-u_{n}\}$, we must 
have $-\gam = \delta = \eps^{k}$ for some $k \in \ZZ$. Now
$$-1 = (-\delta^{n})^{3} = (\det(t))^{3} = (\det(u))^{3} = 
  (\prod^{n}_{j=1}u_{j})^{3} = \eps^{3\sum^{n}_{j=1}m_{j}} = 1$$
(since $2|n$), a contradiction.}
\end{examp}

Finally, we prove an analogue of Theorem \ref{impr-a}(ii) for $\age \leq 1$:

\begin{theor}\label{impr-b}
{\sl Let $G < GL(V)$ be a finite imprimitive, irreducible subgroup. Assume $G$ contains 
non-central elements $g$ with $\age(g) \leq 1$, and  
$ZG = Z \cdot \langle g^{G}\rangle$ for any such element $g$, where  
$Z := Z(GL(V))$. Then $\dim(V) \leq 8$.}
\end{theor}

\begin{proof}
Assume the contrary: $\dim(V) \geq 9$ for such a group $G$. Clearly, $G$ satisfies the 
set-up $\diam$. Hence $G$ satisfies one of the conclusions (i) and (ii) of Lemma 
\ref{impr3}.

\smallskip
1) Suppose the conclusion (i) of Lemma \ref{impr3} holds. Since $n = \dim(V) \geq 9$, 
$\pi(G) \geq \AAA_{n}$. Let $D = \Ker(\pi)$ be the subgroup of all diagonal elements of
$G$, in a basis $(e_{1}, \ldots ,e_{n})$ such that $V_{i} = \langle e_{i}\rangle$. 

First we consider the case $D \not\leq Z$. Then we may assume that 
$D \ni x = \diag(x_{1}, \ldots ,x_{n})$ with $x_{1} \neq x_{2}$. Choosing $s \in G$ with 
$\pi(s) = (123)$, we get 
$$G \ni y = [s,x] = \diag(x_{3}/x_{1},x_{1}/x_{2},x_{2}/x_{3},1, \ldots ,1) = 
  \diag(e^{2\pi i a},e^{2\pi i b},e^{2\pi i c},1, \ldots 1),$$
where $0 \leq a,c < 1$, $0 < b < 1$, and $a+b+c \in \ZZ$. It follows that 
either $a+b+c = 1$, in which case $\age(y) = 1$, or $a+b+c = 2$, in which case 
$\age(y^{-1}) = 1$. In either case, we have found a diagonal non-central element 
$z$ with $\age(z) = 1$. It is clear that $Z \langle z^{G} \rangle$ is diagonal and so
cannot contain $G$, a contradiction.   

We have shown that $D \leq Z$ and so $V$ yields an irreducible projective representation
of degree $n \geq 9$ of $\SSS_{n}$ or $\AAA_{n}$, which is impossible by degree consideration.

\smallskip
2) Now we assume that the conclusion (ii) of Lemma \ref{impr3} holds: $G = D:\SSS_{n}$, with 
$n \geq 3$, $D < GL_{2}(\CC)^{n}$ and the action of $\SSS_{n}$ described in (\ref{4sn}) 
for a fixed basis $(u_{i},v_{i})$ of each $V_{i}$. Let 
$A = \{ x_{1} \mid \exists ~\diag(x_{1}, \ldots ) \in D\}$ be 
the finite subgroup of $GL_{2}(\CC)$ afforded by the action of $D$ on $V_{1}$, with respect to
the basis $(u_{1},v_{1})$. Also, let 
$$\begin{array}{c}B = \{ x_{1}x_{2}^{-1} \mid \exists ~\diag(x_{1},x_{2}, \ldots) \in D\},\\ 
  \vspace{-3mm}\\
  C = \{ \diag(z_{1}, \ldots ,z_{n}) \mid z_{i} \in B,\prod^{n}_{i=1}z_{i} = I\},\end{array}$$
where $I$ denotes the identity $2 \times 2$-matrix.
Note that, by their definition, $B$ and $C$ are {\it finite sets}. 
Consider any $a \in B$. Then we can find $x = \diag(x_{1},x_{2},x_{3}, \ldots ,x_{n}) \in D$
with $a = x_{1}x_{2}^{-1}$, and 
some conjugate $y = \diag(x_{2},x_{1},x_{3}, \ldots ,x_{n}) \in D$ of $x$, whence
$D \ni xy^{-1} = \al := \diag(a,a^{-1},I, \ldots,I)$. 
Conjugating $\al$ suitably, we see that any matrix 
$\diag(I, \ldots, I,a,a^{-1},I, \ldots ,I)$ belongs to $D$.  
Similarly,
if $b \in B$, then $\beta := \diag(b,b^{-1},I, \ldots ,I) \in D$, and  
$D \ni \al\beta = \diag(ab,a^{-1}b^{-1},I, \ldots,I)$. Conjugating the latter element
suitably, we see that $D \ni \diag(ab,I,a^{-1}b^{-1},I, \ldots,I)$ and so $ab \in B$. Thus  
$B$ is closed under multiplication and so it is a group by finiteness.
By the above observation applied to $ab$, $\gam = \diag(ab,(ab)^{-1},I, \ldots,I)$ belongs to
$D$, and so does $\delta := \gam^{-1}\al\beta = \diag(I,[a,b],I, \ldots,I)$.
Note that, since $[a,b] \in SL_{2}(\CC)$ (and has finite order $|\delta|$), 
we have either $\age([a,b]) = 1$, or $[a,b] = I$.
In the former case, $\age(\delta) = 1$ and $Z\langle \delta^{G} \rangle \leq ZD \not\geq G$, 
a contradiction.
Hence, $[a,b] = I$ for all $a,b \in B$, i.e. $B$ is an abelian group. This in turn implies
that $C$ is a subgroup of $D$. 

Observe that $A$ normalizes $B$. (Indeed, for any $x_{1} \in A$ and $b \in B$, there is some 
$x = \diag(x_{1},x_{2}, \ldots ,x_{n}) \in D$ and $u = \diag(b,I,b^{-1},I, \ldots ,I) \in D$. 
Hence $D \ni xu = \diag(x_{1}b,x_{2}, \ldots )$ and so 
$B \ni x_{1}b(x_{2})^{-1} = x_{1}bx_{1}^{-1} \cdot x_{1}x_{2}^{-1}$. But 
$x_{1}x_{2}^{-1} \in B$, hence $x_{1}bx_{1}^{-1} \in B$ as stated.) Since 
$D \leq A \times A \times \ldots \times A$ and $G = D:\SSS_{n}$, we see that $C \lhd G$. 
In fact, we claim that $[\SSS_{n},D] \leq C$. To prove this, let us identify the action of 
$\sigma \in \SSS_{n}$ with its action on $\{u_{1}, \ldots ,u_{n}\}$ and on 
$\{v_{1}, \ldots ,v_{n}\}$, cf. (\ref{4sn}). Then for any 
$x = \diag(x_{1},x_{2}, \ldots ,x_{n}) \in D$ we have 
$\sigma^{-1}x\sigma x^{-1} = \diag(b_{1}, \ldots ,b_{n})$, where 
$b_{i} = x_{\sigma(i)}x_{i}^{-1} \in B$. Now choose $\sigma = (j,j+1)$ for $1 \leq j \leq n-1$.
Then we get $b_{i} = I$ for $i \neq j,j+1$, and
$b_{j}b_{j+1} = x_{j+1}x_{j}^{-1} \cdot x_{j}x_{j+1}^{-1} = I$. It follows 
that $\sigma^{-1}x\sigma x^{-1} \in C$, and so $\sigma^{-1} (xC) \sigma = xC$ in $D/C$ for 
all $\sigma = (j,j+1)$. Consequently, $\sigma^{-1} (xC) \sigma = xC$ in $D/C$ for all
$\sigma \in \SSS_{n}$, as stated. 

We have shown that $C \lhd G = D\SSS_{n}$ and $[D,\SSS_{n}] \leq C$. Hence 
the subgroup $C\SSS_{n}$ is normal in $G$. Recall that $g_{1} = (12)$ has 
$\age = 1$ and $g_{1} \in \SSS_{n}$. By our assumptions, 
$ZG = Z\langle (g_{1})^{G} \rangle \leq K := ZC\SSS_{n}$. It follows that 
$ZG = K$. Let 
$\Phi$, resp. $\Phi_{1}$, denote the representation of $ZG$ on $V$, resp. of 
$G_{1} := Stab_{ZG}(V_{1})$ on $V_{1}$. Then $G_{1} = ZC\SSS_{n-1}$, where $\SSS_{n-1}$ is
acting trivially on $V_{1}$. Hence   
$\Phi_{1}(G_{1}) = \Phi_{1}(ZC) = \CC^{\times}B$ is abelian. But $\dim(V_{1}) = 2$, so 
$\Phi_{1}$ is reducible. Since $\Phi = \Ind^{ZG}_{G_{1}}(\Phi_{1})$, we conclude that
$ZG$ is reducible on $V$, a contradiction.          
\end{proof}

\subsection{Extraspecial case}\label{SS-ext2} 
Here, $G \leq N := N_{\GC}(E)$ for 
some $p$-group $E$ of extraspecial type. By \cite[Lemma 2.4]{GT1},
either $|\chi(g)| = 0$, or $|\chi(g)|^{2} = |C_{E/Z(E)}(g)|$. It follows that 
$\Delta(g) \geq p^{m}(1-1/\sqrt{p})$. Recall that 
$C = 9.111$ in the set-up $\diam$. Hence
$\dim(V) = p^{m} \leq 9.111(1+1/\sqrt{2})$ and so $p^{m} \leq 13$. 
Since we are assuming
$\dim(V) \geq 4$, we need to consider the following cases.

\smallskip
$\bullet$ $\dim(V) = p \geq 11$. Then $\Delta(g) \geq 11-\sqrt{11} > 7.68$, and so
$\ages(g) > 1$.

\smallskip
$\bullet$ $\dim(V) = p^{m} = 9$. Here, $\Delta(g) \geq 9-3\sqrt{3}$;
moreover, if $|\chi(g)| \leq 3$ then $\Delta(g) \geq 6$. 
Thus we may assume that $|\chi(g)| = 3\sqrt{3}$. Next, $E = 3_{+}^{1+4}$, and 
the character table of 
$N = Z(\GC)E:Sp_{4}(3)$ has been constructed explicitly by T. Breuer.
Now one can verify directly that $N/Z(N)$ contains two classes of 
elements with $|\chi(g)| = 3\sqrt{3}$; any such an element
acts on $E/Z(E) = \FF_{3}^{4}$ as a symplectic transvection. One of these classes
has $\ages = 1$; the other class and all remaining 
non-central elements in $G$ have $\ages > 1$.

\smallskip
$\bullet$ $\dim(V) = p^{m} = 8$. Here, $\Delta(g) \geq 8-4\sqrt{2}$
and $E = C_{4} * 2_{+}^{1+6}$. The character table of 
$N = Z(\GC)E \cdot Sp_{6}(2)$ has been constructed explicitly by Breuer. In 
particular, $\Irr(N)$ contains two, complex-conjugate, characters of degree 
$8$. Hence it suffices to consider one of these two characters and the 
classes of $g$ with $|\chi(g)| \geq 4$. Now 
one can verify directly that $N/Z(N)$ contains three conjugacy classes 
of elements $g$ with $\ages(g) = 1$; their spectra are listed in items (b1) 
and (c6) of Remark \ref{types}. In all other cases, $\ages(g) > 1$ 
by Lemmas \ref{small1} and \ref{small2}.

\smallskip
$\bullet$ $\dim(V) = p^{m} = 7$. Here, $\Delta(g) \geq 7-\sqrt{7}$;
moreover, if $|\chi(g)| \leq 1$ then $\Delta(g) \geq 6$. 
Thus we may assume that $|\chi(g)| = \sqrt{7}$. Next, $E = 7_{+}^{1+2}$, and 
the character table of 
$N = Z(\GC)E:Sp_{2}(7)$ has been constructed explicitly by Breuer. In 
particular, $\Irr(N)$ contains seven characters of degree $7$, with exactly 
six being faithful on $E$, each of which is uniquely determined by its central
character. Hence it suffices to consider one of these six characters. Now 
one can verify directly that in the cases where $|\chi(g)| = \sqrt{7}$, 
the smallest arc of $\SA$ that contains all eigenvalues of $g$ has length
$\delta \geq \pi$, and so $\ages(g) > 1$ by Corollary \ref{bound2}.

\smallskip
$\bullet$ $\dim(V) = p^{m} = 5$. Here, $\Delta(g) \geq 5-\sqrt{5}$
and $E = 5_{+}^{1+2}$. The character table of 
$N = Z(\GC)E:Sp_{2}(5)$ has been constructed explicitly by Breuer. In 
particular, $\Irr(N)$ contains five characters of degree $5$, with exactly 
four being faithful on $E$, each of which is uniquely determined by its central
character. Hence it suffices to consider one of these four characters. Now 
one can verify directly that $N/Z(N)$ contains three conjugacy classes 
of elements $g$ with $\ages(g) = 1$; their spectra are listed in items (b1), 
(c4), and (c5) of Remark \ref{types}. In all other cases, $\ages(g) > 1$ 
by Corollary \ref{bound2} (with $\delta \geq 6\pi/5$).

\smallskip
$\bullet$ $\dim(V) = p^{m} = 4$. Here, $\Delta(g) \geq 4-2\sqrt{2}$
and $E = C_{4} * 2_{+}^{1+4}$. The character table of 
$N = Z(\GC)E \cdot Sp_{4}(2)$ has been constructed explicitly by Breuer. In 
particular, $\Irr(N)$ contains two pairs $(\al,\overline{\al})$ and 
$(\beta,\overline{\beta})$ of complex-conjugate characters of 
degree $4$; furthermore, $\beta$ can be obtained from $\al$ by tensoring 
with the sign character of $Sp_{4}(2) \simeq \SSS_{6}$. Hence we may assume
that $\chi = \al$. Now one can verify directly that $N/Z(N)$ contains three 
conjugacy classes of elements $g$ with $\ages(g) = 1/2$ and spectra as listed in
items (a1), (a4), and (c1) of Remark \ref{types}. Fixing an isomorphism 
between $Sp_{4}(2)$ and $\SSS_{6}$, we may assume that these three
classes project onto the classes of $(12)$, resp. $(123)$, $(12)(34)(56)$, in 
$\SSS_{6}$. $N/Z(N)$ also contains several classes of elements $g$ 
with $\ages(g) = 1$ and spectra as listed in items (b1), (b2), (b4), (c2),  
and (c3) of Remark \ref{types}. In all other cases, $\ages(g) > 1$. Now we
show that $N$ contains a subgroup $M$ leading to a basic non-RT pair not
of reflection type.

\begin{lemma}\label{newRT2}
{\sl There is a subgroup 
$M = C_{3} \times (C_{4} * 2^{1+4}_{+}) \cdot \AAA_{6} < GL(\CC^{4})$ which gives
rise to a basic non-RT pair not of reflection type.}
\end{lemma}

\begin{proof}
Since $M \rhd E$, $M$ acts irreducibly on $V = \CC^{4}$. 
Notice that $Z(\GC)M = Z(\GC)[N,N]$ has index $2$ in $Z(\GC)N$.
By the above analysis, all non-central elements $g \in M$ with 
$\ages(g) < 1$ in $M$ are $[N,N]$-conjugate to an element $g$ with spectrum
$(e^{2\pi i/3},e^{2\pi i/3},1,1)$ which corresponds to the class of $(123)$ in
$\AAA_{6}$.  In fact one can choose such an element $g$ in   
$C_{3} \times 2\AAA_{6} < M$ with $\age(g) < 1$. We have shown that 
$g^{M} = \{ h \in M \setminus Z(M) \mid \age(h) < 1\}$ and that $M$ contains 
no complex reflection. It remains to show that $\langle g^{M} \rangle = M$.

Denote $C := C_{3} \times C_{4} < Z(\GC)$, $E := C_{4} * 2^{1+4}_{+}$, 
$O = EC$, $M_{1} := E \cdot \AAA_{6} < M = C_{3} \times M_{1}$, and
$K := \langle g^{M} \rangle$. Since $(123)^{\AAA_{6}}$ generates 
$\AAA_{6}$, we must have that $KO = M$. Next, since 
$O/C$ is the unique minimal normal subgroup of $M/C$, we see that $KC \geq O$ 
and so $KC = KO = M$.
Observe that $[M,M] = [M_{1},M_{1}] = M_{1}$. (Indeed, it is easy to check 
that $[M_{1},M_{1}]$ contains $2^{1+4}_{+} \cdot \AAA_{6}$ and so it has  
index at most $2$ in $M_{1}$. But $M_{1}$ is a normal subgroup of
index $2$ in $E \cdot \SSS_{6}$, and one can check that $E \cdot \SSS_{6}$ 
has only two linear characters. It follows that $[M_{1},M_{1}] = M_{1}$.)
Now we have $K \geq [K,K] = [KC,KC] = [M,M] = M_{1}$. Also, $M_{1}$ is a 
perfect subgroup of $GL(V)$, whence $M_{1} < SL(V)$. But 
$\det(g) = e^{4\pi i/3} \neq 1$, so $M \geq K > M_{1}$. Since 
$M/M_{1} \cong C_{3}$, we conclude that $K = M$.    
\end{proof}

We will need the following complement to Theorem \ref{main-b}:

\begin{theor}\label{min-age}
{\sl Let $G < GL(V)$ be a finite, irreducible, primitive, tensor indecomposable
subgroup and let $g \in G \setminus Z(G)$. If the $G$-module $V$ is tensor
induced, assume in addition that $g$ acts nontrivially on the set of 
tensor factors of $V$. Then the following statements holds.

{\rm (i)} If $\dim(V) = 2$, then $\ages(g) \geq 1/5$.

{\rm (ii)} If $\dim(V) = 3$ or $4$, then $\ages(g) \geq 1/3$.

{\rm (iii)} If $\dim(V) > 4$, then $\ages(g) \geq 1/2$.}
\end{theor}

\begin{proof}
1) First we consider the case $V$ is tensor induced; in particular,
$d := \dim(V) = a^{m}$ for some integers $a,m \geq 2$, and 
$\Delta(g) \geq d(1-1/a)$ by Lemma \ref{t-ind1}. Now if $d = 4$, then
$\Delta(g) \geq 2$ and so $\ages(g) \geq 4/(2.9\pi) > 0.43$ by 
Proposition \ref{arc}(iii). If $d \geq 5$, then $\Delta(g) \geq 4$ and so 
$\ages(g) \geq 8/(2.9\pi) > 0.86$ again by Proposition \ref{arc}(iii)
(in fact, $\ages(g) > 1$ unless $d=8$). From now on we may assume that
$V$ is not tensor induced.

2) Consider the case $d = 2$ and assume that $\ages(g) < 1/5$. 
Then we may write $\Spec(g) = \{1,e^{i\al}\}$ with $0 < \al < 2\pi/5$. 
It follows that $|\Tr(g)| = 2\cos(\al/2) > 2\cos(\pi/5) = (1+\sqrt{5})/2$ and so 
$\Delta(g) < (3-\sqrt{5})/2$, contradicting Theorem \ref{main-b}(i).

Next assume that $d = 3$ and $\ages(g) < 1/3$. 
Then we may write $\Spec(g) = \{1,e^{i\al},e^{i\beta}\}$ with 
$0 \leq \al \leq \beta \leq \al+\beta < 2\pi/3$; in particular,
$$|\Tr(g)|^{2} = 3 +2\cos(\beta) + 4\cos(\beta/2)\cos(\beta/2-\al).$$
Now $\cos(\beta) > -1/2$ and $\cos(\beta/2) > 1/2$ as $0 \leq \beta < 2\pi/3$.
Also, $\cos(\beta/2-\al) > 1/2$, since 
$-\pi/3 < -\al/2 \leq \beta/2-\al \leq \beta/2 < \pi/3$. It follows that 
$|\Tr(g)|^{2} > 3-1+1 = 3$ and so 
$\Delta(g) < 3-\sqrt{3}$, contradicting Theorem \ref{main-b}(ii). 
 
3) Now we may assume that $d \geq 4$. If we are in the extraspecial case, then
$\ages(g) \geq 1/2$ by the results of \S\S\ref{SS-ext2}. Otherwise, 
by \cite[Proposition 2.8]{GT3} we may apply Theorem \ref{main-d}. 
\end{proof}
 
Note that the lower bounds given in Theorem \ref{min-age} are best possible,
cf. the examples of $SL_{2}(5) < GL_{2}(\CC)$, 
$3^{1+2}_{+}:SL_{2}(3) < GL_{3}(\CC)$, and Table I for examples in dimensions
$\geq 4$.

\subsection{Tensor decomposable case}
  
\begin{lemma}\label{tensor1}
{\sl In the set-up $\diam$, assume that the $G$-module $V$ is primitive
and tensor decomposable. Then $d:= \dim(V) \leq 10$.}
\end{lemma}

\begin{proof}
Write $V = V_{1} \otimes \ldots \otimes V_{m}$, where $V_{i}$ are 
irreducible, primitive, tensor indecomposable $G$-modules of dimension 
$\geq 2$, and $m \geq 2$.

1) Consider the case where $\dim(V_{i}) \geq 3$, say for $i = 1$, and set
$W := V_{2} \otimes \ldots \otimes V_{m}$. Then we can find $g \in \XC$ 
such that $g|_{V_{1}}$ is non-scalar. In the case the $G$-module $V_{1}$ is 
tensor induced, among such elements $g$ we can find one that acts 
nontrivially on the set of tensor factors of $V_{1}$ (as otherwise 
the $G$-module $V_{1}$ would be tensor decomposable). By Theorem 
\ref{min-age}, $\ages(g|_{V_{1}}) \geq 1/3$. By Lemma \ref{trivial}(iv),
$1 \geq \ages(g) \geq \dim(W)\cdot \ages(g|_{V_{1}})$. It follows that 
$\dim(W) \leq 3$ and so $m = 2$. Again, we can find $h \in \XC$ such that
$h|_{V_{2}}$ is non-scalar. Notice that the $G$-module $V_{2}$ is not 
tensor induced. 

Assume $\dim(V_{2}) = 3$. Then $\ages(h|_{V_{2}}) \geq 1/3$ by Theorem 
\ref{min-age}(ii). Now by Lemma \ref{trivial}(iv),
$1 \geq \ages(h) \geq \dim(V_{1})\cdot \ages(h|_{V_{2}})$. It follows that 
$\dim(V_{1}) \leq 3$ and so $d \leq 9$. 

Assume $\dim(V_{2}) = 2$. Then $\ages(h|_{V_{2}}) \geq 1/5$ by Theorem 
\ref{min-age}(i). Now by Lemma \ref{trivial}(iv),
$1 \geq \ages(h) \geq \dim(V_{1})\cdot \ages(h|_{V_{2}})$. It follows that 
$\dim(V_{1}) \leq 5$ and so $d \leq 10$.

2) Now assume that $\dim(V_{i}) = 2$ for all $i$; in particular, 
the $G$-module $V_{i}$ is not tensor induced. Then we can find $g \in \XC$ 
such that $g|_{V_{1}}$ is non-scalar. By Theorem \ref{min-age}(i), 
$\ages(g|_{V_{1}}) \geq 1/5$. It now follows from Lemma \ref{trivial}(iv)
that $1 \geq \ages(g) \geq (d/\dim(V_{1})) \cdot \ages(g|_{V_{1}})$,
$d/\dim(V_{1}) \leq 5$, and so $d \leq 10$. In fact, $d = 4$ or $8$ in
this case.     
\end{proof}

The example of $(C_{5} \times SL_{2}(5)) * (C_{2} \times SU_{4}(2))$ acting on
$\CC^{2} \otimes \CC^{5}$ shows that the bound $10$ in Lemma \ref{tensor1} is 
best possible.
 
\begin{lemma}\label{tensor2}
{\sl In the set-up $\diam$, assume that the $G$-module $V$ is primitive
and tensor decomposable and that $G$ contains a non-central element
$g$ with $\ages(g) < 1$. Then $\dim(V) \leq 8$.}
\end{lemma}

\begin{proof}
Assume the contrary: $d:=\dim(V) \geq 9$. Notice that $d \leq 10$ by Lemma 
\ref{tensor1}. It follows that $d = 9$ or $10$, and
$V = A \otimes B$, where $A$ and $B$ are 
irreducible, primitive, tensor indecomposable, not tensor induced, 
$G$-modules of dimension $> 1$. Since $g \notin Z(G)$, we may assume 
that $g|_{A}$ is not scalar.

Assume $\dim(A) = 3$ (and so $\dim(B) = 3$). Then 
$\ages(g|_{A}) \geq 1/3$ by Theorem 
\ref{min-age}(ii), and so $\ages(g) \geq \dim(B)\cdot \ages(g|_{A}) \geq 1$ 
by Lemma \ref{trivial}(iv), a contradiction. 
Thus $\dim(A) = 2$ or $5$. Assume $\dim(A) = 2$ (and so $\dim(B) = 5$). Then 
$\ages(g|_{A}) \geq 1/5$ by Theorem 
\ref{min-age}(i). Again by Lemma \ref{trivial}(iv),
$\ages(g) \geq \dim(B)\cdot \ages(g|_{A}) \geq 1$, a contradiction.
Finally, let $\dim(A) = 5$ (and so $\dim(B) = 2$). Then 
$\ages(g|_{A}) \geq 1/2$ by Theorem 
\ref{min-age}(iii), and so $\ages(g) \geq \dim(B)\cdot \ages(g|_{A}) \geq 1$, 
again a contradiction.   
\end{proof}

The example of $(C_{5} \times SL_{2}(5)) * (C_{3} \times Sp_{4}(3))$ acting on
$\CC^{2} \otimes \CC^{4}$ shows that the bound $8$ in Lemma \ref{tensor2} is 
best possible.

\begin{lemma}\label{tensor3}
{\sl Let $g = A \otimes B$, where $A \in G < GL_{m}(\CC)$, 
$B \in H < GL_{n}(\CC)$, with $m \geq 3$ and $n \geq 2$. Assume that $G$ and $H$ are finite
primitive irreducible subgroups, and that $A,B$ are non-scalar. Then $\ages(g) > 1$.}
\end{lemma}

\begin{proof}
1) Assume the contrary: $\ages(g) \leq 1$. By a well-known result of 
Blichfeldt, cf. \cite{D}, the smallest arc that contains all eigenvalues of 
any non-central element in any finite, primitive, irreducible linear group has 
length $\geq \pi/3$. Thus $\al,\beta \geq 1/6$, where $2\pi\al$, resp. $2\pi\beta$, is the
length of such smallest arc for $A$, resp. for $B$. In particular, 
$\ages(A),\ages(B) \geq 1/6$. On the other hand, $\ages(A) \leq 1/n \leq 1/2$ and 
$\ages(B) \leq 1/m \leq 1/3$ by Lemma \ref{trivial}(iv), whence $\al \leq 1/2$ and
$\beta \leq 1/3$. By Lemma \ref{trivial}(ii), we can multiply $g$ by
a suitable scalar and assume that $\age(g) \leq 1$. 
Multiplying $B$ by a suitable $\mu \in \SA$
and $A$ by $\mu^{-1}$, we may assume that 
$\Spec(B) \ni 1,e^{2\pi i\beta}$, and all other eigenvalues of $B$ belong to the arc
$[1,e^{2\pi i\beta}]$ of $\SA$. Write
$A = \diag(e^{2\pi i\al_{1}}, \ldots ,e^{2\pi i\al_{m}})$ with $0 \leq \al_{j} < 1$.

\smallskip
2) Here we consider the case $n \geq 3$. Then
$B$ has a third eigenvalue $e^{2\pi i \delta}$ with $0 \leq \delta \leq \beta$. 
For $0 < \gam \leq 1$, observe that
$$\age(e^{-2\pi i\gam}B) \geq \left\{ \begin{array}{ll}
  1+\beta+\delta-3\gam \geq 1+\beta-2\delta, & 0 < \gam \leq \delta,\\
  2+\beta+\delta-3\gam \geq 2(1-\beta) > 1, & \delta < \gam \leq \beta,\\
  \beta+\delta + 3(1-\gam) \geq \beta+\delta, & \beta < \gam \leq 1. \end{array} \right.$$ 
We will apply this observation to $\gam = \gam_{j} := 1-\al_{j}$. Since 
$1 \geq \age(A \otimes B) = \sum^{m}_{j=1}\age(e^{-2\pi i\gam_{j}}B)$, we see that 
all $\gam_{j}$ must belong to $(\beta,1]$, and 
$$1 \geq \age(A \otimes B) \geq \sum^{m}_{j=1}(\beta + \delta + 3(1-\gam_{j})) \geq 
  m\beta+3\sum^{m}_{j=1}\al_{j} = m\beta +3\age(A).$$  
Recall $m \geq 3$, $\beta \geq 1/6$, and $\age(A) \geq 1/6$. It follows that 
$m = 3$ and $\age(A) = 1/6$. The last equality however contradicts Theorem \ref{min-age}(ii)
applied to the element $A$ of $G$.

\smallskip
3) Now we let $n = 2$. For $0 < \gam \leq 1$, we have
$$\age(e^{-2\pi i\gam}B) = \left\{ \begin{array}{ll}
  1+\beta-2\gam \geq 1-\beta, & 0 < \gam \leq \beta,\\
  2+\beta-2\gam \geq \beta, & \beta < \gam \leq 1. \end{array} \right.$$ 
We will again apply this observation to $\gam = \gam_{j} := 1-\al_{j}$ to estimate 
$\age(A \otimes B) = \sum^{m}_{j=1}\age(e^{-2\pi i\gam_{j}}B)$. If at least one $\gam_{j}$ 
belongs to $(0,\beta]$, then
$\age(g) \geq 1-\beta + (m-1)\beta \geq 1+(m-2)\beta > 1$, a contradiction.
Hence, all $\gam_{j}$ belong to $(\beta,1]$, and so  
$$1 \geq \age(A \otimes B) = \sum^{m}_{j=1}(\beta + 2(1-\gam_{j})) = 
  m\beta+2\sum^{m}_{j=1}\al_{j} = m\beta +2\age(A).$$  
Recall that $m \geq 3$, $\age(A) \geq 1/6$, and $\beta = \age(B) \geq 1/5$ by 
Theorem \ref{min-age}(i) applied to the element $B$ of $H$. It follows that 
$m \leq (1-2/6)/(1/5)$ and so $m = 3$. But in this case, Theorem \ref{min-age}(ii)
applied to the element $A$ of $G$ yields that $\age(A) \geq 1/3$ and 
so $\age(A \otimes B) \geq 3/5 + 2/3 > 1$.   
\end{proof}

\begin{corol}\label{tensor4}
{\sl Let $G < GL(V)$ be a finite primitive irreducible subgroup. Assume that $\dim(V) \geq 5$ and 
that the $G$-module $V$ is tensor decomposable. Then, for any $g \in G$ with
$\ages(g) \leq 1$, $ZG \neq Z\langle g^{G} \rangle$, where $Z := Z(GL(V))$.}
\end{corol}      

\begin{proof}
Write $V = A \otimes B$ for some $G$-modules $A$, $B$ of dimension $> 1$. Then $G$ is 
irreducible and primitive on both $A$ and $B$. Now consider any $g \in G$ with
$\ages(g) \leq 1$. By Lemma \ref{tensor3}, $g$ must act scalarly on $A$ or on $B$,
say on $A$. In this case, $H := \langle g^{G} \rangle$ also acts scalarly on $A$ and so
$ZG \neq ZH$ by irreducibility of $G$ on $A$.  
\end{proof}

\subsection{Tensor induced case}

\begin{propo}\label{t-ind2}
{\sl In the set-up $\diam$, assume that the $G$-module $V$ is primitive, tensor
indecomposable, but tensor induced. Then 
$\dim(V) = 4$ or $8$. Moreover, if $\dim(V) = 8$, then $G$ cannot be generated by its elements 
$h$ with $\ages(h) < 1$ (modulo scalars).}
\end{propo}

\begin{proof}
1) By the assumptions, there is a tensor decomposition $V = V_{1}^{\otimes m}$ with
$a := \dim(V_{1})> 1$ and $m > 1$ such that $G < GL(V_{1})^{\otimes m}:\SSS_{m}$, and 
there is some $g \in \XC$ such that $0 < \ages(g) \leq 1$ and $g$ acts nontrivially on
the $m$ tensor factors of $V$. By Lemma \ref{t-ind1}, 
$4.556 > \Delta(g) = \dim(V) -|\Tr(g)| \geq a^{m-1}(a-1)$. It follows that $a = 2$ and 
$m = 2$ or $3$. 

From now on we assume that $(a,m) = (2,3)$, i.e.
$V = V_{1} \otimes V_{2} \otimes V_{3}$ and $\dim(V_{i}) = 2$. Then $g$ must project onto a $2$-cycle of
$\SSS_{3}$, as otherwise by Lemma \ref{t-ind1}, $\Delta(g) \geq 8-2 = 6$ and so $\ages(g) > 1$. 
Without loss we may assume that $g = A \otimes B$, with $A < GL(V_{1})^{\otimes 2}:\SSS_{2}$
permuting the two tensor factors $V_{1}$ and $V_{2}$, and $B \in GL(V_{3})$, and that $g$ has finite 
order: $g^{N} = I_{8}$ for some integer $N > 1$. (Here we let $I_{n}$ denote the identity 
$n \times n$-matrix.)  It follows that 
$A^{N} \otimes B^{N} = I_{8} = I_{4} \otimes I_{2}$. By (the first sentence of) the proof of 
Lemma \ref{tensor0}, we can multiply $A$ by a suitable $\lam \in \CC^{\times}$ (and $B$ 
by $\lam^{-1}$) such that $A^{N} = I_{4}$ and $B^{N} = I_{2}$.   

\smallskip
2) Here we show that $\ages(A) \geq 1/2$ and $\ages(g) \geq 1$. 
By our assumptions, there are some bases 
$(e_{1},e_{2})$ of $V_{1}$ and $(f_{1},f_{2})$ of $V_{2}$, and matrices $X, Y \in GL_{2}(\CC)$ such
that, in the basis $(e_{1} \otimes f_{1},e_{2} \otimes f_{1},e_{1} \otimes f_{2},e_{2} \otimes f_{2})$
of $V_{1} \otimes V_{2}$, $A = {\bf j}(X \otimes Y)$, where 
${\bf j}~:~e_{i} \otimes f_{j} \mapsto e_{j} \otimes f_{i}$. Now direct computation shows that 
$$\det(A-tI) = t^{4} - \Tr(XY)\cdot t^{3} + \Tr(XY)\cdot \det(XY) \cdot t - \det(XY)^{2}.$$
In particular, writing $\Spec(XY) = \{x,xu^{2}\}$ for some $x,u \in \CC^{\times}$, we get
$\Spec(A) = \{x,xu^{2},xu,-xu\}$. Hence $\ages(A) \geq 1/2$ by Lemma \ref{small1}.           
Now by Lemma \ref{trivial}(iv) we have $\ages(g) \geq 2\ages(A) \geq 1$. This lower bound is
best possible as $\age({\bf j} \otimes I_{2}) = 1$ (in fact, ${\bf j} \otimes I_{2}$ acts as 
a bireflection on $V$). 

\smallskip
3) We have shown that any element of $G$ that acts nontrivially on the set of
$3$ tensor factors of $V$ has $\ages \geq 1$. In particular, if $\ages(h) < 1$ for some 
$h \in G$, then $h$ belongs to the base subgroup 
$G \cap GL(V_{1}) \otimes GL(V_{2}) \otimes GL(V_{3})$. Thus 
$Z(GL(V)) \cdot \langle h \in G \mid \ages(h) < 1 \rangle < Z(GL(V))G$. 
\end{proof}

\subsection{Proof of Theorem \ref{main4}}
Let $d := \dim(V) > 4$ and let $G < GL(V)$ satisfy the hypotheses of the Theorem. 
If $G$ is imprimitive, then the statement follows from Lemma \ref{impr3} and Theorem \ref{impr-a}(ii).
So we may assume that the $G$-module $V$ is primitive. Now by Corollary \ref{tensor4}, 
$V$ is tensor indecomposable, and so it cannot be tensor induced by Proposition \ref{t-ind2}.
The extraspecial case cannot occur either, by the results of \S\S\ref{SS-ext2}. Thus $G$ is 
almost quasi-simple by \cite[Proposition 2.8]{GT3}, i.e. $S \lhd G/Z(G) \leq \Aut(S)$ for 
some simple non-abelian group $S$. 

We can now apply Theorem \ref{main-d};
in particular, either $(d,S) = (n-1,\AAA_{n})$ or $d \leq 8$. In the former case, up to
scalars, $G = \SSS_{n}$ (in its action on the deleted natural permutation module) and so a c.r.g.
Consider the latter case. If $d = 8$, then $S = \Omega^{+}_{8}(2)$, $G/Z(G) = S \cdot 2$, and
up to scalars, $G$ is the Weyl group of type $E_{8}$. If $d = 7$, then $G/Z(G) = S = Sp_{6}(2)$, 
and up to scalars, $G$ is the Weyl group of type $E_{7}$. Assume $d = 6$. If $S = SU_{4}(2)$, then 
$G/Z(G) = S \cdot 2$, and up to scalars, $G$ is the Weyl group of type $E_{6}$. If 
$S = PSU_{4}(3)$, then $G/Z(G) = S \cdot 2_{2}$ in the notation of \cite{Atlas} (the other two 
involutions in $\Out(S)$ does not preserve the $6$-dimensional representation in question of 
$G^{(\infty)}$), and so $G$ is a c.r.g. modulo scalars. In all these cases, there is only one 
conjugacy class in $G/Z(G)$ that contain non-central elements $g$ with $\ages(g) < 1$, and these
elements $g$ are scalar multiples of reflections. Finally, if 
$d = 5$, then $G = SU_{4}(2) \cdot Z(G)$ and so it is also a c.r.g. modulo scalars.   
\hfill $\Box$ 

\subsection{Proof of Theorem \ref{main2}}
Let $(G,V)$ satisfy the hypotheses of the Theorem. 
If the $G$-module $V$ is imprimitive, then the statement follows from Lemmas \ref{impr3} and
\ref{impr4}. (Notice that in case (iii) the 
transpositions in the subgroup $\SSS_{n}$ act on $V$ as bireflections.) 
So we may assume that $V$ is primitive, of dimension $\geq 11$. 
Hence the extraspecial case cannot occur by the analysis in 
\S\S\ref{SS-ext2}. Next, the $G$-module $V$ cannot be tensor decomposable or tensor 
induced by Lemma \ref{tensor1} and Proposition \ref{t-ind2}. Thus we are in the 
almost quasi-simple case and can apply Theorem \ref{main-d}. Since $\dim(V) \geq 11$,
we arrive at the conclusion (i).
\hfill $\Box$ 

\subsection{Proof of Theorem \ref{main3}}
Let $(G,V)$ satisfy the hypotheses of the Theorem. First we consider the case where the $G$-module
$V$ is imprimitive and apply Lemma \ref{impr3}. In the case (ii) of Lemma \ref{impr3},
we arrive at conclusion (ii) of the Theorem (notice that all elements of $G$ with $\ages < 1$
are contained in $D$ and so cannot generate $G$ modulo scalars). Suppose we are in the case (i)
of Lemma \ref{impr3}. Then $G$ satisfies the hypotheses of Theorem  
\ref{impr-a}(i), and so we are done. So we may assume that $V$ is primitive and 
$\dim(V) \geq 9$. Hence the extraspecial case cannot occur by the analysis in 
\S\S\ref{SS-ext2}. Next, the $G$-module $V$ cannot be tensor decomposable or tensor 
induced by Lemma \ref{tensor2} and Proposition \ref{t-ind2}. Thus we are in the 
almost quasi-simple case and can apply Theorem \ref{main-d}. Since $\dim(V) \geq 9$,
we must now have that $G = \SSS_{d+1}$ modulo scalars, as stated in (i). 
\hfill $\Box$    

\begin{remar}\label{r14}
{\em (a) The group $(C_{5} \times SL_{2}(5))*(C_{3} \times Sp_{4}(3)) < GL_{8}(\CC)$ is generated 
by its elements of $\age <1$, but yet does not contain any complex reflection by Lemma 
\ref{tensor1}. Thus the bound $d \geq 9$ in Theorem \ref{main3} is best possible.

(b) The case (ii) of Theorem \ref{main3} indeed occurs, as shown in the following example. 
Consider the subgroup $A = C_{7} \times SL_{2}(5)$ of $GL_{2}(\CC)$ and let $G$ be 
the wreath product $A \wr \SSS_{n}$ acting on $V = \CC^{2n}$ for any $n \geq 2$. It is easy to check
that $G$ is generated by its (non-central) elements with $\age \leq 1$, and $G$ contains non-central 
elements with $\age = 2/7$. However, $G$ does not contain any complex reflection. For, suppose
$g \in G$ is conjugate to $\diag(\al, \al, \ldots, \al,\beta)$ for some $\al \neq \beta \in \SA$. 
By Lemma \ref{impr3}, $g = \diag(a_{1}, \ldots ,a_{n}) \in A^{n}$ since $0 < \ages(g) < 1$. It follows 
that $\al^{14} = 1$ and that some element $x \in SL_{2}(5)$ has eigenvalues 
$\mu\al$, $\mu\beta$ for some $\mu \in \SA$ with $\mu^{7} = 1$. Now $(\mu\al)^{14} = 1$, and so
$x$ must be scalar, whence $\al = \beta$, a contradiction.}
\end{remar}    

\subsection{Proof of Corollary \ref{main1g}}
By the assumption, $2\pi\cdot ||g|| \leq L$ for some $1 \neq g \in G$. Let $N$ be the (normal) subgroup
generated by all elements in $G$ with this property. By Corollary \ref{basic2m}, $N$ is generated 
by a set of elements $g$ with $\DB(g)^{2} \leq C$, where $C = \max\{4,L^{2}\}$. Also by 
\cite[Lemma 2.5]{GT3}, either $N \leq Z(G)$, or $N$ is irreducible on $V$. In the former case,
the cyclic group $Z(G)$ contains an element $g = e^{2\pi i j/s} \cdot 1_{V}$ with $s := |Z(G)|$, 
$1 \leq j \leq s-1$, and $L \geq 2\pi\cdot ||g|| \geq 2\pi\sqrt{\dim(V)}/s$, whence 
$\dim(V) \leq (Ls/2\pi)^{2}$. In the latter case, we may apply Theorem \ref{main1} to $N$.
Assume that the conclusion (i) of Theorem \ref{main1} does not hold; in particular,
$d := \dim(V) > 40C \geq 160$. In the case the conclusion (ii) of Theorem \ref{main1} holds for
$N$, we have that $G \rhd M := N^{(\infty)} \cong \AAA_{d+1}$ and $M$ acts irreducibly on $V$. 
By Schur's Lemma, $C_{G}(M) = Z(G)$, and $G/C_{G}(M) \leq \Aut(M) = M \cdot 2$, whence the conclusion
(ii) of Theorem \ref{main1} holds for $G$. 

Finally, assume that the conclusion (iii) of Theorem
\ref{main1} holds for $N$, and let $D$ be the normal subgroup of $N$ that fixes each $V_{i}$ 
(setwise); in particular, $\AAA_{m} \leq N/D \leq \SSS_{m}$. In this case, 
$40C < d = m\dim(V_{1}) \leq mC/4$, whence $m > 160$. If, in addition, $m \leq C/4 + 1$, then 
$C \geq 640$ and so $d \leq mC/4 \leq (C/4+1)\cdot C/4 < 4C^{2}/63 \leq \fl(C)$. Hence we may 
assume $m > \max\{160,C/4+1\}$. 
Set $e := \dim(V_{1})$ and let $\CL$ be the collection of 
all finite simple groups $S$ with the property that either $S$ is cyclic, or 
$S \cong X/Y$ for some finite subgroups $Y \lhd X < PGL_{e}(\CC)$. Observe that every composition 
factor of $D$ belongs to $\CL$. (Indeed, consider the chain 
$D = D_{0} \rhd D_{1} \rhd D_{2} \rhd \ldots \rhd D_{m} = 1$, where $D_{i}$ is the kernel of the 
action of $D_{i-1}$ on $V_{i}$ for $1 \leq i \leq m$. Now let $S$ be any non-abelian composition 
factor of $D_{i}/D_{i-1}$. Thus $S \cong A/B$ for some $B \lhd A < GL(V_{i})$ since 
$D_{i}/D_{i-1} \hookrightarrow GL(V_{i})$. Since $S$ is non-abelian, $S$ is also a composition 
factor of $AZ/BZ$ for $Z := Z(GL(V_{i}))$. It follows that $S$ is a composition factor of 
$AZ/Z < PGL(V_{i}) \cong PGL_{e}(\CC)$, i.e. $S \in \CL$.) Let $R$ be the largest normal 
subgroup of $N$ with every composition factor belonging to $\CL$, cf. Lemma \ref{maxc}. Then 
$D \lhd R \lhd N$. Assume that $R > D$. Then $\AAA_{m}$ is a composition factor of $R$ and so
$\AAA_{m} \in \CL$. The latter inclusion means that $\AAA_{m} \cong X/Y$ for some finite subgroups 
$Y \lhd X < PGL_{e}(\CC)$, with $e \leq C/4 < m-1$ and $m > 160$. This however contradicts 
the Feit-Tits Theorem, cf. \cite[Theorem 3]{KlL}. Thus $R = D$. By Lemma \ref{maxc},
$R \lhd G$. We have shown that $D \lhd G$. Since $D$ is reducible on $V$, by \cite[Lemma 2.5]{GT3}
we must have $D \leq Z(G) \cap N = Z(N)$. It follows that $V$ yields an irreducible projective 
representation of degree $me \leq mC/4 < m^{2}$ of $N/D \in \{\AAA_{m},\SSS_{m}\}$. Recall that 
$m > 160$. Using the information on the small degrees of irreducible projective representations of 
$N/D$ as given in \cite{Ra} and \cite{KT}, we see that $me = m(m-3)/2$, 
$M = N^{(\infty)} \cong \AAA_{m}$, and $V|_{M}$ equals the restriction of the Specht module 
$S^{(m-2,2)}$ (labeled by the partition $(m-2,2)$) to $\AAA_{m}$. But the latter restriction is 
primitive, whereas $V$ is imprimitive, a contradiction.    
\hfill $\Box$  

\subsection{Proof of Corollary \ref{main2cr}}
By the assumption and the Reid-Tai criterion \cite{R1}, $\age(g) \leq 1$ for some $1 \neq g \in G$. 
Let $N$ be the (normal) subgroup generated by all elements in $G$ with this property. By 
\cite[Lemma 2.5]{GT3}, either $N \leq Z(G)$, or $N$ is irreducible on $V = \CC^{d}$. In the former 
case, the cyclic group $Z(G)$ contains an element $g = e^{2\pi i j/s} \cdot 1_{V}$ with $s := |Z(G)|$, 
$1 \leq j \leq s-1$, and $1 \geq \age(g) \geq \dim(V)/s$, whence 
$\dim(V) \leq s$. In the latter case, we may apply Theorem \ref{main2} to $N$.
Assume that the conclusion (i) of Theorem \ref{main1} holds for
$N$. Then $G \rhd M := N^{(\infty)} \cong \AAA_{d+1}$ and $M$ acts irreducibly on $V$. 
By Schur's Lemma, $C_{G}(M) = Z(G)$, and $G/C_{G}(M) \leq \Aut(M) = M \cdot 2$, whence the conclusion
(i) of Corollary \ref{main2cr} holds for $G$. 

Next, assume that either the conclusion (ii) or (iii) of Theorem
\ref{main2} holds for $N$. In the case of (ii), define $D$ to be the normal subgroup of $N$ that 
fixes each $V_{i}$ (setwise); in the case of (iii), consider the normal subgroup $D$ defined therein.
In particular, $\AAA_{n} \leq N/D \leq \SSS_{n}$. Also, let $\CL$ be the collection of 
all finite simple groups $S$ with the property that either $S$ is cyclic, or 
$S \cong X/Y$ for some finite subgroups $Y \lhd X < PGL_{2}(\CC)$ (and so $S \cong \AAA_{5}$, as 
easily seen). Notice $n \geq 6$ since $n \geq d/2$. Arguing as in the last part of the proof of 
Corollary \ref{main1g}, we see that $D$ is the largest normal subgroup of $N$ with all composition 
factors belonging to $\CL$, and so $D \lhd G$ by Lemma \ref{maxc}. Since $D$ is reducible on 
$V$, by \cite[Lemma 2.5]{GT3} we must have that $D \leq Z(G) \cap N = Z(N)$. Now in the case of 
the conclusion (iii) of Theorem \ref{main2} for $N$, we have $N = \SSS_{n}Z(N)$, with 
$\SSS_{n}$ acting reducibly on $V$, a contradiction. Thus we are in the case of the conclusion
(ii) of Theorem \ref{main2}, in particular $n \geq 11$, and $V$ yields an irreducible 
projective representation of degree $n$ of $N/D \in \{\AAA_{n},\SSS_{n}\}$, again a contradiction.
\hfill $\Box$
 
\subsection{Small dimension case}

\begin{propo}\label{main-s}
{\sl Let $G$ satisfy the set-up $\diam$, with $4 \leq d:= \dim(V) \leq 10$. 
Then one of the following
statements holds.

{\rm (i)} $G$ preserves a decomposition $V = V_{1} \oplus \ldots \oplus V_{d}$ of $V$ into 
$1$-spaces, and $(\pi(G),d) = (\SSS_{d},d)$, $(\AAA_{d},d)$,
$(ASL_{3}(2),8)$, $(SL_{3}(2),7)$, $(\AAA_{5},6)$, $(D_{10},5)$, if $\pi$ denotes the induced
permutation action of $G$ on $\{V_{1}, \ldots ,V_{d}\}$. 

{\rm (ii)} $2|d$, and $G = D:\SSS_{d/2} < GL_{2}(\CC) \wr \SSS_{d/2}$, a split extension of 
$D < GL_{2}(\CC)^{d/2}$ by $\SSS_{d/2}$.

{\rm (iii)} $G$ preserves a decomposition $V = A \otimes B$, with $\dim(A), \dim(B) > 1$. 

{\rm (iv)} $G$ preserves a tensor structure $V = A^{\otimes m}$, with $\dim(A)= 2$ and 
$m = 2,3$.

{\rm (v)} $4 \leq \dim(V) = p^{a} \leq 9$ for some prime $p$, and $G$ normalizes a $p$-group $E$ of 
extraspecial type, with $|E/Z(E)| = p^{2a}$.

{\rm (vi)} $G$ is almost quasi-simple, and $G$ satisfies the conclusions of Theorem 
\ref{main-d}.}
\end{propo}

\begin{proof}
If the $G$-module $V$ is imprimitive, then the statement follows from Lemma \ref{impr3}. Assume 
$V$ is primitive. Next, (iii), resp. (iv), (v), corresponds to the case when the $G$-module $V$  
is tensor decomposable, resp. tensor induced, or $G$ is in the extraspecial case. Otherwise,
by \cite[Proposition 2.8]{GT3} $G$ satisfies the hypothesis, and so the conclusions, of
Theorem \ref{main-d}.
\end{proof}

\end{document}